\documentclass[british]{amsart}

\pdfoutput=1

\usepackage{babel}
\usepackage[T1]{fontenc}
\usepackage{lmodern}

\usepackage[left=3.5cm, right=3.5cm, top=3cm, bottom=3cm]{geometry}
\usepackage{needspace}
\usepackage{setspace}
\usepackage{aliascnt}
\usepackage{varioref}
\labelformat{equation}{\textup{(#1)}}
\labelformat{enumi}{\textup{(#1)}}

\usepackage{amssymb}
\usepackage{subfigure}
\usepackage{graphicx}
\usepackage{mathtools}

\usepackage[all]{xy}
\UseTips

\usepackage{ifpdf}
\ifpdf
\IfFileExists{xypdf}{\xyoption{pdf}}{}
\usepackage{microtype}
\fi

\usepackage[pdfusetitle=true, colorlinks=true, linkcolor=blue, citecolor=blue, urlcolor=blue, linktoc=page]{hyperref}

\title{Derived localisation of algebras and modules}

\author{C.~Braun}
\address{Department of Mathematics and Statistics\\
	Lancaster University\\
	Lancaster LA1 4YF\\United Kingdom}
\email{c.braun@lancaster.ac.uk}

\author{J.~Chuang}
\address{Department of Mathematics\\
City, University of London\\
Northampton Square\\
London EC1V 0HB\\United Kingdom}
\email{j.chuang@city.ac.uk}

\author{A.~Lazarev}
\address{Department of Mathematics and Statistics\\
Lancaster University\\
Lancaster LA1 4YF\\United Kingdom}
\email{a.lazarev@lancaster.ac.uk}

\thanks{This work was partially supported by EPSRC grants EP/J00877X/1 and EP/J008451/1.}

\hypersetup{pdfauthor={C. Braun, J. Chuang, and A. Lazarev}}

\theoremstyle{plain}
\newtheorem{theorem}{Theorem}[section]
\newcommand{\newautoreftheorem}[2]{
\newaliascnt{#1}{theorem}\newtheorem{#1}[#1]{#2}\aliascntresetthe{#1}%
\expandafter\def\csname #1autorefname\endcsname{#2}}

\newautoreftheorem{proposition}{Proposition}
\newautoreftheorem{lemma}{Lemma}
\newautoreftheorem{corollary}{Corollary}
\newtheorem*{theorem*}{Theorem}

\theoremstyle{definition}
\newautoreftheorem{definition}{Definition}
\newautoreftheorem{remark}{Remark}
\newautoreftheorem{example}{Example}
\newautoreftheorem{convention}{Convention}

\newcommand{\co}{\colon}
\newcommand{\itbbk}{\ensuremath{\mathrlap{I}\kern 0.4ex K}}
\newcommand{\cof}{\rightarrowtail}
\newcommand{\fib}{\twoheadrightarrow}
\newcommand{\acyccof}{\stackrel{\sim}{\rightarrowtail}}
\newcommand{\acycfib}{\stackrel{\sim}{\twoheadrightarrow}}
\newcommand{\we}{\stackrel{\sim}{\rightarrow}}
\newcommand{\ground}{\mathsf{k}}

\newcommand{\id}{\mathrm{id}}
\newcommand{\degree}[1]{\lvert #1 \rvert}
\newcommand{\op}{\mathrm{op}}

\newcommand{\dgalg}{\mathsf{dgAlg}}
\newcommand{\alg}{\mathsf{Alg}}
\newcommand{\undercat}[2]{#1\downarrow #2}
\newcommand{\derundercat}[2]{#1\downarrow^{\mathbb{L}} #2}
\newcommand{\underalg}[1]{\undercat{#1}{\dgalg}}
\newcommand{\derunderalg}[1]{\derundercat{#1}{\dgalg}}
\mathchardef\mhyphen="2D
\newcommand{\modcat}[1]{#1\mhyphen\mathsf{Mod}}
\newcommand{\amod}{\modcat{A}}

\newcommand{\Hmo}{\mathsf{Hmo}}

\newcommand{\SMon}{\mathsf{SMon}}
\newcommand{\SAlg}{\mathsf{SAlg}}

\newcommand{\B}{\mathrm{B}}
\newcommand{\GL}{\mathrm{GL}}
\newcommand{\gl}{\mathfrak{gl}}
\newcommand{\HCC}{\mathrm{HC}}
\newcommand{\CE}{\mathrm{CE}}
\newcommand{\HCE}{\mathrm{HCE}}
\newcommand{\HH}{\mathrm{HH}}

\newcommand{\feyn}{\mathsf{F}}
\newcommand{\ass}{\mathcal{A}ss}
\newcommand{\fass}{\feyn\ass}
\newcommand{\frc}{\vert}

\DeclareMathOperator{\Ho}{Ho}
\DeclareMathOperator*{\colim}{colim}
\DeclareMathOperator*{\holim}{holim}
\DeclareMathOperator*{\hocolim}{hocolim}
\DeclareMathOperator{\Hom}{Hom}
\DeclareMathOperator{\Map}{Map}
\DeclareMathOperator{\RHom}{\mathbb{R}Hom}
\DeclareMathOperator{\REnd}{\mathbb{R}End}
\DeclareMathOperator{\End}{End}
\DeclareMathOperator{\Aut}{Aut}
\DeclareMathOperator{\Diff}{Diff}
\DeclareMathOperator{\Tor}{Tor}

\begin{document}
\def\sectionautorefname{Section}

\begin{abstract}
For any dg algebra $A$, not necessarily commutative, and a subset $S$ in $H(A)$, the homology of $A$, we construct its derived localisation $L_S(A)$ together with a map $A\to L_S(A)$, well-defined in the homotopy category of dg algebras, which possesses a universal property, similar to that of the ordinary localisation, but formulated in homotopy invariant terms. Even if $A$ is an ordinary ring, $L_S(A)$ may have non-trivial homology. Unlike the commutative case, the localisation functor does not commute, in general, with homology but instead there is a spectral sequence relating $H(L_S(A))$ and $L_S(H(A))$; this spectral sequence collapses when, e.g.~$S$ is an Ore set or when $A$ is a free ring.

We prove that $L_S(A)$ could also be regarded as a Bousfield localisation of $A$ viewed as a left or right dg module over itself. Combined with the results of Dwyer--Kan on simplicial localisation, this leads to a simple and conceptual proof of the topological group completion theorem. Further applications include algebraic $K$--theory, cyclic and Hochschild homology, strictification of homotopy unital algebras, idempotent ideals, the stable homology of various mapping class groups and Kontsevich's graph homology.
\end{abstract}

\maketitle
\tableofcontents

\section{Introduction}
Localisation of a commutative ring is among the fundamental tools in commutative algebra and algebraic geometry; it has been well-understood and documented for a long time. Let us recall the basic construction.

Given an element $s$ in a commutative ring $A$, we form its localisation $A[s^{-1}]$ as $A[s^{-1}]:=A\otimes_{\mathbb{Z}[x]}\mathbb{Z}[x,x^{-1}]$ where $A$ is viewed as a $\mathbb{Z}[x]$--algebra via the map $x\mapsto s$. Then $A[s^{-1}]$ is an $A$--algebra having a universal property: any ring map $A\to B$ taking $s$ to an invertible element in $B$, factors uniquely through $A[s^{-1}]$. This construction, and the universal property, easily generalise to an arbitrary multiplicatively closed set of elements $S\subset A$ to produce an $A$--algebra $A[S^{-1}]$. The latter has many good properties: it is flat over $A$, its elements could be represented as fractions with denominators in $S$ and the kernel of the map $A\to A[S^{-1}]$ is easily described in terms of $S$.

This story has an equally straightforward analogue for modules: given an $A$--module $M$, we form its localisation $MS^{-1}$ as $A[S^{-1}]\otimes_AM$. Then there is an $A$--module map $M\to MS^{-1}$ which is easily seen to satisfy an appropriate universal property. Since $A[S^{-1}]$ is $A$--flat, the functor $M\mapsto MS^{-1}$ is exact. Note that (trivially) the localisation of $A$ as an $A$--module coincides with its localisation as a ring. The derived category of $A[S^{-1}]$ is a full subcategory of the derived category of $A$.

Now suppose that the ring $A$ is \emph{not commutative}. It is still possible to form $A[s^{-1}]$ by formally inverting $s$, in other words setting $A[s^{-1}]:=A*_{\mathbb{Z}[x]}\mathbb{Z}[x,x^{-1}]$ where $*_{\mathbb{Z}[x]}$ stands for the coproduct in the category of associative rings under $\mathbb{Z}[x]$. There is still a ring map $A\to A[s^{-1}]$ satisfying an appropriate universal property \emph{in the category of associative rings}. This could be generalised to the case of an arbitrary multiplicatively closed set of elements $S\subset A$. The localisation of a (left) $A$--module $M$ can be defined by the same formula as in the commutative situation and it will also satisfy the same universal property.

However this is as far as the analogy could be extended. The ring $A[S^{-1}]$ will not, in general, be flat over $A$, its elements may not be representable as fractions with denominators in $S$ and the functor $M\mapsto MS^{-1}$ may not be exact. The derived category of $A[S^{-1}]$ may not be a full subcategory of that of $A$.

There is one situation when localisation of a noncommutative ring does have all the good properties of the commutative localisation: when $S$ is central in $A$ or, more generally, is a (left or right) \emph{Ore set}, cf.~for example \cite{Cohn} regarding this notion. On the other hand, one wants to have a good theory of localisation for rings which are not Ore, such as free rings. Cohn's theory of localisations of free ideal rings (firs) \cite{Cohn} suggests that this could be possible.

The main idea of the present paper is to `derive' the localisation functor $A\mapsto A[S^{-1}]$ described above to a functor $A\mapsto L_S(A)$ so that the latter has all the good exactness properties of the commutative localisation. To do that, we embed the category of rings into the category of differential graded (dg) rings (or, more generally, dg algebras over a given commutative ring $\ground$). The category of dg algebras is a closed model category in the sense of Quillen and its category of dg modules is also such. The universal properties described above should then be formulated in the homotopy categories of dg $A$--algebras and (left) dg $A$--modules.

This simple idea indeed works and we can construct the derived algebra localisation of $A$ (which is now taken to be a dg ring) at its homology class $s$ as $L_s(A)=A*^{\mathbb{L}}_{\mathbb {Z}[x]}\mathbb{Z}[x,x^{-1}]$, the \emph{derived} coproduct under $\mathbb{Z}[x]$. In this particular case, we will see that the latter could be computed by replacing just $A$ with a cofibrant dg algebra under $\mathbb{Z}[x]$. This construction can then be extended to a \emph{collection} $S$ of homology classes of $A$. A priori, this construction, although very natural, is rather intractable in general and difficult to compute.

The main technical theorem of the paper is that the derived localisation $L_S(A)$ so defined turns out to be also the Bousfield localisation of $A$ as a (left) dg module over itself; in contrast with the non-derived case, this fact is highly non-trivial to prove. The Bousfield localisation of $A$ as a dg module, being accessible via tools from classical homological algebra, is much more amenable to computation. Therefore, this theorem constitutes the fundamental ingredient one needs in order to work with derived localisation in a meaningful way and is the central technical tool in all our applications. Although this theorem is likely not too surprising for a specialist, it seems that this result, or rather the lack of it, has been the main roadblock to the development of a useful theory of derived noncommutative localisation thus far.

One consequence of our main result is that the homotopy category of dg modules over $L_S(A)$ is equivalent to the homotopy category of dg $A$--modules upon which elements of $S$ act by quasi-isomorphisms. In other words, the close relationship between the derived categories of $A$ and that of its derived localisation is restored. One should mention that the notion of a \emph{module} localisation over a possibly noncommutative ring was already considered in \cite{dwyer:localisationinhomotopytheory} and this work of Dwyer's served as an inspiration for many constructions in the present paper.

It needs to be emphasised that even for an ordinary (i.e.~ungraded) ring $A$ its derived localisation will not, in general, be concentrated in degree zero (although its degree zero homology will coincide with its non-derived localisation). This reinforces the intuition that the derived localisation is a derived functor of the ordinary localisation.

A related phenomenon, which is not seen in the commutative world, is that the homology of $L_S(A)$ is not, generally, the homology of the derived localisation of $H(A)$, but rather, there is a spectral sequence whose $E^2$~term is $H(L_S(H(A)))$ converging to $H(L_S(A))$. The spectral sequence collapses in some special situations, e.g. if $S$ is an Ore set. We extend the classical theory of Ore localisation, obtaining a characterisation of derived localisations at countable Ore sets as the localisations expressible as `derived modules of fractions', a result which appears to be new even in the non-derived setting.

Derived localisation appears in a range of situations, both directly and indirectly. Building on work of Tabuada, To\"en \cite{Toenmorita} has developed in detail the homotopy theory of dg categories and, as one application of this theory, he showed the existence of a derived localisation of dg categories which possesses a certain universal property, essentially equivalent to the one we consider in the case of dg algebras. However, this `many object' derived localisation is, like the dg algebra localisation, rather inaccessible without first developing our fundamental theorem in this setting. The techniques and theory developed here are general enough that they can be extended beyond dg algebras to dg categories and this will be done in a sequel to this paper, in order to provide the means to work effectively with this dg category localisation.

We have decided to avoid any mention of dg categories in the present paper in order to keep our exposition elementary and focused. We do, however, discuss briefly the derived version of Cohn's (or matrix) localisation which is a sort of a half-way house between the categorical localisation and the ordinary one treated here. This is because Cohn's localisation is generally thought to be part of noncommutative ring theory (as opposed to category theory); it is very closely related to the notion of a \emph{homological epimorphism} of dg algebras: we show that the two notions are the same up to the failure of the telescope conjecture. The theory constructed here suggests a far-reaching generalisation of the usual applications of Cohn's localisation, in particular the theory of embeddings of noncommutative rings into skew-fields, but again we refrain from developing this theme in the present paper.

It is likely that a similar theory of localisation can be developed with other enriching closed symmetric monoidal model categories (or symmetric monoidal $\infty$--categories); simplicial localisation of Dwyer--Kan \cite{dk:simploc} is an example of this kind -- indeed, we show that the functor of singular chains carries simplicial localisation to (derived) dg localisation. In this vein, Lurie has also begun to develop localisation in the context of $\mathbb{E}_1$--algebras in symmetric monoidal $\infty$--categories \cite[\S{}7.2.4]{lurie}, focusing particularly on the simpler special case of localisation at an Ore set. It is worth noting that in the context of arbitrary enriching closed symmetric
monoidal model categories, localisations of commutative monoids in the category of associative monoids need not be commutative, c.f. \cite{HopkinsHill,HopkinsHillB} in the case of equivariant $E_\infty$ ring spectra. On the other hand, one expects that the relationship between module localisation and localisation of \emph{associative} monoids should hold in more general situations than dg algebras.

Our tools for working with derived localisation have a wide range of applications.  Many standard theorems involving the ordinary notion of localisation hold only under certain constraints and we are able to show that these can often be removed when localisation is considered in the derived framework. For example, it is well-known that localisation at a central set of elements in an algebra preserves Hochschild homology and cohomology and we show that the corresponding derived result is true without any centrality assumptions.

Next, one has a localisation long exact sequence in algebraic $K$--theory of noncommutative rings under the assumption of stable flatness; again this assumption is shown to be unnecessary in the derived context. In a similar fashion we show that certain recollements of derived categories associated to quotients by so-called stratifying idempotent ideals, important in the theory of quasi-hereditary rings, are valid for all idempotent ideals, as long as quotients are replaced by `derived quotients'. Underlying these results is the basic phenomenon, already mentioned, that the derived category of a derived localised ring is a full subcategory of the derived category of the original ring; this only holds in the non-derived context under the strong condition of stable flatness. Furthermore, we show that the derived localisation of a dg ring at an idempotent is a version of Drinfeld's quotient of a dg category \cite{Drinfeld_dg} and construct the corresponding small model for it. We also give a a very short proof of a strictification result: any dg algebra that is unital up to homotopy can be replaced by a strictly unital dg algebra; this replacement is unique up to homotopy.

One of the most striking applications of the theory of derived localisation is that, by combining our techniques with the simplicial localisation of Dwyer--Kan, we arrive at a very simple and conceptual proof of a general form of the well-known group completion theorem. Previously, this theorem has been regarded as somewhat mysterious (\cite[Section 3.2]{tillmann:gctcomment}). A number of versions of this statement have been given, most of which use adaptations of the proof of McDuff--Segal \cite{mcduffsegal:gct}; a proof which is not particularly homotopy theoretic in nature, especially given the homotopy theoretic nature of the statement. Our proof is completely different and driven entirely by the engine of homotopical algebra. Indeed, we show that it is really just a topological statement of a special case of an entirely natural and unmysterious result concerning derived localisation.

Influenced by this new algebraic perspective on the group completion theorem, we obtain an interpretation of the Loday--Quillen theorem on cyclic homology as a calculation of the derived localisation of a certain dg algebra and a similar interpretation of Kontsevich's theorem on graph homology. We also consider various monoids arising from mapping class groups of $2$--dimensional surfaces and compute the homology of their \emph{partial} group completions in terms of the stable mapping class groups and finally, we compute the derived localisation of the dg algebra of ribbon graphs.

The authors are grateful to the referee for a careful reading of this paper and a host of useful suggestions for improvement.
\subsection{Notation and conventions}
Throughout this paper we work with homologically graded chain complexes, unless otherwise stated. For a homologically $\mathbb{Z}$--graded chain complex $A$ with differential $d$, given $n\in \mathbb{Z}$ we denote $\Sigma^nA$ the $n$--fold suspension/desuspension given by $(\Sigma^n A)_i = A_{i-n}$ with differential $(\Sigma^n d)_i\co (\Sigma^nA)_i \to (\Sigma^n A)_{i-1}$ given by $(\Sigma^n d)_i(a) = (-1)^nd_{i-n}(a)$.

We will normally use the abbreviation `dg' for `differential graded'. By `dg (co)algebra' we will mean `dg (co)associative (co)algebra' unless otherwise stated.

Throughout this paper, $\ground$ will denote a fixed unital commutative ring. Unadorned tensor products will be assumed to be taken over $\ground$.

\section{Prerequisites on model categories of dg algebras and modules}
Denote by $\dgalg$ the category of unital differential $\mathbb{Z}$--graded associative algebras over $\ground$, in other words the category of unital monoids in the category of $\mathbb{Z}$--graded chain complexes of $\ground$--modules. This has the structure of a cofibrantly generated model category (the model structure transferred \cite[Theorem 4.1 (3)]{schwedeshipley:algebras} from that of $\mathbb{Z}$--graded chain complexes of $\ground$--modules \cite[Theorem 2.3.11]{hovey:modelcategories}) in which the weak equivalences are quasi-isomorphisms and the fibrations are surjections. As a technical point note that the domains in the set of generating cofibrations $\dgalg$ are compact, resulting in what was termed in \cite{MayP} a \emph{compactly generated} model category. This results in a minor technical simplification coming from the fact that the small object argument requires no transfinite induction. Consequently, cell dg algebras are constructing by \emph{countably} iterating the procedure of attaching simultaneously (a possibly uncountable) collection of cells.

For a dg algebra $A\in \dgalg$ we denote the under category of $A$ by $\underalg{A}$. Recall that this is the category with objects being dg algebras $C$ equipped with a map $A\to C$ and morphisms being commuting triangles. We will refer to the objects of $A\downarrow \dgalg$ as \emph{dg $A$--algebras}. Note that if $A$ happens to be commutative, we \emph{do not} insist that the map $A\to C$ be central. Recall that an under category in any model category has a natural model structure with weak equivalences, cofibrations and fibrations those maps which are weak equivalences, cofibrations and fibrations in the original model category.

A map of dg algebras $f\co A\to B$ induces a Quillen adjunction $\underalg{A}\rightleftarrows \underalg{B}$ given by restriction or pushout along $f$. It follows from \cite[Theorem B]{reedy1974homotopy} that if $A$ and $B$ are cofibrant then this adjunction is a Quillen equivalence whenever $f$ is a quasi-isomorphism. Consequently we make the following definition.

\begin{definition}
For any $A$ in $\dgalg$ we denote by $\derunderalg{A}$ the \emph{derived under category of $A$}, which is given by taking the homotopy category of the under category of a cofibrant replacement $\ground \cof Q \acycfib A$, so that $\derunderalg{A}=\Ho(\underalg{Q})$.

This is well-defined up to natural equivalence of categories. More precisely, different choices of cofibrant replacement functors give rise to different functors $A\mapsto \derunderalg{A}$, but they are naturally equivalent, since quasi-isomorphic cofibrant dg algebras have Quillen equivalent under categories.
\end{definition}

Given $A\in \dgalg$, denote by $\amod$ the category of differential $\mathbb{Z}$--graded $A$--modules. This has the structure of a cofibrantly generated model category in which the weak equivalences are quasi-isomorphisms and the fibrations are surjections \cite[Theorem 4.1 (1)]{schwedeshipley:algebras}.

\begin{convention}
We will refer to the objects of $\amod$ as \emph{$A$--modules}, in particular they will always be assumed to be differential graded unless explicitly stated.
\end{convention}

The category $\amod$ is left proper, i.e.~weak equivalences are preserved by pushouts along cofibrations. This can be seen as follows. Cofibrations of $\ground$--modules are, in particular, degreewise split injections \cite[Proposition 2.3.9]{hovey:modelcategories} and it follows that the same is true for the generating cofibrations in $\amod$ (which is given the transferred model structure). But pushouts in $\amod$ are created in $\ground$--modules and pushouts of chain complexes along injections preserve quasi-isomorphisms.

Let $M,N\in\amod$ and denote by $\Hom_A(M,N)$ the chain complex of $A$--linear maps from $M$ to $N$. If $M$ is cofibrant then $H_n(\Hom_A(M,N))$ is the set of maps from $\Sigma^n M$ to $N$ in the homotopy category. We write $\RHom_A(M,N)$ for the corresponding derived functor obtained by replacing $M$ with a cofibrant replacement if necessary.

\subsection{Derived free products of dg algebras}
In $\dgalg$, the pushout will be called the free product.

\begin{definition}
The \emph{free product}, denoted $B\ast_A C$, of dg algebras is the pushout in dg algebras:
\[
\xymatrix{
A \ar[r]\ar[d] & B\ar[d]\\
C\ar[r] & B\ast_A C
}
\]
This will often be viewed as the left adjoint $\underalg{A}\to\underalg{B}$ given by $C\mapsto B\ast_A C$, with right adjoint given by restriction along the map $A\to B$.
\end{definition}

Of particular importance to us is a corresponding derived construction.

\begin{definition}
Let $\ground \cof Q \acycfib A$ and $Q \cof P \acycfib B$ be cofibrant replacements for $A$ and $B$. The functor $\underalg{Q}\to\underalg{P}$ given by $C\mapsto P\ast_Q C$ is a left Quillen functor with right adjoint given by restriction.

The total derived functors give an adjunction $\derunderalg{A}\rightleftarrows \derunderalg{B}$, well-defined up to natural equivalence. The left adjoint will be denoted by $C\mapsto B\ast_A^{\mathbb{L}} C$ and called the \emph{derived free product}.

As an object of the homotopy category of $\dgalg$, $B\ast_A^{\mathbb{L}} C$ is simply the homotopy pushout in dg algebras.
\end{definition}

We need some basic facts about the homotopical properties of derived free products of dg algebras. In particular we will see that the derived free product can often be computed more simply, without needing to cofibrantly replace all the dg algebras involved.

\begin{definition}
An object $A\in\dgalg$ will be called \emph{left proper} if for any cofibrant replacement $P\acycfib A$ and any cofibration $P\cof X$, the map $g$ in the pushout diagram
\[
\xymatrix{
P \ar@{->>}^{\sim}[r]\ar@{>->}[d] & A\ar[d]\\
X\ar[r]^g & A\ast_{P} X
}
\]
is a quasi-isomorphism.
\end{definition}

Cofibrant dg algebras are left proper \cite[Theorem B]{reedy1974homotopy}. Left proper objects are of interest to us because they have homotopically correct under categories, in the following sense.

\begin{proposition}\label{prop:derivedundercatproper}
A dg algebra $A$ is left proper if and only if the adjunction \[\derunderalg{A}\rightleftarrows\Ho(\underalg{A})\] given by the total derived functors pushout and restriction is an equivalence.
\end{proposition}

\begin{proof}
This follows directly from the definition of left properness.
\end{proof}

\begin{remark}
Since the adjunction in the proposition above comes from choosing any cofibrant replacement $\ground \cof Q \acycfib A$ and considering the Quillen adjunction $\underalg{Q}\rightleftarrows\underalg{A}$ the proposition above could also be restated as follows: A dg algebra $A$ is left proper if and only if for any such cofibrant replacement this Quillen adjunction is a Quillen equivalence.
\end{remark}

Given a map $A\to B$ the following diagram of functors commutes (up to natural isomorphism)
\[
\xymatrix{
\derunderalg{A}\ar[r]\ar[d] & \Ho(\underalg{A})\ar[d]\\
\derunderalg{B}\ar[r] & \Ho(\underalg{B})
}
\]
where the arrows are all induced by (total derived functors of) pushouts. If $A$ and $B$ are left proper then the horizontal arrows are equivalences of categories. In this case, the derived pushout between derived under categories of left proper objects is therefore equivalent to the total left derived functor of the usual pushout between the normal under categories.

\begin{remark}
This means that, when dealing with left proper dg algebras, the derived under category and derived free product are modelled correctly, in a homotopical sense, by the usual under category and the usual free product. This connection between left properness and under categories was observed in \cite{rezk}.
\end{remark}

In particular, when all dg algebras involved are left proper, one can compute the derived free product by cofibrantly replacing either dg algebra:

\begin{corollary}[Balancing for the derived free product]\label{cor:balancing}
If $B$ and $C$ are dg $A$--algebras and $A,B,C$ are all left proper then the derived free product $B\ast_A^{\mathbb{L}} C$ can be computed by cofibrantly replacing just one of either $B$ or $C$ in $\underalg{A}$.\qed
\end{corollary}

\begin{remark}\label{rem:replaceone}
The corollary above also implies that, since cofibrant dg algebras are left proper, provided $A$ and $C$ are left proper (but perhaps $B$ is not) then to compute $B\ast_A^{\mathbb{L}} C$ it is sufficient to just cofibrantly replace $B$. In more detail, if $A$ and $C$ are left proper and $B'$ is a cofibrant replacement for $B$ under $A$ then $B'\ast_A^{\mathbb{L}} C$ is quasi-isomorphic to $B \ast_A^{\mathbb{L}} C$. Since $A,B',C$ are now all left proper we can compute the former derived free product by cofibrantly replacing just $B'$ in $\underalg{A}$. But $B'$ is already cofibrant under $A$.
\end{remark}

\begin{remark}
The definitions and results above are, of course, valid for any model category. In particular, a model category is left proper (meaning that weak equivalences are preserved under pushout along cofibrations) if and only if all the objects are left proper.
\end{remark}

\begin{definition}
A left $A$--module $M$ is \emph{flat} over $A$ if for any right $A$--module $N$, it holds that $N \otimes_A^{\mathbb{L}} M\to N\otimes_A M$ is a quasi-isomorphism.

We will call a dg algebra $A$ flat if the underlying $\ground$--module of $A$ is flat over $\ground$. In particular, if $A$ is a cofibrant dg algebra then it is flat.
\end{definition}

\begin{remark}
The above notion of flatness for dg modules is also called \emph{homotopically flat} \cite{Drinfeld_dg} or \emph{$K$--flat} \cite{spalt}.
\end{remark}

\begin{theorem}\label{thm:flatleftproper}
Any flat dg algebra is left proper.
\end{theorem}

\begin{proof}
Let $S^{n-1} = \ground[x]$ denote the free dg algebra generated by a single element $x$ of degree $n-1$ with $d(x)=0$. Let $D^n = \ground \langle z, dz \rangle$ denote the free dg algebra generated by two elements $z$ and $dz$ of degrees $n$ and $n-1$ with $d(z) = dz$. Note that $H(D^n) \cong \ground$. There is a natural inclusion $S^{n-1}\hookrightarrow D^n$ given by sending $x$ to $dz$. Recall \cite{schwedeshipley:algebras} that the model category $\dgalg$ is cofibrantly generated, with generating cofibrations given by $I = \{\, S^{n-1}\cof D^n : n\in \mathbb{Z}\,\}$ and generating acyclic cofibrations given by $J= \{\, \ground \acyccof D^n : n\in\mathbb{Z}\,\}$.

Let $A$ be a flat dg algebra, let $f\co P\acycfib A$ be a cofibrant replacement for $A$ and let $i\co P\cof X$ be a cofibration.

First assume that $i\co P\cof X$ is obtained as some pushout of a generating cofibration in $I$ along a map $S^{n-1}\to P$ so that $X \cong P\ast_{S^{n-1}} D^n$ and consider the commutative diagram:
\[
\xymatrix@C=4em{
S^{n-1}\ar[r]\ar@{>->}[d]& P\ar@{->>}[r]^{\sim}_f\ar@{>->}[d]^i & A\ar[d]\\
D^{n}\ar[r] & P\ast_{S^{n-1}} D^n\ar[r]^g & A\ast_P (P\ast_{S^{n-1}} D^n)
}
\]
We want to show $g$ is a quasi-isomorphism. Since the two inner squares are pushouts so is the outer square, by the pasting property of pushouts, so $A\ast_P(P\ast_{S^{n-1}}D^n)\cong A\ast_{S^{n-1}} D^n$ and the map $g\co P\ast_{S^{n-1}} D^n\to A\ast_{S^{n-1}} D^n$ is the obvious map induced by $f$. Elements of $P\ast_{S^{n-1}} D^n$ are all sums of the form $a_1 z^{j_1}a_2z^{j_2}\dots a_kz^{j_k}a_{k+1}$ where $a_i\in P$ and $z$ is the degree $n$ element in $D^n$. In particular, forgetting the differential, this means that the underlying $\ground$--module of $P\ast_{S^{n-1}} D^n$ is isomorphic to a direct sum of iterated tensor products over $\ground$ of $P$ (again, forgetting the differential). Consider the filtration on $P\ast_{S^{n-1}} D^n$ given by setting $F_i$ to be the subspace generated by elements of this form which have $j_1+j_2+\dots+j_k \leq i$. Then the filtration of chain complexes $0\subset F_0 \subset F_1 \subset \dots \subset P\ast_{S^{n-1}}D^n$ is bounded below and exhaustive. The quotient $F_i/F_{i-1}$ of chain complexes is isomorphic, as a chain complex, to a direct sum of iterated tensor products of $P$ with itself over $\ground$ since the part of the differential coming from that on $D^n$ decreases filtration degree so only the contribution from the differential on $P$ survives. Similarly we get a filtration $0\subset F'_0\subset F'_1\subset \dots \subset A\ast_{S^{n-1}}D^n$ with the same property. Furthermore, $P$, being cofibrant, is in particular flat over $\ground$ and since $A$ is also flat over $\ground$, the quasi-isomorphism $f$ induces a quasi-isomorphism between the chain complexes $F_i/F_{i-1}$ and $F'_i/F'_{i-1}$ (since they are both just a direct sum of iterated tensor products of quasi-isomorphic flat $\ground$--modules) and so the $E^1$~terms of the associated convergent spectral sequences are isomorphic and therefore $g$ is a quasi-isomorphism.

The result now follows by a general argument for cofibrantly generated model categories. In particular if $i\co P\cof X$ is obtained as a transfinite composition of pushouts of generating cofibrations (in other words, $i$ is a relative cell complex) then by transfinite induction it follows that pushouts of quasi-isomorphisms along relative cell complexes are quasi-isomorphisms (since a filtered colimit of quasi-isomorphisms is a quasi-isomorphism). Finally, since the set of generating cofibrations admits the small object argument then any cofibration $i\co P\cof X$ can be factored as $P\stackrel{j}{\cof} Y \stackrel{p}{\fib} X$ with $j$ a relative cell complex and $p$ an acyclic fibration. But then $i$ satisfies the left lifting property with respect to $p$ so, by the retract lemma, $i$ is a retract in $\underalg{P}$ of $j$. In particular it follows that pushouts along $i$ are retracts of pushouts of along $j$. But $j$ is a relative cell complex so pushouts of quasi-isomorphisms along $j$ are quasi-isomorphisms. Retracts of quasi-isomorphisms are quasi-isomorphisms so this is true for $i$ as well.
\end{proof}

\begin{remark}
If $\ground$ is a field then every dg algebra is of course flat and hence \autoref{thm:flatleftproper} shows that, in this case, $\dgalg$ is left proper. On the other hand, if $\ground$ is not a field, then $\dgalg$ is \emph{not} left proper, in general. The following construction is a slight variation of \cite[Example 2.11]{rezk}. Let $\ground=\mathbb Z$ and consider a $\mathbb Z$--algebra $A$ (with vanishing differential). Let
\[
A\langle x\rangle:=A*_{\mathbb Z}{\mathbb Z}[x]\cong A\oplus A\otimes_{\mathbb Z}A\oplus A^{\otimes_{\mathbb Z}3}\oplus\dots,
\]
the algebra obtained from $A$ by adjoining freely a generator $x$ in degree zero. Note that $A\langle x\rangle$ is a free product where one of the factors, namely $\mathbb Z[x]$, is cofibrant, but the other, $A$, need not be. If $A$ is $\mathbb Z$--flat, i.e.~it has no torsion, then $A\langle x\rangle$ has the correct homotopy type by \autoref{cor:balancing} and \autoref{thm:flatleftproper}. Suppose that $A$ does have torsion and let $\tilde{A}$ be its cofibrant (hence torsion-free) replacement. Then the homology of the dg algebra
\[
\tilde{A}*^{\mathbb{L}}_{\mathbb Z}{\mathbb Z}[x]\simeq\tilde{A}\langle x\rangle=\tilde{A}*_{\mathbb Z}{\mathbb Z}[x]\cong\tilde{A}\oplus \tilde{A}\otimes_{\mathbb Z}\tilde{A}\oplus \tilde{A}^{\otimes_{\mathbb Z}3}\oplus\dots
\]
has $\Tor^{\mathbb Z}(A,A)$ as a summand and so it will definitely be different from the homology of $A\langle x\rangle$. We conclude that an algebra over $\mathbb Z$ is left proper if and only if it has no torsion.
\end{remark}

\subsection{Bigraded resolutions of algebras}
It is a well-known fact that if $A$ is a graded algebra and $M$, $N$ are graded $A$--modules (all with vanishing differentials) then $\Tor^A(M,N)$ is \emph{bigraded}. The reason is that one can choose bigraded $A$--projective resolutions of $M$ and $N$. An analogous result holds for derived free products of algebras.

\begin{proposition}\label{prop:bigraded}
Let $A$, $B$ be $\ground$--algebras with vanishing differentials and assume that $B$ is an $A$--algebra through an algebra map $A\to B$. Then there is a differential bigraded $A$--cofibrant algebra $\tilde{B}=\{\tilde{B}_{i,j}\}, i\in \mathbb{Z}, j\in\mathbb{N}$ supplied with a differential $d\co\tilde{B}_{i,j}\to \tilde{B}_{i,j-1}$. There is a quasi-isomorphism of differential bigraded $A$--algebras $\tilde{B}\to B$ where $B$ is viewed as trivially bigraded: $B=\{B_{i,0},i\in\mathbb Z\}$.
\end{proposition}

\begin{proof}
Let us consider the free $A$--algebra $\tilde{B}_0:=A\langle\{x^{\alpha_0}\}\rangle, \alpha_0\in B$ with one generator $x^{\alpha_0}$ for every homogeneous element $\alpha_0\in B$. Clearly there is a surjective map of $A$--algebras $i_0\co\tilde{B_0}\to B$.

Next, for every element $\alpha_1\in\ker i_0$ consider a generator $x^{\alpha_1}$ and a dg algebra $\tilde{B}_1:=A\langle\{x^{\alpha_0}\},\{x^{\alpha_1}\}\rangle$ with the differential $d(x^{\alpha_1})=\alpha_1$. We extend the map $i_0$ to a map $i_1\co\tilde{B}_1\to B$ by setting $i_1(x^{\alpha_1})=0$.

Proceeding by induction we construct a sequence of dg $A$--algebras and dg $A$--algebra maps
\[
\xymatrix
{
\tilde{B_0}\ar[r]\ar_{i_0}[d]&\tilde{B_1}\ar^{i_1}[dl]\ar[r]&\dots\ar[r]&\tilde{B}_n\ar^{i_n}[dlll]\\
B
}
\]
such that $\tilde{B}_n:= A\langle\{x^{\alpha_0}\},\{x^{\alpha_1}\},\dots, \{x^{\alpha_n}\}\rangle$. Assigning the generators $x^{\alpha_j}$ the second (weight) grading $j$, we can view $\tilde{B}_n$ as bigraded. We construct $\tilde{B}_{n+1}$ by adjoining to $\tilde{B}_n$ one generator $x^{\alpha_{n+1}}$ for every homogeneous element $\alpha_{n+1}\in \ker H(i_n)\co H(\tilde{B}_n)\to B$ of weight grading $n$. Choosing a representative $\tilde{\alpha}_{n+1}\in \tilde{B}_{n}$ of the corresponding cohomology class, we define the differential on $\tilde{B}_{n+1}$ by $d(x^{\alpha_{n+1}})=\tilde{\alpha}_{n+1}$. Furthermore, the map $i_n$ extends to $i_{n+1}\co\tilde{B}_{n+1}\to B$ by setting $i_{n+1}(x^{\alpha_{n+1}})=0$.
It is clear that the resulting dg $A$--algebra map $\tilde{i}\co\tilde{B}:=\lim_{ n}\tilde{B}_n\to B$ is a quasi-isomorphism and the differential on $\tilde{B}$ is compatible with the bigrading as required.
\end{proof}

\begin{remark}\label{comparison}
Given another bigraded resolution $\hat{B}=\{\hat{B}_{ij}\}$ of the $A$--algebra $B$, not necessarily cofibrant, it is easy to see that there is a bigraded quasi-isomorphism $\tilde{B}\to \hat{B}$ such that the diagram of dg $A$--algebras
\[
\xymatrix
{
\tilde{B}\ar[r]\ar[d]&\hat{B}\ar[dl]\\
B
}
\]
is homotopy commutative.
\end{remark}

\begin{corollary}\label{cor:bigraded}
Let $A$, $B$ and $C$ be graded algebras with vanishing differentials and let $B$ and $C$ be also $A$--algebras via maps $A\to B$ and $A\to C$. Assume also that $A$ is left proper. Then $H(B*^{\mathbb{L}}_AC)$ is a naturally differential bigraded $\ground$--algebra.
\end{corollary}

\begin{proof}
By  \autoref{cor:balancing}, the $\ground$--algebra $H(B*^{\mathbb{L}}_AC)$ can be computed as the homology of $\tilde{B}*_A\tilde{C}$ where $\tilde{B}$ and $\tilde{C}$ are $A$--cofibrant differential bigraded resolutions of $B$ and $C$ constructed above. It follows that $\tilde{B}*_A\tilde{C}$ possesses a total bigrading that is compatible with the differential and the conclusion follows.
\end{proof}

\begin{remark}\mbox{}
\begin{itemize}
\item If the algebras $A$, $B$ and $C$ as above are \emph{ungraded} (or concentrated in degree $0$) then clearly the bigrading on $H(B\ast^\mathbb{L}_AC)$ reduces to a single grading (which coincides with the weight, or resolution, grading).
\item We have $H_{i,0}(B*^\mathbb{L}_AC)\cong B*_AC$.
\end{itemize}
\end{remark}

\subsection{Derived endomorphism dg algebras of bimodules}
Let $A,B$ be dg algebras and let $M$ be an $(A,B)$--bimodule (in other words $M\in\modcat{A\otimes B^{\op}}$). Then $\Hom_B(M,M)$ is a dg $A$--algebra with the map $A\to \Hom_B(M,M)$ given by $a\mapsto l_a$ with $l_a\co M\to M$ the left multiplication by $A$. We will now show that this construction can be derived in an appropriate sense.

Note that this is only a slight generalisation of the standard construction of the derived endomorphism dg algebra of a module. However, since we start with a bimodule, we are obtaining not just a dg algebra but a dg $A$--algebra and as such we must show that quasi-isomorphic bimodules give not just quasi-isomorphic dg algebras, but quasi-isomorphic dg $A$--algebras.

We will make use of the following almost obvious lemma.

\begin{lemma}\label{lem:bimodtomod}
If $A,B$ are dg algebras and $A$ is cofibrant as a $\ground$--module and $P$ is a cofibrant $(A,B)$--bimodule then $P$ is also cofibrant when regarded as a right $B$--module.
\end{lemma}

\begin{proof}
It is enough to show that the forgetful functor from $(A,B)$--bimodules to right $B$--modules preserves cofibrations. Since it preserves colimits, it suffices to show that it preserves generating cofibrations. The image of a generating cofibration is just a generating cofibration of right B-modules tensored with $A$ (which now is regarded just as a dg $\ground$--module). This is a cofibration since, because $A$ is a cofibrant $k$--module, tensoring with $A$ is a left Quillen functor on the category of $B$--modules.
\end{proof}

\begin{remark}
The condition that $A$ is cofibrant as a $\ground$--module would hold, for example, if $A$ were a cofibrant dg algebra \cite[Theorem 4.1 (3)]{schwedeshipley:algebras}.
\end{remark}

\begin{lemma}\label{lem:endzigzag}
Let $A,B$ be dg algebras, with $A$ cofibrant as a $\ground$--module and let $P,Q$ be cofibrant quasi-isomorphic $(A,B)$--bimodules. Then $\Hom_B(P,P)\simeq \Hom_B(Q,Q)$ as dg $A$--algebras.
\end{lemma}

\begin{proof}
Since $P$ and $Q$ are cofibrant $(A,B)$--bimodules and quasi-isomorphic, this quasi-isomorphism is represented by an actual map $f\co P\xrightarrow{\sim} Q$ of $(A,B)$--modules (recall that all objects are fibrant here). Factor this map as $p\circ i\co P\acyccof R\acycfib Q$. The maps $i$ and $p$ have a left inverse and a right inverse respectively by standard model category arguments, namely applying the left lifting property and right lifting property to the following diagrams:
\[
\xymatrix{
0\ar@{>->}[d]\ar[r] & R\ar@{->>}[d]^{p} \\
Q\ar[r]^{\id}\ar@{.>}[ru] & Q
}\qquad
\xymatrix{
P\ar@{>->}[d]^{i}\ar[r]^{\id} & P\ar@{->>}[d] \\
R\ar[r]\ar@{.>}[ru] & 0
}
\]
Therefore it is sufficient to prove the special cases when $f$ has a left inverse and when $f$ has a right inverse. Furthermore if $f$ has a right inverse then the right inverse is, of course, a quasi-isomorphism with left inverse $f$, so it is in fact sufficient to just prove the case when $f$ has a left inverse.

Denote the left inverse by $f^l\co Q\twoheadrightarrow P$ and denote the kernel of $f^l$ by $L$. Since $f^l$ is split, so is $L\hookrightarrow Q$ therefore $L$ is a retract of $Q$ and hence is also a cofibrant $(A,B)$--bimodule. Since $A$ is a cofibrant $\ground$--module, $P,Q$ and $L$ are cofibrant as right $B$--modules. Therefore, $\Hom_B(L,L)\simeq 0$ and the map of complexes $\Hom_B(P,P)\to \Hom_B(Q,Q)$ given by $g\mapsto f\circ g \circ f^l$ is a quasi-isomorphism.

Since $Q\cong P\oplus L$ then there is a map $\Hom_B(L,L)\times \Hom_B(P,P)\to \Hom_B(Q,Q)$. This is a map of dg $A$--algebras, which, by the arguments above, is also a quasi-isomorphism. Similarly, the projection map $\Hom_B(L,L)\times \Hom_B(P,P)\to \Hom_B(P,P)$ is also a quasi-isomorphism of dg $A$--algebras.

Therefore there is a zig-zag of quasi-isomorphisms of dg $A$--algebras between $\Hom_B(Q,Q)$ and $\Hom_B(P,P)$ so they are in the same quasi-isomorphism class of dg $A$--algebras as required.
\end{proof}

\begin{remark}
Although there appears to be direct map of dg algebras $\Hom_B(Q,Q)\to \Hom_B(P,P)$, this not a map of \emph{unital} dg algebras and so is not sufficient to prove the preceding lemma.
\end{remark}

\begin{definition}
Let $A,B$ be dg algebras with $A$ cofibrant as a $\ground$--module and let $M$ be an $(A,B)$--bimodule, then we define the \emph{derived endomorphism dg $A$--algebra of $M$}, denoted $\REnd_B(M)$, to be the dg $A$--algebra $\Hom_B(P,P)$, where $P$ is a cofibrant replacement for $M$ as an $(A,B)$--bimodule.
\end{definition}

The following theorem tells us that the derived endomorphism algebra is well-defined up to quasi-isomorphism and is quasi-isomorphism invariant.

\begin{theorem}\label{thm:endhomotopy}
Let $M$ be an $(A,B)$--bimodule. Then $\REnd_B(M)$ is well-defined up to quasi-isomorphism of dg $A$--algebras, i.e.~it does not depend on the choice of cofibrant replacement for $M$. Furthermore if $N$ is an $(A,B)$--bimodule and $M\simeq N$ as $(A,B)$--bimodules then $\REnd_B(M)\simeq \REnd_B(N)$ as dg $A$--algebras.
\end{theorem}

\begin{proof}
In both cases it is sufficient to prove that if $P$ and $Q$ are quasi-isomorphic cofibrant $(A,B)$--modules then $\Hom_B(Q,Q)$ and $\Hom_B(P,P)$ are quasi-isomorphic as dg $A$--algebras, which is just \autoref{lem:endzigzag}.
\end{proof}

\begin{remark}
Let $A$ be a dg algebra which is cofibrant as a $\ground$--module and let $f\co A'\acycfib A$ be a cofibrant replacement for $A$ as an $(A,A)$--bimodule. Since $A$ is a cofibrant $\ground$--module, $A'$ is cofibrant as a right $A$--module. We have the following commutative diagram
\[
\xymatrix{
\Hom_A(A',A')\ar[r]^{f_*}_{\simeq} & \Hom_A(A',A)\\
A\ar[u]\ar[r]_{\cong} & \Hom_A(A,A)\ar[u]^{f^*}_{\simeq}
}
\]
where the left vertical map is the map making $\REnd_A(A)\simeq \Hom_A(A',A')$ into a dg $A$--algebra. The maps $f_*$ and $f^*$ are quasi-isomorphisms since $f$ is a quasi-isomorphism between cofibrant--fibrant right $A$--modules. It follows that the left vertical map is a quasi-isomorphism, in other words $\REnd_A(A)\simeq A$ as dg $A$--algebras.
\end{remark}

\section{Derived localisation of dg algebras}

Let $A$ be a dg algebra and let $S\subset H(A)$ be an arbitrary subset of homogeneous homology classes.

\begin{definition}\label{def:S-inverting}
A dg $A$--algebra $f\co A\to Y$ will be called \emph{$S$--inverting} if for all $s\in S$ the homology class $f_*(s)\in H(Y)$ is invertible in the algebra $H(Y)$.

Let $p\co A'\acycfib A$ be a cofibrant replacement for $A$, so that $p_*\co H(A')\to H(A)$ is an isomorphism. A dg algebra $Y\in\derunderalg{A}\simeq \Ho(\underalg{A'})$ will be called $S$--inverting if it is $S$--inverting as a dg $A'$--algebra. Furthermore, for a dg algebra map $f\co A\to B$, we will refer to $f_*(S)$--inverting dg $B$--algebras as simply $S$--inverting dg $B$--algebras.
\end{definition}

\begin{proposition}\label{prop:sinvertingadjunction}
Given a map $A\to B$ between cofibrant dg algebras, the adjunction $\derunderalg{A}\rightleftarrows\derunderalg{B}$ restricts to an adjunction between the full subcategories of $S$--inverting dg algebras.
\end{proposition}

\begin{proof}
Restriction clearly preserves the property of being $S$--inverting. Given $C\in\derunderalg{A}$ then there is a map $C\to B\ast_{A}^{\mathbb{L}} C$ in the derived under category of $A$ which is $S$--inverting if and only if the map $B\to B\ast_A^{\mathbb{L}} C$ is $S$-inverting, since both maps agree upon precomposition with the maps from $A$. Therefore $B\ast_{A}^{\mathbb{L}} C$ is $S$--inverting if $C$ is. Thus the derived free product preserves the property of being $S$--inverting.
\end{proof}

\begin{definition}\label{def:alglocalisation}
The derived localisation of $A$, denoted $L_S^{\dgalg}(A)$, is the initial object in the full subcategory on $S$--inverting dg algebras of $\derunderalg{A}$.
\end{definition}

\begin{remark}
Clearly, a localisation of $A$, if it exists, is unique up to a unique isomorphism in the derived under category of $A$. Because of that, we will refer to it as \emph{the} localisation of $A$. The existence will be established later on but in the mean time, it will simply be assumed; of course care will be taken to avoid circularity of arguments.
\end{remark}

Any quasi-isomorphism of dg algebras $A\to A'$ induces an equivalence $\derunderalg{A}\rightleftarrows\derunderalg{A'}$ and this restricts to an equivalence between the full subcategories on $S$--inverting dg algebras. This means that the definition of localisation above is well-defined in the sense that the subcategory of $S$--inverting dg algebras, up to natural equivalence, is independent of the choice of cofibrant replacement. Moreover, it implies that localisation is a quasi-isomorphism invariant:

\begin{proposition}\label{prop:equivloc}
Let $A\we A'$ be a quasi-isomorphism. Then this induces a quasi-isomorphism $L_S^{\dgalg}(A)\we L_S^{\dgalg}(A')$ as dg algebras.
\end{proposition}

\begin{proof}
The dg $A'$--algebra $L_S^{\dgalg}(A')$ and the dg $A$--algebra $L_S^{\dgalg}(A)$ are initial $S$--inverting objects in the corresponding derived under categories. Since the restriction functor is an equivalence between these under categories, it sends $L_S^{\dgalg}(A')$ into $L_S^{\dgalg}(A)$.
\end{proof}

\begin{remark}
It is clear that in order to study the localisation of $A$ we may assume, without loss of generality, that $A$ is cofibrant. However if $A$ is flat, or more generally, left proper, from the practical point of view it is not necessary to cofibrantly replace it.
\end{remark}

If $A$ is cofibrant (or more generally, left proper) the localisation of $A$ is represented by an $S$--inverting cofibrant dg $A$--algebra $A\to L_S^{\dgalg}(A)$ with the property that for any other $S$--inverting dg $A$--algebra $A\to Y$ there is a map of dg $A$--algebras $L_S^{\dgalg}(A)\to Y$ which is unique up to homotopy of dg $A$--algebra maps.

\begin{lemma}\label{lem:pushoutlemma}
Let $B$ be a dg $A$--algebra. Then $B\ast_A^{\mathbb{L}}L_S^{\dgalg}(A)$ is the localisation of $B$.
\end{lemma}

\begin{proof}
This follows directly from \autoref{prop:sinvertingadjunction}.

\end{proof}

\begin{remark}
Let $Y$ be an arbitrary dg algebra. Given any homotopy class of maps from $L_S^{\dgalg}(A)$ to $Y$ we obtain a homotopy class of maps from $A$ to $Y$ which is $S$--inverting. The universal property of localisation implies that this map is surjective. It is also injective since a right homotopy of $S$--inverting maps from (a cofibrant replacement of) $A$ to $Y$ is also $S$--inverting and hence, again by the universal property, factors through a right homotopy of maps from $L_S^{\dgalg}(A)$ to $Y$.
This gives an equivalent definition for the localisation of $A$ as a representing object of the functor which takes a dg algebra $Y$ to the set of homotopy classes of maps from $A$ to $Y$ which are $S$--inverting.
\end{remark}

Let $\ground \langle S \rangle$ denote the free graded algebra over $\ground$ generated by the set $S$, where $s\in S$ has the same degree in $\ground \langle S \rangle$ as it does in $H(A)$. Denote $\ground \langle S, S^{-1} \rangle$ the graded algebra generated by the symbols $\{s, s^{-1}\}_{s\in S}$, with $s^{-1}$ having degree the negative of the degree of $s$, modulo the ideal generated by the relations $s^{}s^{-1} = s^{-1}s = 1$. We will prove the following lemma later but we state it here since, in combination with the previous lemma, it allows us to identify explicitly the localisation of a dg algebra and, in particular, show that it always exists.

\begin{lemma}\label{lem:simplelocalisation}
The localisation of $\ground \langle S\rangle$ is given by $\ground \langle S,S^{-1}\rangle$.
\end{lemma}

By choosing a cycle representing each homology class $s\in S\subset H(A)$ we obtain a map $\ground \langle S \rangle \to A$. We obtain the following theorem immediately from the previous lemmata:

\begin{theorem}\label{cor:alglocalisation}
The localisation of $A$ is given by $ L_S^{\dgalg}(A)\simeq A \ast_{\ground \langle S \rangle}^{\mathbb{L}}\ground\langle S,S^{-1}\rangle$.\qed
\end{theorem}

\begin{remark}\mbox{}
\begin{itemize}
\item Since $\ground\langle S \rangle$ and $\ground \langle S, S^{-1} \rangle$ are free, and thus flat, over $\ground$, it is only necessary to cofibrantly replace $A$ under $\ground \langle S \rangle$ when computing the derived free product $A \ast_{\ground \langle S \rangle}^{\mathbb{L}}\ground\langle S,S^{-1}\rangle$, by \autoref{rem:replaceone} and \autoref{thm:flatleftproper}. On the other hand, if $A$ is flat over $k$, or more generally left proper, then instead just cofibrantly replacing $\ground \langle S, S^{-1} \rangle$ under $\ground\langle S \rangle$ does the job; an explicit replacement is given in \autoref{sec:explicitmodel} below.
\item The above result implies that the derived localisation of a dg algebra can be represented both as a derived pushout and a strict pushout in the homotopy category. This is a surprising phenomenon since one normally does not expect the existence of strict colimits in homotopy categories. Related to this is the non-existence of `higher' localisations, cf.~below.
\end{itemize}
\end{remark}

\subsection{Non-derived localisation}\label{nonderivedalgebralocalisation}
Recall the standard definition of the (non-derived) localisation. Let $A$ be a dg algebra and let $S\subset Z(A)$ be an arbitrary subset of homogeneous cycles in $A$.

\begin{definition}
A dg $A$--algebra $f\co A\to Y$ will be called \emph{strictly $S$--inverting} if for all $s\in S$ the cycle $f(s)\in Y$ is invertible in the dg algebra $Y$.
\end{definition}

\begin{definition}
The non-derived localisation of $A$, denoted $A[S^{-1}]$, is the initial object in the full subcategory on strictly $S$--inverting dg algebras of $\underalg{A}$.
\end{definition}

\begin{remark}
Normally one would define non-derived localisation for ungraded algebras without a differential. The definition above, of course, reduces to the same thing in that case by regarding such algebras as concentrated in degree $0$.
\end{remark}

Note, therefore, that the definition of the derived localisation is essentially the same, but formulated in the derived under category of $A$ in order to be a quasi-isomorphism invariant notion. However, this makes derived localisation a good deal less trivial. In particular, the following non-derived analogue of \autoref{cor:alglocalisation} is obvious, whereas the proof of \autoref{cor:alglocalisation} is much more involved.

\begin{proposition}
The non-derived localisation of $A$ is given by $A[S^{-1}]\cong A\ast_{\ground \langle S \rangle } \ground \langle S, S^{-1} \rangle$.
\end{proposition}

\begin{proof}
It is clear that any strictly $S$--inverting dg $A$--algebra admits a unique map from $\ground \langle S, S^{-1} \rangle$ extending the map from $\ground \langle S \rangle$. By the universal property of pushouts the result follows.
\end{proof}

\subsection{Homotopy coherence}
Our definition of localisation is as a universal $S$--inverting object in a certain homotopy category (`initial $S$--inverting dg algebra up to homotopy'). One could also demand a stronger property be satisfied: to be the universal $S$--inverting object in a certain $\infty$--category (`initial $S$--inverting dg algebra up to coherent higher homotopy'), as opposed to just in the underlying homotopy category. We will now show that this higher homotopy coherence is, in fact, automatically satisfied.

Recall that the derived mapping space $\Map(X,Y)$ between objects $X$ and $Y$ in a category with weak equivalences is (the homotopy type of) the space of morphisms in the Dwyer--Kan simplicial localisation \cite{dk:simploc}. In particular, the set of connected components of the derived mapping space is the set of morphisms in the homotopy category.

Recall also the notion of a homotopy epimorphism of \cite{muro}.
\begin{definition}\label{def:homotopy_epimorphism}
	A morphism $f\co X \to Y$ in a model category is said to be a
	homotopy epimorphism if for any object $Z$, the induced morphism on (derived)
	mapping spaces
	\[
	\Map(Y,Z)\to\Map(X,Z)
	\]
	gives an injection on connected components of the corresponding simplicial sets and an isomorphism on homotopy groups for any choice of a base point.
\end{definition}
\begin{proposition}\label{prop:homotopy_epimorphism}
For a dg algebra $A$, the localisation map $A\to L_S^{\dgalg}(A)$ is a homotopy epimorphism.
\end{proposition}
\begin{proof}
The map $L_S^{\dgalg}(A)\to L_S^{\dgalg}(A)*^{\mathbb L}_AL_S^{\dgalg}(A)$ given by the inclusion of the right factor is the localisation of $L_S^{\dgalg}(A)$ by \autoref{lem:pushoutlemma}. Since the latter is already $S$--inverting, this map is a quasi-isomorphism. Now the claim follows from the characterisation of homotopy epimorphisms given in \cite[Proposition 2.1]{muro}.
\end{proof}

Let $A$ be a dg algebra and let $A'$ be a cofibrant replacement for $A$. For $X,Y\in\underalg{A'}$ we denote by $\textrm{Map}_{A}(X,Y)$ the derived mapping space from $X$ to $Y$ in the model category $\underalg{A'}$. This does not depend, up to weak equivalence, on the choice of cofibrant replacement for $A$. The set of connected components of $\textrm{Map}_A(X,Y)$ is the set of maps from $X$ to $Y$ in $\derunderalg{A}$.

\begin{theorem}
Let $L_S^{\dgalg}(A)\in\derunderalg{A}$ be the localisation of $A$. Then for any $B\in\derunderalg{A}$ which is $S$--inverting, $\Map_{A}(L_S^{\dgalg}(A),B)\simeq *$. If $B$ is not $S$--inverting then the derived mapping space is empty.
\end{theorem}

\begin{proof}
Without loss of generality, assume $A$ is cofibrant and assume that the localisation $A\cof L_S^{\dgalg}(A)$ is a cofibrant dg $A$--algebra. Since the localisation map $A\to L_S^{\dgalg}(A)$ is a homotopy epimorphism (\autoref{prop:homotopy_epimorphism}), the homotopy fibre of the map
\[
\Map(L_S^{\dgalg}(A),B)\to\Map(A,B)
\] is empty or contractible. If $B$ is $S$--inverting, then it is contractible, otherwise it is empty.  But this homotopy fibre is precisely $\Map_{A}(L_S^{\dgalg}(A),B)$.
\end{proof}

\begin{remark}
Combining the result above with \autoref{lem:simplelocalisation} it follows that the localisation of $A$ at $S$ is a fibrant replacement in the Bousfield localisation of the model category of dg $\ground \langle S \rangle$--algebras with respect to the map of $\ground\langle S \rangle$--algebras $f\co \ground \langle S \rangle \to \ground \langle S, S^{-1} \rangle$.

Indeed, \autoref{lem:simplelocalisation} and the result above implies that a $\ground \langle S \rangle$--algebra is $f$--local if and only if it is $S$--inverting. Then the above result implies that $A\to L_S^{\dgalg}(A)$ is an $f$--local equivalence.
\end{remark}

\subsection{An explicit model for localisation}\label{sec:explicitmodel}
If $A$ is a left proper dg algebra, then the localisation $ L_S^{\dgalg}(A)\simeq A \ast_{\ground \langle S \rangle}^{\mathbb{L}}\ground\langle S,S^{-1}\rangle$ may be computed by cofibrantly replacing $\ground\langle S,S^{-1}\rangle$ in the under category
$\underalg{\ground \langle S \rangle}$ and without having to replace $A$, cf.~\autoref{cor:balancing}. The condition on $A$ holds, for example, when $A$ is flat over $k$, by~\autoref{thm:flatleftproper}.

In the case of $S=\{s\}$, in other words inverting a single homogeneous element $s$, we can easily obtain a simple and explicit cofibrant replacement for $\ground \langle s,s^{-1}\rangle = \ground [ s,s^{-1}]$. We will denote this replacement by $Q$, which we now describe for completeness.

For the case of an arbitrary set $S$ a cofibrant replacement for $\ground \langle S, S^{-1} \rangle$ is then given by a free product of multiple copies of $Q$, one for each $s\in S$. Indeed, using \autoref{cor:balancing} and \autoref{thm:flatleftproper} we can reduce this, by induction, to showing that $\ground \langle S\setminus \{s\}, S^{-1}\setminus \{s^{-1}\} \rangle \ast_\ground Q\to \ground\langle S, S^{-1}\rangle$ is a quasi-isomorphism. But this is true since both sides are seen to compute the derived localisation $L_S^{\dgalg}(\ground \langle S \rangle)$.

Let $Q$ be the dg algebra
\[
Q = \ground \langle s=h^0_0, t=h^0_1, h^1_0, h^1_1, h^2_0, h^2_1, \dots \rangle
\]
with
\[\degree{h^n_m}=
\begin{cases}
n+\degree{s} &\text{if $n$ and $m$ are even}\\
n-\degree{s} &\text{if $n$ is even and $m$ is odd}\\
n &\text{if $n$ is odd}
\end{cases}
\]
and differential defined by $d(t)=d(s)=0$, $d(h^1_0) = st - 1$, $d(h^1_1) = ts - 1$ and for $n\geq 2$
\[
d(h^n_m) = \sum_{i=0}^{n-1}(-1)^{\degree{h^i_m}+1}h^i_mh^{n-1-i}_{m-1-i}
\]
where the lower index of an element $h^n_m$ is of course understood as an element in $\mathbb{Z}_2$.

\begin{proposition}\label{prop:cofibreplacement}
The dg algebra $Q$ is a cofibrant replacement for $\ground \langle s,s^{-1}\rangle $ as a dg $\ground \langle s \rangle$--algebra.
\end{proposition}

\begin{proof}
The dg algebra $Q$ is clearly cofibrant. Note that $Q$ is a bigraded complex, with $h^n_m$ in bidegree $(\degree{h^n_m},n)$ for $n=0,1,\dots$. Denote by $C$ the bicomplex defined by the exact sequence $0\to \ground \to Q \to C \to 0$ of chain bicomplexes. Then $C$ is the cobar construction of the \emph{non-unital} associative algebra $\ground \langle s, t \rangle / (s^2 = t^2 = 0)$. Since the ideal of relations is generated by monomials, this algebra is Koszul \cite{priddy:koszul} and so $H(C)$ is concentrated in bidegrees $(i,0), i\in \mathbb{Z}$ and it follows that $H(Q)$ is also concentrated in bidegrees $(i,0)$. Clearly, the space of cycles in $Q$ of bidegrees $(i,0), i\in\mathbb{Z}$ is just $\ground \langle s, t \rangle$ and the space of boundaries is the ideal $(st-1, ts-1)$, so $H(Q)\cong\ground \langle s,s^{-1}\rangle $.
\end{proof}

\begin{remark}
The dg algebra $Q$ is the cobar construction of a certain \emph{curved algebra}: the curved algebra $\tilde{\ground}\langle \sigma, \tau \rangle /(\sigma^2=\tau^2=0, \sigma\tau + \tau\sigma = \Omega)$, where here $\tilde{\ground}\langle \sigma, \tau \rangle$ is the free curved algebra on two generators, with curvature denoted by $\Omega$. \autoref{prop:cofibreplacement} therefore says that this is the Koszul dual of $\ground \langle s,s^{-1}\rangle $ (cf.~\cite[Section 6]{positselski:curvedkoszul} for Koszul duality of unital/curved algebras).
\end{remark}

\subsection{Homology of localisations}
In general, it appears that little can be said about the homology of the derived localisation of a dg algebra $A$. For example, it is not true that the homology of the derived localisation is necessarily isomorphic to the derived localisation of the homology. In fact, we have a spectral sequence relating the two, that does not always collapse as we show later on. Note, first of all, that if $B$ is a graded algebra with vanishing differential and $S\subset B$ then $H(L_S^{\dgalg}(B))\cong H(B*^{\mathbb{L}}_{\ground\langle S\rangle}\ground\langle S, S^{-1}\rangle)$ is bigraded (cf.~\autoref{cor:bigraded}) and so we could write $H_{ij}(L_S^{\dgalg}(B))$ for its bigraded component $(i,j)$.

\begin{theorem}\label{thm:spectralseq}
There exists a spectral sequence converging to $H_{i+j}(L_S^{\dgalg}(A))$ whose $E^2$~term is isomorphic to $H_{i,j}(L_S^{\dgalg}(H(A)))$, $i\in\mathbb{Z}, j\in\mathbb{N}$.
\end{theorem}

\begin{proof}
According to \cite[Theorem 3.4]{sagave:derivedainfinity}  there exists a $\ground$--projective (and hence proper) dg algebra resolution $\widehat{A}$ of $A$ that is $\mathbb N\times \mathbb Z$-graded and $E_2$-quasi-isomorphic to $A$. That means, in particular, that it is filtered: $0=F_{-1}(\widehat{A})\subset F_0(\widehat{A})\subset\ldots $ with $\bigcup_{i} F_i(\widehat{A})=\widehat{A}$ and that associated graded is also a $\ground$--projective resolution of $H(A)$. It follows that any homology class in $H(A)$ has a representative cocycle in $\widehat{A}$ having filtration degree 0. In particular, considering $\ground\langle S\rangle$ to be trivially filtered, the map $\ground\langle S\rangle\to \widehat{A}$ corresponding to choosing representatives of homology classes of $S\subset H(A)$, respects the filtration, i.e. its image lies in $F_0(\widehat{A})$.

 Now choose a cofibrant \emph{bigraded} resolution of $\ground\langle S, S^{-1}\rangle$ in the under category of $\ground\langle S\rangle$; denote it by $\widetilde{\ground \langle S, S^{-1}\rangle}$. Then $L_S^\dgalg (A)\simeq \widehat{A}*_{\ground\langle S\rangle}\widetilde{\ground \langle S, S^{-1}\rangle}:=\tilde{A}$. Moreover, since $\widetilde{\ground \langle S, S^{-1}\rangle}$ is obtained from $\ground\langle S\rangle$ by adjoining cells $\{x^{\alpha_n}\}$ (as in the proof of \autoref{prop:bigraded}), we have an isomorphism of $\ground\langle S\rangle$-modules (disregarding the differential):
 \[\tilde{A}=\widehat{A}*_{\ground\langle S\rangle}\widetilde{\ground \langle S, S^{-1}\rangle}\cong \widehat{A}\langle \{x^{\alpha_0}\},\{x^{\alpha_1}\},\ldots\rangle\]
Under this isomorphism the differential on $\widehat{A}$ is the given one while $d(x^{\alpha_n})\subset\widehat{A}\langle \{x^{\alpha_0}\},\ldots,\{x^{\alpha_{n-1}}\}\rangle$. Let us say that  $x^{\alpha_n}$ has weight $n$; this gives $\tilde{A}$ an increasing weight filtration $W_0\subset W_1\subset\ldots$ where $W_n$ is spanned by those monomials whose weight is $\leq n$.  Thus, $\tilde{A}$ has two filtrations: by weight and the one coming from $\hat{A}$. We consider the diagonal filtration and the corresponding spectral sequence.

Since the differential on every generator $x^{\alpha_n}$ lowers the weight degree, all these generators survive to the $E_1$ term. Remembering that $\widetilde{H(A)}$ is $\ground$-projective, we conclude that the $E_1$ term of the associated spectral sequence is isomorphic to
$\widehat{H(A)}\langle \{x^{\alpha_0}\},\{x^{\alpha_{1}}\},\ldots\rangle$ and it follows that the $E_2$-term is as claimed.
\end{proof}

When the derived localisation of $H(A)$ is the same as the non-derived localisation of $H(A)$, the situation simplifies.

\begin{theorem}\label{thm:homologycollapse}
Let $A$ be a dg algebra. If $L_S^{\dgalg}(H(A))\simeq H(A)[S^{-1}]\cong H(A)\ast_{\ground \langle S \rangle } \ground \langle S, S^{-1} \rangle$ then $H(L_S^{\dgalg}(A))\cong H(A)[S^{-1}]$.
\end{theorem}

\begin{proof}
Applying the spectral sequence in \autoref{thm:spectralseq} we see that its $E^2$~term vanishes for values of $j>0$ from which the conclusion follows.
\end{proof}

\begin{corollary}\label{prop:homologywhenfree}
If $H(A)$ is a cofibrant dg $\ground\langle S \rangle$--algebra, then $H(L_S^{\dgalg}(A)) \cong H(A)[S^{-1}]\simeq L_S^{\dgalg}(H(A))$.
\end{corollary}

\begin{proof}
Since $H(A)$ is cofibrant, by \autoref{cor:balancing} the derived localisation of the homology is $H(A)\ast_{\ground \langle S\rangle } \ground \langle S,S^{-1}\rangle\cong H(A)[S^{-1}]$.
\end{proof}

The condition that $H(A)$ is cofibrant over $\ground\langle S \rangle$ means that $H(A)$ is a free algebra, or a retract of such. Later on, we will also see that we can compute the homology of the localisation of $A$ at the other extreme, when $S$ satisfies some centrality condition (such as being an Ore set) in $H(A)$.

\section{Derived localisation of modules}
Let $A$ be a dg algebra and let $S\subset H(A)$ be an arbitrary subset of homogeneous homology classes.

\begin{remark}\label{rem:leftright}
In this section we will choose to work with localisations in the category of \emph{left} $A$--modules. Of course, the same theory holds, with obvious modifications, for \emph{right} $A$--modules.
\end{remark}

\begin{definition}\label{def:local}
A left $A$--module $M$ will be called \emph{$S$--local} if for all $s\in S$ the map $l_s\co H(M) \to H(M)$ given by $l_s(m)=sm$ is an isomorphism.
\end{definition}

A dg $A$--algebra is in particular an $A$--module. The notions of $S$--inverting and $S$--local are equivalent:

\begin{proposition}\label{prop:localinverting}
Let $Y$ be a dg $A$--algebra. Then $Y$ is $S$--inverting if and only if it is $S$--local.
\end{proposition}

\begin{proof}
Note that since $Y$ is a dg algebra, for any $s\in S$ the map $l_s$ is just left multiplication by the homology class of $s$ in $H(Y)$. Assume $Y$ is $S$--inverting. Then left multiplication by the multiplicative inverse of the homology class of $s$ is an inverse to the map $l_s$.

Conversely, assume $Y$ is $S$--local as a left $A$--module so that for any $s\in S$ the map $l_s$ has an inverse $f_s\co H(Y)\to H(Y)$. Then $s f_s(1) = (l_s\circ f_s)(1) = 1$ so $f_s(1)$ is a right multiplicative inverse for $s$. Since $l_s$ is a map of right $H(Y)$--modules so is its inverse $f_s$ and therefore $f_s(1)s = f_s(s) = (f_s\circ l_s)(1) = 1$ so $f_s(1)$ is also a left multiplicative inverse for $s$, hence $s$ is invertible in $H(Y)$.
\end{proof}

\begin{definition}\label{def:localisation}
A localisation of a module $M$ is an $S$--local $A$--module $Y$ with a map $f\co M\to Y$ such that for any $S$--local module $N$, the map $f^*\co \RHom_A(Y,N)\to \RHom_A(M,N)$ is a quasi-isomorphism.
\end{definition}

Two localisations are clearly quasi-isomorphic as $A$--modules.

\begin{proposition}\label{prop:twodefs}
An $A$--module $Y$ with a map $f\co M\to Y$ is a localisation of $M$ if and only it it is the initial $S$--local object in the homotopy category of $\undercat{M}{\amod}$.
\end{proposition}

\begin{proof}
Let $N$ be an arbitrary $S$--local $A$--module with a map $f'\co M \to N$. There is a homotopy pullback of simplicial sets
\[
\xymatrix{
\Map_{\undercat{M}{\amod}}(Y,N) \ar[r]\ar[d] & \ast\ar[d]\\
\Map_{\amod}(Y,N)\ar[r]^{f^*} & \Map_{\amod}(M,N)
}
\]
where the right vertical map is given by $\ast\mapsto f'$. But then $Y$ is a localisation of $M$ if and only if the bottom horizontal map is a quasi-isomorphism which holds if and only if $\Map_{\undercat{M}{\amod}}(Y,N)\simeq *$, which is true if and only if $Y$ is the initial $S$--local object in the homotopy category of $\undercat{M}{\amod}$ (since suspension/desuspension preserves the notion of being $S$--local).
\end{proof}

\begin{remark}
\autoref{prop:twodefs} should be compared to the definition of localisation for dg algebras. Note that left properness of $\amod$ is only required in order that we may consider the usual under category of $M$, instead of the derived under category.
\end{remark}

It follows that two localisations are not only quasi-isomorphic as $A$--modules, but also as $A$--modules \emph{under $M$}:

\begin{corollary}
Any two localisations of $M$ are isomorphic in the homotopy category of $\undercat{M}{\amod}$.\qed
\end{corollary}

\autoref{prop:twodefs} makes the following definition sensible.

\begin{definition}
We denote by $L_S^{\amod}(M)$ the initial $S$--local $A$--module in the homotopy category of $\undercat{M}{\amod}$. We will refer to it as \emph{the} localisation of $M$.
\end{definition}

\begin{remark}\mbox{}
\begin{itemize}
\item In light of \autoref{prop:twodefs} we will often abuse terminology and not distinguish between the different, but equivalent, notions of localisation of a module.
\item The reader may recognise that the localisation of a module is precisely a fibrant replacement in the Bousfield localisation of the model category of $A$--modules at the morphisms $\{r_s\co \Sigma^{\degree{s}}A\to A\}_{s\in S}$ given by right multiplications by the elements of $S$. Consequently, we can deduce existence of the localisation of a module from the general existence theorems for Bousfield localisation \cite[Chapter 4]{hirschhorn:localisation}.
\end{itemize}
\end{remark}

\begin{remark}\label{nonderivedmodulelocalisation}
Non-derived module localisations may be defined in analogy to non-derived algebra localisations, cf.~\autoref{nonderivedalgebralocalisation}. Let $A$ be a dg algebra and $S\subset Z(A)$ an arbitrary set of homogeneous cycles in $A$. An $A$--module is \emph{strictly $S$--local} if any $s\in S$ acts strictly invertibly on it. The \emph{non-derived localisation} of an $A$--module $M$ is the initial strictly $S$--local $A$--module in $\undercat{M}{\amod}$.
\end{remark}

\begin{convention}
Let $A'\acycfib A$ be a cofibrant replacement for $A$. Since restriction and extension of scalars along the quasi-isomorphism $A'\acycfib A$ induces a Quillen equivalence between the category of $A'$--modules and $A$--modules, then $L_S^{\amod}(A)$ is also the localisation of $A'$ as an $A'$--module. Therefore, without loss of any generality, we will from now on assume $A$ is always first cofibrantly replaced as a differential graded algebra, whenever it is necessary.
\end{convention}

The following theorem, essentially due to Dwyer \cite{dwyer:localisationinhomotopytheory}, says that the localisation of $A$ in the category of $A$--modules is, up to quasi-isomorphism, a dg algebra.

\begin{theorem}[{\cite[Proposition 2.5]{dwyer:localisationinhomotopytheory}}]\label{thm:dwyer}
There is a dg algebra $X\in \derunderalg{A}$ such that as $A$--modules $X\simeq L_S^{\amod}(A)$.
\end{theorem}

We give here a slightly different proof to Dwyer's, although the main idea is the same.

\begin{proof}
Denote by $S\otimes 1$ the subset of $H(A\otimes A^{\op})$ of elements of the form $s\otimes 1$ with $s\in S$ and let $Y = L_{S\otimes 1}^{\modcat{A\otimes A^{\op}}}(A)$ be the localisation of $A$ as an $A\otimes A^{\op}$--module (i.e.~an $A$--bimodule). Let $N$ be an $S$--local $A$--module. Then, observing that $\RHom_k(A,N)$ is $(S\otimes 1)$--local,
\begin{align*}
\RHom_A(Y,N) &\simeq \RHom_{A\otimes A^{\op}}(Y,\RHom_\ground(A,N))\\
& \simeq \RHom_{A\otimes A^{\op}}(A,\RHom_\ground(A,N))\simeq \RHom_A(A,N)
\end{align*}
and so $Y$ is also a localisation of $A$ as an $A$--module. Set $X$ to be the dg algebra $\REnd_{A^\op}(Y)$. Then we have a string of quasi-isomorphisms of $A$-modules: $X\simeq \RHom_{A^\op}(Y,Y)\simeq \RHom_{A^\op}(A,Y)\simeq Y$ as required.
\end{proof}

\begin{remark}
The ability to cofibrantly replace $A$ in the above proof was used where it was understood that $A\otimes A^{\op}$ computes the \emph{derived} tensor product and that a cofibrant $A\otimes A^{\op}$--module is also a cofibrant right $A$--module by restriction. If $A$ is not cofibrant, or at least left proper, we cannot conclude that there is a dg algebra map from $A$ itself into its module localisation.
\end{remark}

From now on, when we write $L_S^{\amod}(A)$ it will be assumed to be a dg algebra model for the localisation of $A$ as an $A$--module. We know there is at least one such model; soon we will show that it is unique in $\derunderalg{A}$. The following proposition says that this is a \emph{smashing} localisation.

\begin{proposition}\label{prop:smashing}
Let $M$ be an $A$--module. Then $M\simeq A\otimes_A^{\mathbb{L}} M \to L_S^{\amod}(A)\otimes_A^{\mathbb{L}} M$ is the localisation $L_S^{\amod}(M)$ of $M$.
\end{proposition}

\begin{proof}
Let $N$ be an $S$--local $A$--module. Since
\begin{align*}
\RHom_A(L_S^{\amod}(A)\otimes_A^{\mathbb{L}} M, N) &\simeq \RHom_A(M,\RHom_A(L_S^{\amod}(A),N))\\&\simeq \RHom_A(M,\RHom_A(A,N)) \simeq \RHom_A(M,N)
\end{align*}
it is sufficient to prove that $L_S^{\amod}(A)\otimes_A^{\mathbb{L}} M$ is $S$--local. Left multiplication by any cycle representing a homology class $s\in S$ gives a (graded) quasi-isomorphism $l_s\co L_S^{\amod}(A)\to L_S^{\amod}(A)$ of right $A$--modules and the derived tensor product with $M$ preserves quasi-isomorphisms, so $L_S^{\amod}(A)\otimes_A^{\mathbb{L}} M$ is indeed $S$--local.
\end{proof}

\begin{corollary}\label{cor:adj}
The Quillen adjunction $\amod\rightleftarrows \modcat{L_S^{\amod}(A)}$ with left adjoint given by extension of scalars $M\mapsto L_S^{\amod}(A)\otimes_A M$ and right adjoint given by restriction along $A\to L_S^{\amod}(A)$ induces an equivalence between $\Ho(\modcat{L_S^{\amod}(A)})$ and the full subcategory of $\Ho(\amod)$ of $S$--local modules.\qed
\end{corollary}

\begin{proof}
In the previous proposition we saw that $L_S^{\amod}(A)\otimes_A^{\mathbb{L}} M$ is the localisation of any $A$--module $M$ and so if $M$ is an $S$--local $A$--module then $M\to L_S^{\amod}(A)\otimes_A^{\mathbb{L}} M$ is a quasi-isomorphism. Moreover, an $L_S^{\amod}(A)$--module is indeed an $S$--local $A$--module since its homology is an $H(L_S^{\amod}(A))$--module and any $s\in S\subset H(L_S^{\amod}(A))$ has an inverse. Then the proposition above implies that $L_S^{\amod}(A)\otimes_A^{\mathbb{L}} M\to M$ is a quasi-isomorphism (it has a right inverse which is a quasi-isomorphism) for any $L_S^{\amod}(A)$--module $M$.
\end{proof}

\begin{remark}
The corollary above identifies $\modcat{L_S^{\amod}(A)}$ as (Quillen equivalent to) the Bousfield localisation of $\amod$.
\end{remark}

We can in fact elevate \autoref{thm:dwyer} further to show that not only is the localisation of $A$ in $A$--modules a dg algebra up to quasi-isomorphism, but the universal map from this localisation to an $S$--inverting dg algebra, which is a priori only a map of modules, is up to homotopy a map of dg algebras, in a precise sense.

\begin{lemma}\label{lem:dwyermap}
Let $C\in \derunderalg{A}$ be $S$--inverting. Then there exists a map $L_S^{\amod}(A)\to C$ in $\derunderalg{A}$.
\end{lemma}

\begin{proof}
We may assume $A\cof L_S^{\amod}(A)$ is a cofibrant, in fact cell, dg $A$--algebra by replacing if necessary. Since $A$ is assumed to be cofibrant, then we may also assume that $L_S^{\amod}(A)$ is also cofibrant as a right $A$--module. To see this note that it can be seen directly, using the fact that $A$ is a cofibrant $k$--module, that a cell dg $A$--algebra can be built as a right $A$--module by attaching cells in $\modcat{A^\op}$; indeed the $k$--cofibrancy of $A$ means that the underlying right $A$--module of a cell dg $A$--algebra is free.

Let $C$ be an $S$--inverting dg $A$--algebra which again, without loss of generality, we assume is cofibrant.

The map $A\to L_S^{\amod}(A)$ is a map of $(A,A)$--bimodules and so consider the map $C\to L_S^{\amod}(A)\otimes_{A} C$ of $(A,C)$--bimodules given by tensoring on the right with the $(A,C)$--bimodule $C$. It is a quasi-isomorphism since $L_s^{\amod}(A)$ is cofibrant as a right $A$--module and so by \autoref{prop:smashing} $L_S^{\amod}(A)\otimes_{A}C$ is the localisation of $C$, yet $C$ is already $S$--local. Now set $C'=\REnd_C(L_S^{\amod}(A)\otimes_{A} C)$. By \autoref{thm:endhomotopy} $C'$ is a dg $L_S^{\amod}(A)$--algebra. Moreover, $C'$ is a dg $A$--algebra by restricting along the map $A\to L_S^{\amod}(A)$ and since $C\simeq L_S^{\amod}(A)\otimes_{A} C$ as $(A,C)$--bimodules then, again by \autoref{thm:endhomotopy}, $C\simeq\REnd_C(C)\simeq C'$ as dg $A$--algebras.
\end{proof}

Note that this lemma does not show that this map is unique in the derived under category of $A$, although it is of course unique in the homotopy category of left $A$--modules under $A$. It turns out that this map is indeed unique, and hence the localisation of $A$ in $A$--modules is also the localisation of $A$ as a dg algebra, but this is a more difficult fact to show, the proof of which is essentially the purpose of the remainder of this section.

A version of the following theorem was proved in the case of an ungraded ring in \cite{dwyer:localisationinhomotopytheory}.

\begin{theorem}\label{lem:selftensor}
Let $A$ be a graded algebra with zero differential. The following are equivalent:
\begin{enumerate}
\item $A[S^{-1}]\otimes_A^{\mathbb{L}} A[S^{-1}]\we A[S^{-1}]\otimes_A A[S^{-1}]$
\item $A[S^{-1}]\otimes_A^{\mathbb{L}} A[S^{-1}]\we A[S^{-1}]$
\item $L_S^{\amod}(A)\we A[S^{-1}].$
\end{enumerate}
Here, these maps are maps of (left) $A$--modules. The first and third maps are induced by the map from a cofibrant replacement of $A$ to $A$, and the second one is the composition of the first and the canonical isomorphism $A[S^{-1}]\otimes_A A[S^{-1}]\cong A[S^{-1}]$.
\end{theorem}

\begin{proof}
Since $A[S^{-1}]$ is the non-derived localisation, by the non-derived analogue of \autoref{prop:smashing} it always holds that $A[S^{-1}]\otimes_A A[S^{-1}]\cong A[S^{-1}]$. Therefore (1) and (2) are equivalent. By \autoref{prop:smashing}, (3) clearly implies (2).

To see (2) implies (3), let $M$ be a right $A[S^{-1}]$--module. Then
\[M\otimes_A^{\mathbb{L}}A[S^{-1}]\simeq \left (M\otimes_{A[S^{-1}]}^{\mathbb{L}}A[S^{-1}] \right )\otimes_A^{\mathbb{L}}A[S^{-1}] \simeq M \otimes_{A[S^{-1}]}^{\mathbb{L}}A[S^{-1}] \simeq M.\]
By \autoref{prop:smashing}, $L_S^{\amod}(A) \otimes_A^{\mathbb{L}} A[S^{-1}] \simeq A[S^{-1}]$. Consider now the K\"unneth spectral sequence associated to this derived tensor product: the $E^2$~term is given by the homology of $H(L_S^{\amod}(A))\otimes_A^{\mathbb{L}} A[S^{-1}]$. Observing that $H(L_S^{\amod}(A))$ is a right $A[S^{-1}]$--module we see that the spectral sequence collapses and the homology is $H(L_S^{\amod}(A))$ from which the result follows.
\end{proof}

\begin{remark}\label{rem:homological_epimorphism}
Recall that a dg algebra map $A\to B$ is a \emph{homological epimorphism} if and only if the map $B\otimes_A^{\mathbb{L}} B\to B$, induced by the multiplication on $B$, is a quasi-isomorphism. If we only have $B\otimes_A^{\mathbb{L}} B\to B\otimes_AB$ a quasi-isomorphism then the map $A\to B$ is called \emph{stably flat}; the latter notion is mostly used in the non-differential case. Then \autoref{lem:selftensor} can be understood as saying that if $A$ is a graded algebra with zero differential then for the map $A\to A[S^{-1}]$ the notions of homological epimorphism and stable flatness coincide and moreover $A\to A[S^{-1}]$ is the localisation of $A$ as an $A$--module if and only if it is a homological epimorphism/stably flat.
\end{remark}

Suppose now that $A$ is a graded algebra with vanishing differential that is \emph{hereditary}, i.e.~whose global dimension is $0$ or $1$. For example, such is a free algebra over a field. Then the conditions of \autoref{lem:selftensor} are always satisfied:

\begin{corollary}\label{cor:schofield}
For a hereditary graded algebra $A$ the map $A\to A[S^{-1}]$ is a homological epimorphism. Therefore, there is a natural equivalence
of dg algebras $L_S^{\amod}(A)\simeq A[S^{-1}]$ in $\derunderalg{A}$.
\end{corollary}

\begin{proof}
We need to show that $\Tor^{A}_i(A[S^{-1}],A[S^{-1}])=0$ for $i=1,2,\dots$. The vanishing of $\Tor^{A}_1(A[S^{-1}],A[S^{-1}])$ is a general result on non-derived localisation, cf.~\cite[Theorem 4.7, Theorem 4.8]{Schofield}; the proof in op.~cit is given for ungraded rings but applies verbatim for graded ones. Since $A$ is hereditary, all higher torsion products automatically vanish.
\end{proof}

\subsection{Localisations of free algebras}
We will now compute the localisations of free algebras. In particular, we will prove \autoref{lem:simplelocalisation}. Recall that we denote the
non-derived localisation by $A[S^{-1}]$, which is given explicitly by $A\ast_{\ground \langle S \rangle} \ground \langle S, S^{-1} \rangle$.
We start by proving that dg modules over cell algebras have a particularly simple cofibrant replacement. Let $A$ be a cell dg
$\ground$-algebra; that means that \begin{itemize}
\item forgetting the differential, $A$ is a free graded algebra $\ground\langle T\rangle$ on a set $T$ of homogeneous generators;
\item the set $T$ is a union of a nested system
$\emptyset=T_0 \subset\ldots \subset T_n \subset \ldots$, and $d(T_n)\subset k\langle T_{n-1}\rangle$.
\end{itemize}
Note that for every $n$ the set $T_n$ is not necessarily countable and so $T$ can likewise be uncountable. Denoting by $A_n$ the subalgebra in $A$ generated by $T_n$, we see that $A_n$ is closed with respect to the differential and that $A=\bigcup_{n=0}^\infty A_n$. The dg algebra $A_n$ plays the role of an $n$-skeleton of $A$ (even though, of course, it contains cells of arbitrary dimension).
\begin{lemma}\label{lem:2-res}
Let $A=\ground\langle T\rangle$ be a cell dg algebra as above and $M$ be a dg $A$-module that is cofibrant as a $\ground$-module. Then $M$ has a cofibrant replacement as the totalisation of the double complex with two columns having the following form; here $\langle T\rangle$ stands for the free $k$--module generated by $T$.
\begin{equation}\label{two_term_res}
\ground\langle T\rangle\otimes\langle T\rangle\otimes M \to \ground\langle T\rangle\otimes M.
\end{equation}
\end{lemma}
\begin{proof}
Consider the standard two-term free resolution of  $\ground\langle T\rangle$ as a bimodule over itself:
\begin{equation}\label{res_free}
\ground\langle T\rangle\otimes\langle T\rangle\otimes \ground\langle T\rangle\to \ground\langle T\rangle\otimes \ground\langle T\rangle\to \ground\langle T\rangle.
\end{equation}
Here the map $\ground\langle T\rangle\otimes \ground\langle T\rangle\to \ground\langle T\rangle$ is the multiplication map and the other arrow is the map that takes $t\in T$ to $t\otimes 1-1\otimes t\in \ground\langle T\rangle\otimes \ground\langle T\rangle$. Remembering that $A=\ground\langle T\rangle$ has its own differential, we can view this resolution as a double complex with three columns. The vertical differential on the $A$-bimodule $\ground\langle T\rangle\otimes\langle T\rangle\otimes \ground\langle T\rangle$ is the one  making this bimodule the kernel of the multiplication map of $A$. Consider a generator
$1\otimes t \otimes 1\in \ground\langle T\rangle\otimes\langle T_n \rangle\otimes \ground\langle T\rangle$.
 Because of the cellular nature of the differential in $\ground\langle T\rangle$ it holds, inside $\ground\langle T\rangle\otimes \ground\langle T\rangle$:
\[
d(t\otimes 1-1\otimes t)\subset \ground\langle T_{n-1}\rangle \otimes \ground\langle T_{n-1}\rangle.
\]
Therefore, it holds, inside $\ground\langle T\rangle\otimes\langle T\rangle\otimes \ground\langle T\rangle$:
\[d(1\otimes t \otimes 1)\subset\ground\langle T_{n-1}\rangle\otimes \langle T_{n-1}\rangle \otimes \ground\langle T_{n-1}\rangle \subset\ground\langle T \rangle\otimes \langle T_{n-1}\rangle \otimes \ground\langle T \rangle   .\]
Thus, the dg $A$-bimodule $\ground\langle T\rangle\otimes\langle T\rangle\otimes \ground\langle T\rangle$ is cellular and so is the totalisation of the double complex $\ground\langle T\rangle\otimes\langle T\rangle\otimes \ground\langle T\rangle\to \ground\langle T\rangle\otimes \ground\langle T\rangle$.

Tensoring \autoref{res_free} with $M$ we obtain the following double complex:
\begin{equation}
\ground\langle T\rangle\otimes\langle T\rangle\otimes M\to \ground\langle T\rangle\otimes M\to M.
\end{equation}
The horizontal differential in it is clearly exact since its homology in degrees $i>0$ is simply $\Tor^{A}_i(A, M)=0$. Therefore the total complex
of $\ground\langle T\rangle\otimes\langle T\rangle\otimes M\to \ground\langle T\rangle\otimes M$ is quasi-isomorphic to $M$; it is also clearly a cofibrant (left) $A$-module as a tensor product over $A$ of a cofibrant $A$-bimodule and a cofibrant $\ground$-module $M$.
\end{proof}

\begin{proposition}\label{prop:freeloc}
Let $S\subset T$. Then $L_S^{\modcat{\ground \langle T \rangle}}(\ground\langle T\rangle) \simeq \ground \langle T \rangle [S^{-1}] \cong \ground \langle T, S^{-1} \rangle$.
\end{proposition}

\begin{proof}
We will use \autoref{lem:selftensor}, applied to $A=\ground \langle T \rangle$. It follows from (the graded version of) \cite[Theorem 4.7, Theorem 4.8]{Schofield} that $\Tor_1^A(A[S^{-1}],A[S^{-1}])\cong 0$. By \autoref{lem:2-res}, the $A$--module $A[S^{-1}]\cong \ground \langle T, S^{-1}\rangle$, being $\ground$-free, thus $\ground$-cofibrant, admits a two term free resolution over $\ground \langle T \rangle$. Therefore, there are no higher $\Tor$ terms.
\end{proof}

\begin{lemma}\label{lem:selftensorgeneral}
Let $\ground\langle S \rangle \cof A$ be a relative cell dg algebra (so that the underlying graded algebra of $A$ is free over $\ground \langle S \rangle$). Then the map induced by the multiplication map $A[S^{-1}]\otimes^{\mathbb{L}}_A A[S^{-1}]\to A[S^{-1}]$ is a quasi-isomorphism.
\end{lemma}

\begin{proof} Consider the cofibrant standard replacement of the $A$-module $A[S^{-1}]$ constructed in \autoref{lem:2-res}:
\[
\ground\langle T\rangle\otimes\langle T\rangle\otimes\ground\langle T\rangle[S^{-1}]\to \ground\langle T\rangle\otimes\ground\langle T\rangle[S^{-1}]
\]
Tensoring it on the left with $k\langle T \rangle [S^{-1}]$ over $A$ we obtain the following (double) complex $C$ computing $A[S^{-1}]\otimes^{\mathbb{L}}_A A[S^{-1}]$:
\[
C:\ground\langle T\rangle[S^{-1}]\otimes\langle T\rangle\otimes\ground\langle T\rangle[S^{-1}]\to \ground\langle T\rangle[S^{-1}]\otimes\ground\langle T\rangle[S^{-1}]
\]
(formally $C$ is the cofibre of the above map).
Let us disregard the differential in $A$ for a moment; in other words, assume that $C$ has vanishing vertical differential. Then
it follows from \autoref{prop:freeloc} and \autoref{lem:selftensor} that it is quasi-isomorphic to $A[S^{-1}]\otimes_A A[S^{-1}]\cong A[S^{-1}]$. In other words, the following is a short exact sequence:
\begin{equation}\label{disregard}
0\to \ground\langle T\rangle[S^{-1}]\otimes\langle T\rangle\otimes\ground\langle T\rangle[S^{-1}]\to \ground\langle T\rangle[S^{-1}]\otimes\ground\langle T\rangle[S^{-1}]\to A[S^{-1}]\to 0.
\end{equation}
Remembering that $A$ has a nontrivial differential, we still have \ref{disregard}, except viewed as a short exact sequence of complexes. It follows that $C$ is quasi-isomorphic to $A[S^{-1}]$ as required.
\end{proof}

\begin{lemma}\label{lem:splitting}
Let $\ground\langle S \rangle \cof A$ be a relative cell dg algebra which is $S$--inverting. Then the localisation map $A\to A[S^{-1}]$ admits a splitting in the homotopy category of dg $A$--algebras.
\end{lemma}

\begin{proof}
Since $A$ is cofibrant under $\ground \langle S \rangle$, then $A[S^{-1}] \cong A\ast_{\ground \langle S \rangle } \ground \langle S , S^{-1} \rangle \simeq A\ast^{\mathbb{L}}_{\ground \langle S \rangle } \ground \langle S , S^{-1} \rangle$ as dg $A$--algebras. It is sufficient to construct a map in the homotopy category of dg $\ground \langle S \rangle$--algebras from $\ground \langle S, S^{-1} \rangle$ to $A$ since taking the derived free product over $\ground \langle S \rangle$ of this map with the identity map on $A$ will yield the desired splitting.

But since $\ground \langle S, S^{-1} \rangle\simeq L_S^{\modcat{\ground \langle S \rangle}}(\ground \langle S \rangle)$ (\autoref{prop:freeloc}) such a map exists by \autoref{lem:dwyermap}.
\end{proof}

We will use the following lemma which characterises the localisation of a dg algebra in terms of derived free products with $S$--inverting dg algebras.

\begin{lemma}\label{lem:freeprodsufficient}
Let $B\in\derunderalg{A}$ be an $S$--inverting dg algebra such that for any other $S$--inverting dg algebra $C\in\derunderalg{A}$ it holds that $C\to C\ast_A^{\mathbb{L}} B$ is a quasi-isomorphism in $\derunderalg{A}$. Then $B\simeq L_S^{\dgalg}(A)$, in other words $B$ is the localisation of $A$.
\end{lemma}

\begin{proof}
Since $B$ itself is $S$--inverting, its localisation as a dg algebra is itself and so $B\simeq B\ast_A^{\mathbb{L}} B$ by \autoref{lem:pushoutlemma}. Let $C\in\derunderalg{A}$ be an $S$--inverting dg algebra. Using the adjunction $\derunderalg{A}\rightleftarrows \derunderalg{B}$ we have:
\begin{align*}
\Hom_{\derunderalg{B}}(B, B\ast_A^{\mathbb{L}} C)&\cong \Hom_{\derunderalg{B}}(B\ast_A^{\mathbb{L}} B, B\ast_A^{\mathbb{L}} C)\\
&\cong \Hom_{\derunderalg{A}}(B, B\ast_A^{\mathbb{L}} C) \cong \Hom_{\derunderalg{A}}(B,C)
\end{align*}
But now the result follows since $B$ is initial in $\derunderalg{B}$.
\end{proof}

\begin{theorem*}[\autoref{lem:simplelocalisation}]
The localisation of $\ground \langle S \rangle$ is $L_S^\dgalg(\ground \langle S \rangle)\simeq \ground \langle S, S^{-1} \rangle$.
\end{theorem*}

\begin{proof}
Let $\ground\langle S \rangle \cof A$ be a relative cell dg algebra which is $S$--inverting. By \autoref{lem:freeprodsufficient}, it is sufficient to prove that $ A\ast_{\ground \langle S \rangle }^{\mathbb{L}} \ground \langle S, S^{-1} \rangle$ is quasi-isomorphic to $A$. Note that by \autoref{thm:flatleftproper},  $\ground\langle S\rangle$ and $\ground\langle S, S^{-1}\rangle$ are left proper and since $A$ is cofibrant, by \autoref{cor:balancing}$, A\ast_{\ground \langle S \rangle }^{\mathbb{L}} \ground \langle S, S^{-1} \rangle$  is quasi-isomorphic to $A[S^{-1}]$. By \autoref{lem:splitting} above, which says $A\to A[S^{-1}]$ admits a splitting in the homotopy category of dg $A$--algebras and hence in particular in the homotopy category of $A$--bimodules, we have a quasi-isomorphism of $A$--bimodules $A[S^{-1}]\simeq A\oplus M$ for some $A$--bimodule $M$, so it suffices to prove that $M$ is quasi-isomorphic to zero. By \autoref{lem:selftensorgeneral}, the multiplication map $A[S^{-1}]\otimes_A^{\mathbb{L}} A[S^{-1}] \simeq (A\oplus M \oplus M \oplus M\otimes_A^{\mathbb{L}} M) \to A \oplus M$ is a quasi-isomorphism. This map is defined by the identity on $M$ on the two copies of $M$ in the source and so can only be a quasi-isomorphism if $M\simeq 0$.
\end{proof}

\subsection{Comparison of localisations}

The following central theorem, upon which most of our results rely, shows that $L_S^{\amod}(A)$ is in fact also the localisation of $A$ as a dg algebra. One corollary is that $L_S^{\amod}(A)$ is a well-defined object in $\derunderalg{A}$ (as opposed to just well-defined up to quasi-isomorphism of modules). More importantly, it implies that the localisation of $A$ as an $A$--module and the localisation of $A$ as a dg algebra are really equivalent in a precise sense.

\begin{theorem}\label{thm:localisationscoincide}
If $L_S^{\amod}(A)$ is a dg algebra in $\derunderalg{A}$ which is the localisation of $A$ as an $A$--module then it is also the localisation of $A$ as a dg algebra.
\end{theorem}

\begin{proof}
By \autoref{lem:dwyermap}, there is a map $L_S^{\amod}(A)\to L_S^{\dgalg}(A)$ in $\derunderalg{A}$. Then by the universal property of the dg algebra localisation, the composition of maps $L_S^{\dgalg}(A) \to L_S^{\amod}(A)\to L_S^{\dgalg}(A)$ in $\derunderalg{A}$ must therefore be the identity. But similarly, the composition of maps $L_S^{\amod}(A)\to L_S^{\dgalg}(A) \to L_S^{\amod}(A)$ must be homotopic (as a map of modules) to the identity. Therefore, the dg algebra map $L_S^{\amod}(A)\to L_S^{\dgalg}(A)$ is a quasi-isomorphism.
\end{proof}

The main computational use of this theorem is the following corollary that allows us to understand the homotopy type of the localisation of $A$ as a dg algebra by instead calculating the homotopy type of the localisation of $A$ as an $A$--module, which is often easier to understand, being more amenable to elementary techniques from homological algebra.

\begin{corollary}
Let $Y$ be a localisation of $A$ as an $A$--module. Then as $A$--modules $L_S^{\dgalg}(A)\simeq Y$.\qed
\end{corollary}

\begin{remark}
Although $L_S^\dgalg(A)$ is, a priori, just an $A'$--module for some $A'\acycfib A$ a cofibrant replacement for $A$, the categories of modules of $A$ and $A'$ are of course Quillen equivalent (by restriction and extension of scalars) so it makes sense to say that the equivalence in the corollary above is an equivalence of $A$--modules.
\end{remark}

Another important consequence (which also takes into account \autoref{cor:adj}) is that the derived category of local modules is equivalent to the derived category of $L_S^{\dgalg}(A)$--modules:

\begin{corollary}\label{cor:localcat}
The full subcategory in $\Ho(\amod)$ consisting of $S$--local $A$--modules is equivalent to $\Ho(\modcat{L_S^\dgalg(A)})$.\qed
\end{corollary}

\begin{remark}
All told, \autoref{thm:localisationscoincide} can be regarded as showing that the Bousfield localisation of $\amod$ at the maps $\{r_s\co \Sigma^{\degree{s}}A\to A \}_{s\in S}$ is Quillen equivalent to $\modcat{L_S^{\dgalg}(A)}$.
\end{remark}

\begin{convention}
From now on, where it is not ambiguous, we will omit the notational difference between $L_S^{\amod}(A)$ and $L_S^\dgalg(A)$ and just refer instead to the dg algebra $L_S(A)$, which we regard as an object in both $\Ho(\amod)$ and $\derunderalg{A}$ and which we call simply \emph{the} localisation of $A$.
\end{convention}

\subsection{Derived matrix localisation}
The main results of our paper can be further generalised to include (the derived version of) the universal Cohn localisation \cite{Cohn, Schofield}. Here we outline how this is done; the proofs will be omitted since they are very similar to the ones given for the ordinary derived localisation and do not involve any new ideas. Let $S=\{S_{\alpha}\}\co L_{\alpha}\to N_{\alpha}$ be a collection of maps in the homotopy category of $A$--modules.

\begin{definition}\label{def:localmod}
An $A$--module $M$ is called \emph{$S$--local} if all induced maps
\[
\RHom_A(N_{\alpha},M)\to\RHom_A(L_{\alpha},M)
\]
are quasi-isomorphisms.
\end{definition}

\begin{remark}
Any collection of elements $S\in H(A)$ can be viewed as a collection of homotopy classes of maps of $A$--modules $A\to A$ and being $S$--local in the sense of \autoref{def:local} is the same as being $S$--local in the sense of \autoref{def:localmod}.
\end{remark}

We define a \emph{localisation} $L_S^{\amod}(M)$ of an $A$--module $M$ precisely as in \autoref{def:localisation}. Then \autoref{thm:dwyer} makes sense and its proof extends with obvious modifications to the present context. It is \emph{not true}, in general, that the functor $M\to L_S^{\amod}(M)$ is smashing, i.e.~that $L_S^{\amod}(M)\simeq L_S^{\amod}(A)\otimes^{\mathbb L}_AM$. It is, however, true if the homotopy cofibres of all the maps $S_{\alpha}$ are perfect $A$--modules, cf.~\cite[Proposition 2.10]{dwyer:localisationinhomotopytheory}.

The following is an analogue of \autoref{def:S-inverting}.

\begin{definition}
A dg algebra $f\co A\to Y$ is called \emph{$S$--inverting} if the induced maps
\[
f_*(S_{\alpha})\co L_{\alpha}\otimes^{\mathbb L}_AY\to N_{\alpha}\otimes^{\mathbb L}_AY
\]
are quasi-isomorphisms.
\end{definition}

This gives rise to the general notion of $S$--localisation (\autoref{def:alglocalisation}). In other words, the functor $A\mapsto L_S^{\dgalg}(A)$ is the initial $S$--inverting dg algebra in $\derunderalg{A}$.

\begin{definition}
Let $S=\{S_{\alpha}\co N_{\alpha}\to L_{\alpha}\}$ be a collection of homotopy classes of maps between perfect $A$--modules. Then the $L_S^{\dgalg}(A)\in\derunderalg{A}$ is called the derived localisation of $A$ with respect to $S$.
\end{definition}

Let us now outline the existence and an explicit form of the derived Cohn localisation, restricting ourselves to the case when the modules $L_{\alpha}$ and $N_{\alpha}$ are $A$--modules which are free of finite rank $l$ and $n$ respectively. We will in that case refer to it as (derived) \emph{matrix localisation}.   We further simplify by assuming that $S$ consists of only one map $L\to N$; the general case being a simple iteration of this construction.

Since the $A$--modules $L$ and $N$ are free, a homotopy class of maps $N\to L$ corresponds to an $n\times l$ matrix $(s_{ij})$ with entries in $H(A)$. A choice of representatives of homology classes then determines a map $\ground\langle(s_{ij})\rangle\to A$ from the free $\ground$--algebra on the symbols $s_{ij}$ to $A$; this is a map of dg algebras where $\ground\langle(s_{ij})\rangle$ is given the trivial differential.

Denote by $\ground\langle (s_{ij}),(s_{ij})^{-1}\rangle$ the algebra obtained from $\ground\langle(s_{ij})\rangle$ by formally inverting the matrix $ (s_{ij})$. In other words, we introduce an $l\times n$ matrix worth of symbols $s^\prime_{ij}$; then $\ground\langle (s_{ij}),(s_{ij})^{-1}\rangle$ has generators $\{s_{ij}\}$ and $\{s^{\prime}_{ij}\}$ and relations written in a matrix form as $(s_{ij})\cdot(s^{\prime}_{ij})=I_{l}$ and $(s^{\prime}_{ij})\cdot(s_{ij})=I_{n}$ where $I_l$ and $I_n$ are the identity matrices of sizes $l$ and $n$ respectively.

We have the following result whose proof is similar to that of \autoref{cor:alglocalisation}. As in the case of the ordinary derived localisation, it reduces to computing the derived localisation of the free algebra.

\begin{theorem}
The derived matrix localisation $L_S^{\dgalg}(A)$ of $A$ is given by $A*^{\mathbb{L}}_{\ground\langle(s_{ij})\rangle}\ground\langle (s_{ij}),(s_{ij})^{-1}\rangle$.
\end{theorem}

\begin{example} Let $\ground$ be a field and consider $A=\ground\langle X\rangle$, the free algebra over $\ground$ on some set of generators $X$. Then for $A$ there exists its universal (skew)-field of fractions $\ground(X)$, the so-called \emph{free field} on $X$ over $\ground$, \cite{Cohn}. The skew-field $\ground(X)$ is obtained from $\ground\langle X\rangle$ by inverting all square full matrices (i.e.~such that the endomorphisms of free modules they represent do not factor through free modules of lower rank). Thus, $\ground(X)$ is a derived, as well as a non-derived, localisation of $\ground\langle X\rangle$.
\end{example}

We have an analogue of \autoref{thm:localisationscoincide} with the same proof:

\begin{theorem}
The derived Cohn localisation $L_S^{\dgalg}(A)$ of $A$ as a dg algebra coincides in the homotopy category of $A$--modules with $L_S^{\amod}(A)$, the $S$--localisation of $A$ as an $A$--module.
\end{theorem}

From now on we will not distinguish between $L_S^{\dgalg}(A)$ and $L_S^{\amod}(A)$ and use the notation $L_S(A)$ for both. As in the case of localising at a set of homology classes in $H(A)$, we have a localisation map $A\to L_S(A)$, well-defined in the homotopy category of dg algebras (which means that $A$ may have to be cofibrantly replaced for this map to exists on the nose).

Cohn localisation is closely related to the notion of a homological epimorphism, cf.~\autoref{rem:homological_epimorphism}. It is not hard to see that a map $A\to B$ is a homological epimorphism if and only if the functor $M\mapsto B\otimes^{\mathbb{L}}_AM$ is a smashing localisation on the derived category of (left) $A$--modules; see \cite[Theorem 3.9 (6)]{Pauksztello} for an equivalent statement.

On the other hand we know that derived Cohn localisation is also smashing. It is proved in \cite[Theorem 6.1]{Krause_Stovicek} that any homological epimorphism from a hereditary ring is actually a Cohn localisation. Because of that, one may be led to conjecture that a homological epimorphism and \emph{derived} Cohn localisation are equivalent notions for any ring, not necessarily hereditary, as well as for any dg algebra. This turns out not to be true, but the difference is measured by the failure of the so-called \emph{telescope conjecture}, originally formulated in the context of stable homotopy theory \cite{Ravenel}. Abstractly, it claims that any smashing localisation of a triangulated category is \emph{finite}, i.e.~its kernel is generated by compact (or perfect) objects. Then we have the following result.

\begin{theorem}
Let $A$ be a dg algebra and $A\to L_S(A)$ be the derived Cohn localisation with respect to some set of maps $S\co L_{\alpha}\to N_{\alpha}$ between perfect $A$--modules. Then it is a homological epimorphism. Conversely, any homological epimorphism $A\to B$ is a derived Cohn localisation precisely when the telescope conjecture holds in $\Ho(\amod)$, the derived category of $A$.
\end{theorem}

\begin{proof}
The first statement is simply a restatement of the fact that a derived Cohn localisation is smashing.

Next, any Cohn localisation functor $M\to L_S(M)$ is finite, because its kernel is generated by homotopy kernels between perfect $A$--modules and thus, are themselves perfect, i.e.~compact. Conversely, any finite localisation functor $F$ of the derived category $\Ho(\amod)$ is a suitable Cohn localisation. Indeed, the kernel of $F$ is a triangulated subcategory in $\Ho(\amod)$ generated by some compact objects $A_{\alpha}$ and then $F$ is the derived Cohn localisation with respect to the collection of maps $A_\alpha\to 0$.
The second claim follows.
\end{proof}

Note that if $A$ is a hereditary ring then the telescope conjecture does hold in $\Ho(\amod)$, as proved in \cite{Krause_Stovicek}.

\section{Computing localisations}
Combining \autoref{thm:localisationscoincide}, \autoref{thm:homologycollapse} and \autoref{lem:selftensor} we obtain the following answer to the question of when the homology of the derived localisation is the non-derived localisation of the homology.

\begin{theorem}\label{thm:stableflat}
Let $A$ be a dg algebra and let $S\subset H(A)$ be an arbitrary subset. Then $H(A)[S^{-1}]\simeq L_S(H(A))$ if and only if $H(A)\to H(A)[S^{-1}]$ is a homological epimorphism/stably flat. In this case, $H(L_S(A))\cong H(A)[S^{-1}]$.\qed
\end{theorem}

Stable flatness is not always so easy to verify. But flatness of course implies stable flatness:

\begin{corollary}\label{thm:dgflat}
Let $A$ be a dg algebra. If $H(A) [S^{-1}] $ is flat as either a left or right $H(A)$--module then $H(L_S(A))\cong H(A)[S^{-1}]\simeq L_S(H(A))$.
\qed
\end{corollary}

In the case that every $s\in S$ is in the centre of $H(A)$ we now see that the homology of the derived localisation of $A$ at $S$ recovers the familiar localisation from commutative algebra.

\begin{theorem}\label{cor:central}
Let $A$ be a dg algebra. If $S$ is a central subset in $H(A)$ then there is an isomorphism of graded algebras $H(A \ast_{\ground \langle S \rangle}^{\mathbb{L}}\ground \langle S, S^{-1} \rangle )\cong H(A)[S^{-1}]\cong H(A) \otimes_{\ground [ S ]} \ground [ S, S^{-1}] $.
\end{theorem}

\begin{proof}
Since $S$ is central in $H(A)$, then $H(A)\ast_{\ground \langle S \rangle } \ground \langle S, S^{-1} \rangle \cong H(A)\otimes_{\ground[S]} \ground [S,S^{-1}]$. Let $P\acycfib M$ be a $H(A)$--module cofibrant replacement for a right $H(A)$--module $M$. Then
\begin{align*}
M\otimes^{\mathbb{L}}_{H(A)} H(A)\otimes_{\ground[S]} \ground [S,S^{-1}] &\simeq P\otimes_{H(A)} H(A)\otimes_{\ground[S]} \ground [S,S^{-1}] \cong P\otimes_{\ground[S]} \ground [S,S^{-1}]\\
 &\simeq M\otimes_{\ground[S]} \ground [S,S^{-1}]\cong M\otimes_{H(A)} H(A)\otimes_{\ground[S]} \ground [S,S^{-1}]
\end{align*}
using the fact that $\ground [S, S^{-1}]$ is flat over $\ground [S]$ (since it's the usual commutative localisation of $\ground [S]$ at $S$) and hence tensoring with it over $\ground [S]$ preserves quasi-isomorphisms. Therefore $H(A)\otimes_{\ground[S]} \ground [S,S^{-1}]$ is a flat left $H(A)$--module. The result now follows from \autoref{thm:dgflat}.
\end{proof}

\begin{remark}
If $A$ is a dg commutative algebra then it is not immediate that the derived localisation of $A$ in the category of dg associative algebras is the same as the derived localisation of $A$ in the category of dg commutative algebras. The latter is clearly the same as the non-derived localisation since in the commutative world the free product is just the tensor product and so $A\ast_{\ground [S]}^{\mathbb{L}} \ground [ S,S^{-1}] = A\otimes_{\ground [S]}^{\mathbb{L}}\ground [S,S^{-1}]\simeq A\otimes_{\ground [S]} \ground [S,S^{-1}]$. However, the results above imply that this is also still the localisation when regarding a dg commutative algebras as an object in the category of dg associative algebras.

By contrast, derived localisation of equivariant ring spectra need not preserve commutativity, c.f.\ \cite[Theorem 2]{McClure} and \cite[Proposition 6.1]{HopkinsHillB}. This explains why even though (an analogue of)
\autoref{thm:localisationscoincide} may hold for equivariant $A_\infty$ ring spectra, it certainly does not for equivariant $E_\infty$ ring spectra, c.f.\ \cite{HopkinsHill}.
\end{remark}

The results above give conditions for the derived localisation of the homology to be the same as the non-derived localisation of the homology, which in turn causes the spectral sequence in \autoref{thm:spectralseq} to collapse at the $E^2$~term. Of course, in general the derived localisation of the homology will be different from the non-derived localisation.

\begin{example}[Higher derived terms]
Let $\ground$ be a field of characteristic zero. Let $A$ be the associative algebra
\[
A = \ground\langle s,t,u \rangle / (st, us)
\]
concentrated in degree zero. Let $C$ be the dg algebra
\[
C = \ground\langle s,t,u, z,w,a \rangle
\]
with $\degree{s}=\degree{t}=\degree{u}=0$, $\degree{z}=\degree{w}=1$, $\degree{a}=2$ and differential defined by $dz=st$, $dw=us$ and $da=uz-wt$. Then $C$ is a cofibrant dg $\ground\langle s\rangle$--algebra and is a cofibrant replacement for $A$ (as above, $A$ is Koszul and $C$ is the cobar construction of $A^!$).

The derived localisation $C\ast_{\ground \langle s \rangle} \ground \langle s,s^{-1}\rangle \simeq A\ast_{\ground \langle s \rangle}^{\mathbb{L}}\ground \langle s,s^{-1}\rangle$ is freely generated over $\ground\langle s,s^{-1}\rangle$ by the degree $0$ elements $t$ and $u$, the degree $1$ elements $\tilde{z}=s^{-1}z$ and $\tilde{w}=ws^{-1}$ and the degree $2$ element $\tilde{a}=a-ws^{-1}z$. The differential on the modified generators is given by $d\tilde{z}=t$, $d\tilde{w}= u$ and $d\tilde{a} = 0$. Therefore $A\ast_{\ground \langle s \rangle}^{\mathbb{L}}\ground \langle s,s^{-1}\rangle$ may be realised as the graded algebra $\ground\langle \tilde{a}, s, s^{-1} \rangle$ with $\degree{s}=0$, $\degree{\tilde{a}}=2$ and vanishing differential.
\end{example}

In the above example it is still the case that the homology of the derived localisation is the same as that of the derived localisation of the homology, albeit tautologically. The next example shows that, in general, the derived localisation of a dg algebra may be quite different from the derived localisation of the homology. In particular, the spectral sequence in \autoref{thm:spectralseq} may not stabilise at the $E^2$~term.

\begin{example}[Spectral sequence does not stabilise at the $E^2$~term]
Let $A$ be the dg algebra
\[
A = \ground\langle s,t,u,z,w\rangle
\]
with $\degree{s}=\degree{t}=\degree{u}=0$, $\degree{z}=\degree{w}=1$ and differential defined by $dz=st$ and $dw=us$. Then $A$ is a cofibrant dg $\ground\langle s\rangle$--algebra. It is not hard to see that the space of degree $1$ cycles in $A$ is generated, as a module over $A_0$, by cycles of the form $us^nz-ws^nt$. But since $d(ws^nz) = us^{n+1}z-ws^{n+1}t$ then the only non-exact of these generating cycles is $uz-wt$. However, as a cycle in $A\ast_{\ground \langle s \rangle}^{\mathbb{L}}\ground \langle s,s^{-1}\rangle \simeq A\ast_{\ground \langle s \rangle}\ground \langle s,s^{-1}\rangle$ it is exact. Therefore, we see that $H_1(A\ast_{\ground \langle s \rangle}^{\mathbb{L}}\ground \langle s,s^{-1}\rangle)=0$.

On the other hand, the non-trivial homology class $[uz-wt]\in H(A)$ is still non-zero in $H(A)\ast_{\ground \langle s \rangle}\ground \langle s,s^{-1}\rangle$. Indeed it would only become zero after inverting $s$ if either $s^n(uz-wt)$ or $(uz-wt)s^n$ were exact in $A$ for some $n\in \mathbb{N}$, which, it is not hard to see, is not the case. Therefore, the non-derived localisation of $H(A)$ has a non-trivial degree $1$ element. But this means the same is true for the derived localisation. Indeed, compute $H(A)\ast_{\ground \langle s \rangle}^{\mathbb{L}}\ground\langle s,s^{-1}\rangle$ by cofibrantly replacing $\ground \langle s,s^{-1}\rangle$ and consider the grading with respect to the grading on the cofibrant replacement of $\ground \langle s,s^{-1}\rangle$. Then, as one would expect, the degree $0$ part with respect to this grading is just the non-derived localisation of $H(A)$.

Therefore, we conclude that $H(L_s(A))\ncong H(L_s(H(A)))$.
\end{example}

\begin{remark}
These examples highlight the stark contrast with the more familiar localisation of dg commutative algebras. In the commutative case, the derived localisation never has non-trivial higher derived terms and is therefore obtained by simply adjoining an inverse. In addition, the homology of the localisation of a dg commutative algebra is always just the localisation of the homology. On the other hand, the localisation of an arbitrary dg associative algebra seems to be rather less accessible.
\end{remark}

\subsection{Localisation of degree zero homology}
If $A$ is non-negatively graded and we are localising at degree zero homology classes then the degree zero homology of the localisation behaves as one would expect of a derived functor.

\begin{theorem}
Let $A$ be a differential non-negatively graded algebra with $S\subset H_0(A)$. Then $H_0(L_S(A))\cong H_0(A)[S^{-1}]\cong H_0(A)\ast_{\ground \langle S \rangle} \ground \langle S, S^{-1} \rangle$.
\end{theorem}

\begin{proof}
Write $\dgalg_{\geq 0}$ and $\alg$ for the categories of differential non-negatively graded algebras and algebras concentrated in degree $0$ respectively. The functor $A\mapsto H_0(A)$ is a left adjoint $\dgalg_{\geq 0}\to \alg$. Moreover, this is a left Quillen functor when equipping $\alg$ with the trivial model structure and it obviously preserves the property of being $S$--inverting, as does the right adjoint. It follows that the localisation of $A$ in $\dgalg_{\geq 0}$, being the initial $S$--inverting object in $A\downarrow^{\mathbb{L}}\dgalg_{\geq 0}$, is sent via this functor to the initial $S$--inverting object in $\undercat{H_0(A)}{\alg}$, which is $H_0(A)[S^{-1}]$.

Note that the localisation of $A$ in $\dgalg_{\geq 0}$ coincides with the localisation of $A$ in $\dgalg$ by a similar argument: the functor $\dgalg_{\geq 0}\to \dgalg$ is the left adjoint of a Quillen adjunction preserving the property of being $S$--inverting.
\end{proof}

However, in general, this theorem does not quite hold as the following example illustrates.

\begin{example}
Let $A$ be the cofibrant dg algebra
\[
A = \ground \langle t,u,s,z,w,h \rangle
\]
where $s,z,w$ have degree $0$, $t,u$ have degree $-1$ and $h$ has degree $1$. The differential on $A$ is defined by $d(s)=d(h)=d(t)=d(u)=0$, $d(z)=st$ and $d(w)=us$.

The degree $0$ cycle $h(uz+wt)$ gives a non-trivial homology class in $H_0(A)$ and is also non-trivial in $H_0(A)\ast_{\ground \langle s\rangle } \ground \langle s ,s ^{-1} \rangle$. However in $L_s(A)\simeq A\ast_{\ground \langle s \rangle} \ground \langle s, s^{-1} \rangle$ it is the boundary of $hws^{-1}z$.
\end{example}

\subsection{Ore localisation}
Let $A$ be a dg algebra and let $S\subset H(A)$ be an arbitrary set of homogeneous homology classes of $A$. \autoref{thm:dgflat} states that the homology $H(L_S(A))$ of the localisation of $A$ is isomorphic to $H(A)[S^{-1}]$, the non-derived localisation of the homology of $A$, if the latter is flat as a left or right $H(A)$--module. It is well known that if $S$ is a so-called \emph{Ore set} in $H(A)$, the flatness condition holds and further that $H(L_S(A))\cong H(A)[S^{-1}]$ may be realised as an Ore localisation, i.e.~as an algebra of fractions.
In this situation the derived localisation $L_S(A)$ itself cannot, in general, be constructed as a dg algebra of fractions in a directly analogous fashion. We will show in this subsection that it is nevertheless possible to identify $L_S(A)$ in $\Ho(\modcat{A})$ as a homotopy module of fractions.

We remark that a version of the main result of this paper, \autoref{thm:localisationscoincide}, in the context of $\mathbb{E}_1$--algebras, is proved in \cite[\S7.2.4]{lurie} under the assumption that $S$ is an Ore set.
Our approach, by contrast, is to establish \autoref{thm:localisationscoincide}
in full generality, without any restrictions on $S$, and then appeal to the theorem in developing the theory of Ore localisation in case $S$ is an Ore set.

We begin with an account of Ore localisation in the non-differential context. We follow the approach of Quillen \cite[Appendix Q.1]{quillen}, which makes a clear (initial) distinction between algebras and modules of fractions, much in keeping with our overall point of view.  Let $B$ be a graded algebra (with zero differential), and let $S$ be a set of homogeneous elements of $B$. We assume that $S$ is multiplicatively closed, i.e.~$1\in S$ and $s,t\in S\implies st\in S$. We form the left $B$--module $B\frc S$ of right fractions of $B$ with denominators in $S$. It is defined as the free graded $B$--module on symbols $\frc s$, one for each $s\in S$, with the degree of $\frc s$ being the same as $s$, modulo the $B$--submodule generated by $x\frc sx - 1\frc x$, for $s\in S$ and homogeneous $x\in B$ such that $sx\in S$. We may understand $B\frc S$ as a colimit as follows. Let $\mathcal{S}$ be the category with objects $s\in S$, a morphism $s\xrightarrow{x} t$ for every $s,t \in S$ and homogeneous $x\in B$ such that $sx=t$, and composition given by $\left( t\xrightarrow{y} u\right)\circ\left( s\xrightarrow{x} t\right) = s\xrightarrow{xy} u$. Then $B\frc S$ is the colimit of the functor $\mathcal{S}\to\modcat{B}$ sending $s$ to the shifted free module $B\frc s:=\Sigma^{-\degree{s}}B$ of rank one, and $s\xrightarrow{x} sx$ to right multiplication $B\frc s\xrightarrow{\cdot x} B\frc sx$ by $x$. There is a canonical $B$--module homomorphism $f\co B\to B\frc S$ sending $b$ to $b\frc 1$.

\begin{lemma}\label{orelemma1}
Let $B$ be a graded algebra, and let $S$ be a multiplicatively closed subset of $B$. Then $f\co B\to B\frc S$ is the non-derived module localisation of $B$ if and only if $B\frc S$ is $S$--local.
\end{lemma}

\begin{proof}
It suffices to show that $f^*\co \Hom_B(B\frc S,N)\to\Hom_B(B,N)\cong N$ is an isomorphism for any $S$--local graded $B$--module $N$. Fix $y\in N_0$, and let $f_y\co B \to N$ be the $B$--module homomorphism such that $f_y(1)=y$. We must show that there is a unique $B$--module homomorphism $g\co B\frc S\to N$ such that $g\circ f = f_y$, i.e.~such that $g(1\frc 1)=y$. If such a $g$ exists, then for all $s\in S$ and $b\in B$, we have $sg(1\frc s)=g(s\frc s)=g(1\frc 1)=y$ and therefore $g(b\frc s)=bl_s^{-1}(y)$. It is easy to check that the latter formula indeed determines a well-defined $B$--module homomorphism from $B\frc S$ to $N$.
\end{proof}

The following lemma gives necessary and sufficient conditions for $B\frc S$ to be $S$--local, under the assumption that $\mathcal{S}$ is filtered, i.e.
\begin{itemize}
\item for all $s,t\in S$, there exist $x,y\in B$ such that $sx=ty\in S$;
\item for all $s\in S$ and $x,y\in B$ such that $sx=sy\in S$, there exists $z\in B$ such that $xz=yz$ and $sxz\in S$.
\end{itemize}

\begin{lemma}\label{orelemma2}
Assume that $\mathcal{S}$ is filtered. Let $s\in S$.
\begin{enumerate}
\item The map $l_s\co B\frc S\to B\frc S$ is surjective if and only if for all $b\in B$, we have $sB\cap bS\neq \emptyset$;
\item The map $l_s\co B\frc S\to B\frc S$ is injective if and only if for all $b\in B$, we have $sb=0 \implies 0\in bS$.
\end{enumerate}
\end{lemma}

\begin{proof}
The assumption that $\mathcal{S}$ is filtered implies every element of $B\frc S$ can be written in the form $b\frc t$ for some $b\in B$ and $t\in S$, and that $b\frc t = 0$ if and only if there exists $x\in B$ such that $bx=0$ and $tx\in S$.
\begin{enumerate}
\item Suppose the condition holds, and let $b\frc t \in B\frc S$. Then there exists $c\in B$ and $u\in S$ such that $sc=bu$, and then $b\frc t=bu\frc tu=sc\frc tu = l_s(c\frc tu)$. Conversely, suppose that $l_s$ is surjective. Given any $b\in B$, we have $l_s(c\frc t)=b\frc 1$ for some $c \in B$ and $t\in S$. Hence $sc\frc t=bt\frc t$, and therefore $scx=btx$ for some $x\in B$ such that $tx\in S$. Thus $sB\cap bS$ is nonempty.
\item Suppose the condition holds, and let $b\frc t \in B\frc S$ such that $l_s(b\frc t)=0$. Then $sbx=0$ for some $x\in B$ such that $tx\in S$. So $bxu=0$ for some $u\in S$, and since then $txu\in S$, we deduce $b\frc t=0$. Conversely, suppose that $l_s$ is injective. Given $b\in B$ such that $sb=0$, we have $l_s(b\frc 1)=0$. Hence $b\frc 1=0$, and therefore $bt=0$ for some $t\in S$.
\end{enumerate}
\end{proof}

We are thus led to consider the following standard conditions.
\begin{definition}
Let $B$ be a graded algebra and $S$ a multiplicatively closed subset of homogeneous elements of $B$. We say that $S$ is a \emph{right Ore set} in $B$ if the following two conditions hold.
\begin{enumerate}
\item for all $b\in B$ and $s\in S$, we have $bS\cap sB\neq \emptyset$;
\item for all $b\in B$ and $s\in S$, we have $sb=0 \implies 0\in bS$.
\end{enumerate}
\end{definition}

We remark that some authors require only the first condition, calling a subset satisfying both a \emph{right denominator set} instead. For a subset $S$ satisfying the first condition alone, it is not necessarily true that $B[S^{-1}]$ is flat as a left $B$--module, cf.~\autoref{oreexample} below, and therefore such $S$ are not sufficiently well-behaved for our purposes.

\begin{proposition}[Ore localisation]\label{prop:ore}
Let $B$ be a graded algebra, and let $S$ be a multiplicatively closed subset of $B$. Then $S$ is a right Ore set in $B$ if and only if $\mathcal{S}$ is filtered and $B\frc S$ is $S$--local. If $S$ is a right Ore set in $B$, then we have an isomorphism $B\frc S \cong B[S^{-1}]$ of $B$--modules given by $b\frc s\mapsto bs^{-1}$. The product induced on $B\frc S$ is given, for $b,c\in B$ and $s,t\in S$, by $(b\frc s)(c\frc t)=ba\frc tu$, where $a\in B$ and $u\in S$ are chosen so that $cu=sa$.
\end{proposition}

\begin{proof}
It is clear that $\mathcal{S}$ is filtered when $S$ is a right Ore set in $B$. The first claim in the theorem is thus a restatement of \autoref{orelemma2}.
	
Let $S$ be a right Ore set in $B$. By \autoref{orelemma1}, $B\frc S$ is the non-derived module localisation of $B$. The isomorphism $B\frc S \cong B[S^{-1}]$ is then an immediate consequence of the non-derived analogue of \autoref{thm:localisationscoincide}; here we supply a more direct argument. We have $B\frc S\cong \Hom_B^{\bullet}(B,B\frc S)\cong \Hom_B^\bullet(B\frc S,B\frc S)$, so $B\frc S$ is a graded algebra. Tracing through the isomorphisms, we arrive at the stated formula for the product, and the map $B\to B\frc S$, $b\mapsto b\frc 1$, makes $B\frc S$ an $S$--inverting graded $B$--algebra. Given an arbitrary $S$--inverting graded $B$--algebra $g\co B\to C$, the unique $B$--module homomorphism $B\frc S\to C$ over $B$, given by $b\frc s\mapsto bg(s)^{-1}$ (see the proof of \autoref{orelemma1}), is in fact an algebra homomorphism. This means that $B\frc S$ satisfies the universal property that defines the non-derived algebra localisation $B[S^{-1}]$.
\end{proof}

Now we return to our main concern, the derived localisations of dg algebras.

\begin{proposition}\label{thm:ore}
Let $A$ be a dg algebra, and let $S$ be a right Ore set in $H(A)$.
\begin{enumerate}
\item We have
\[
H(L_S(A))\cong H(A)[S^{-1}] \simeq L_S(H(A))
\]
as $H(A)$--algebras.	
\item Let $y_1, y_2,\dots$ be a sequence of homology classes of $A$ such that
\begin{enumerate}
\item for all $n\geq 1$, we have $y_1 y_2 \cdots y_n\in S$.
\item for all $m\geq 1$ and $s\in S$, there exists $n\geq m$ such that $s$ is a left factor of $y_m y_{m+1} \cdots y_n$ in $H(A)$.
\end{enumerate}
Then
\[
L_S(A) \simeq \hocolim A\xrightarrow{\cdot x_1} \Sigma^{-\degree{x_1}}A\xrightarrow{\cdot x_2}\dots
\]
in $\Ho(\modcat{A})$, for any cycles $x_1, x_2,\dots$ in $A$ representing $y_1, y_2,\dots$. Moreover, if $S$ is countable, then such a sequence $y_1, y_2,\dots$ always exists.
\end{enumerate}
\end{proposition}

\begin{proof}
By \autoref{prop:ore}, $H(A)[S^{-1}]\cong H(A)\frc S$ is a filtered colimit of free $A$--modules, hence flat, and (1) follows by \autoref{thm:dgflat}. Let $y_1,y_2,\dots$ be a sequence of homology classes of $A$ satisfying (a) and (b). The functor from $\mathbb{N}$ to $\mathcal{S}$ sending $n$ to $y_1\cdots y_n$ and $n\rightarrow n+1$ to $y_1\cdots y_n\xrightarrow{\cdot y_{n+1}} y_1\cdots y_{n+1}$ is cofinal, and therefore
\[
H(A)\frc S\cong\colim H(A)\xrightarrow{\cdot y_1} \Sigma^{-\degree{y_1}}H(A)\xrightarrow{\cdot y_2}\dots.
\]
Now let
\[
\overline{A}_\infty:=\hocolim A\xrightarrow{\cdot x_1}  \Sigma^{-\degree{x_1}}A \xrightarrow{\cdot x_2}\dots,
\]
where $x_1, x_2,\dots$ are cycles in $A$ representing $y_1, y_2,\dots$. Then $H(\overline{A}_\infty)\cong H(A)\frc S$, and in particular $\overline{A}_\infty$ is an $S$--local $A$--module. The induced map $L_S(A)\to \overline{A}_\infty$ is a quasi-isomorphism, by (1).
	
Finally, suppose $S$ is countable. Fix a sequence $s_1, s_2,\dots$ in $S$ containing each element of $S$ at least once. We inductively construct a sequence $y_1,y_2,\dots$ of homology classes of $A$ such that, for all $r\geq 1$, $s_r$ is a left factor in $H(A)$ of the product $y_m y_{m+1} \cdots y_{m+r-1}$ of any $r$ consecutive terms, and the product of the first $r$ terms is in $S$. We take $y_1=s_1$. Suppose we are given $y_1,\dots, y_{n}\in H(A)$ satisfying the desired properties for $r\leq n$. Using the Ore condition $s_{n+1}S\cap y_1\cdots y_n H(A)\neq \emptyset$, we find homogeneous $y\in H(A)$ such that $s_{n+1}$ is a left factor of $y_1\cdots y_n y$ in $S$. Next, since $y_2\cdots y_n y S\cap s_n H(A)\neq \emptyset$, there exists $t_1\in S$ such that $s_n$ is a left factor of $y_2\cdots y_n y t_1$. Continuing in this way, we obtain $t_1,\dots t_{n-1}\in S$ so that $t_r$ is a left factor of $y_{n-r+2}\cdots y_n y t_1 \cdots t_{n-r+1}$, for $r=1,\dots, n+1$.  Put $y_{n+1}=yt_1\cdots t_n$. Then $y_1\cdots y_{n+1}=(y_1\cdots y_n y)t_1\cdots t_n \in S$, and $s_r$ is a left factor of $y_{n-r+2}\cdots y_{n+1}$ for $r=1,\dots, n+1$. Hence the sequence $y_1,\dots, y_{n+1}$ satisfies the desired conditions for $r\leq n+1$.	
\end{proof}

The second part of \autoref{thm:ore} shows that if $S$ is a right Ore set in $H(A)$, then the module localisation $L_S(A)$ can be realised as a homotopy direct limit of free $A$--modules, provided $S$ is countable. We now establish a converse statement, for which the assumption on the cardinality of $S$ is not required.

Observe that in localising a dg algebra $A$ at an arbitrary set $S$ of homology classes, there is no harm in replacing $S$ by the larger set $\hat{S}$ of homology classes that become invertible in $H(L_S(A))$, in the sense that $L_{\hat{S}}(A)\simeq L_S(A)$.

\begin{proposition}\label{thm:ore2}
Let $A$ be a dg algebra, and let $S\subset H(A)$ be an arbitrary set of homology classes of $A$. Assume that there exists a sequence
$y_1,y_2,\dots$ in $\hat{S}$ with representing cycles $x_1, x_2,\ldots$ in $A$ such that the homotopy direct limit
\[
\overline{A}_\infty=\hocolim A\xrightarrow{\cdot x_1} \Sigma^{-\degree{x_1}}A\xrightarrow{\cdot x_2}\dots
\]
is an $S$--local $A$--module. Then $L_S(A)\simeq \overline{A}_\infty$. Moreover any multiplicatively closed subset $S'$ of $\hat{S}$ containing $y_1, y_1y_2,\dots$, is a right Ore set in $H(A)$ such that $L_{S'}(A)\simeq L_S(A)$.
\end{proposition}
	
Note that the minimal choice for $S'$, the multiplicatively closed subset generated by $y_1, y_1y_2,\dots$, is countable.

\begin{proof}
Let $N$ be an $S$--local $A$--module. Then
\[
\RHom_A(\overline{A}_\infty, N)\simeq \holim (\dots \xrightarrow{l_{x_2}} \Sigma^{\degree{x_1}}N \xrightarrow{l_{x_1}} N) \simeq N,
\]
since by assumption each $x_n$ acts as a quasi-isomorphism on $N\simeq L_S(A)\otimes_A^{\mathbb{L}} N$. Thus $\overline{A}_\infty$ is an $S$--local $A$--module satisfying the universal
property defining the module localisation $L_S(A)$. The same argument applies with $S$ replaced by $S'$, so we have $L_S(A)\simeq \overline{A}_\infty\simeq L_{S'}(A)$.
	
Observe that the homology of $\overline{A}_\infty$ is isomorphic to the direct limit
\[
\overline{H(A)}_\infty:=\colim H(A)\xrightarrow{\cdot y_1} \Sigma^{-\degree{y_1}}H(A)\xrightarrow{\cdot y_2}\dots,
\]
which can be realised as the free $H(A)$--module generated by symbols $\frc y_1y_2\cdots y_n$ of the same degree as $y_1y_2\cdots y_n$, modulo the $H(A)$--submodule
generated by $y_{n+1}\frc y_1y_2\cdots y_{n+1} - 1\frc y_1y_2\cdots y_n$. 	
	
We will use the fact that $\overline{H(A)}_\infty\cong H(L_{S'}(A))$ is an $S'$--local $H(A)$--module to establish that $S'$ is a right Ore set in $H(A)$. Let $b\in H(A)$ and $s\in S'$. Since $l_s\co \overline{H(A)}_\infty\to\overline{H(A)}_\infty$ is surjective, we have $sc\frc y_1y_2\cdots y_n = b\frc 1$ in $\overline{H(A)}_\infty$ for some $c\in H(A)$ and some $n$. Hence $scy_{n+1}y_{n+2}\cdots y_m = by_1y_2\cdots y_m$ for some $m$, and therefore $sH(A)\cap bS\neq\emptyset$. Next, suppose that $sb=0$ for some $s\in S'$ and $b\in H(A)$. Then $sb\frc 1=0$ in $\overline{H(A)}_\infty$. By the injectivity of $l_s\co \overline{H(A)}_\infty\to\overline{H(A)}_\infty$, we deduce that $b\frc 1=0$ in $\overline{H(A)}_\infty$, and therefore that $by_1y_2\cdots y_n=0$ for some $n$. Hence $0\in bS$, as required.
\end{proof}

When $S=\{s\}$ contains a single homology class, we can take $s=y_1=y_2=\dots$ in \autoref{thm:ore} and \autoref{thm:ore2} to arrive at the following result.

\begin{corollary}\label{oneelementore}
Let $A$ be a dg algebra, and let $s\in H(A)$ be a homology class in $A$. The following statements are equivalent.
\begin{enumerate}
\item $A\otimes_{\ground\langle s\rangle} \ground\langle s,s^{-1}\rangle$ is an $s$--local $A$--module.
\item $L_s(A)\simeq A\otimes_{\ground\langle s\rangle} \ground\langle s,s^{-1}\rangle$ as $A$--modules.
\item $H(A)[s^{-1}]\cong H(A)\otimes_{\ground\langle s\rangle} \ground\langle s,s^{-1}\rangle$ as $H(A)$--modules
\item $\{1,s,s^2,\dots\}$ is a right Ore set in $H(A)$.
\end{enumerate}	
If any of these statements holds, then $H(L_s(A)) \cong H(A)[s^{-1}] \simeq L_s(H(A))$ as
$H(A)$--algebras.
\end{corollary}

\begin{remark}
It is of course possible to define what it means for a subset $S$ of homogeneous elements of a graded algebra $B$ to be a \emph{left Ore set} and to relate the property to certain right modules of fractions. All of the definitions and results above go through with the obvious modifications.
\end{remark}

\begin{example}\label{oreexample}
Let $\ground$ be a field, and consider the algebra $B=\ground\langle s,t \rangle \frc  (st-1)$, with $s$ and $t$ in degree $0$. The set $S=\{1,s,s^2,\ldots\}$ is not a right Ore set in $H(B)=B$, since $s(1-ts)=0$ but $(1-ts)s^n\neq 0$ for all $n$. On the other hand, the element $t$ clearly becomes invertible in $L_S(A)$, and $S'=\{1,t,t^2,\ldots\}$ \emph{is} a right Ore set in $B$. We have $L_S(B)\simeq L_{S'}(B) \simeq B[S'^{-1}] \cong B[S^{-1}]\cong \ground\langle s, s^{-1}\rangle$.
	
Next, let $A=\ground\langle s,t,y \rangle \frc  (st-1)\cong B \ast_{\ground}\ground\langle y \rangle$. The set $S=\{1,s,s^2,\ldots\}$ is not a right Ore set in $H(A)=A$; as in the previous example, it satisfies the first Ore condition but not the second. This time we cannot replace $S$ by an Ore set $S'$ such that $L_S(A)\simeq L_{S'}(A)$, because $A[S^{-1}]\cong \ground \langle s,s^{-1},y\rangle$ is not flat as a left $A$--module. Indeed the map $l_x\co A\to A$ of right $A$--modules given by left multiplication by $x=(1-ts)y$ is injective, but $l_x\otimes_A A[S^{-1}]\co A[S^{-1}]\to A[S^{-1}]$, being left multiplication by $0$, the image of $x$ in $A[S^{-1}]$, is not. Incidentally, a similar argument interchanging the roles of $s$ and $t$ shows that $A[S^{-1}]$ is not flat as a right $A$--module either. On the other hand, $L_S(A)\simeq L_S(B)\ast_{\ground}\ground\langle y \rangle\cong B[S^{-1}]\ast_{\ground}\ground\langle y \rangle\cong \ground \langle s,s^{-1},y\rangle \cong A[S^{-1}]$, and therefore $A[S^{-1}]$ is stably flat over $A$, by \autoref{thm:stableflat}.
\end{example}

\section{Hochschild homology and cohomology}
In this section we will show that derived localisation preserves Hochschild homology and cohomology. This can be seen as generalising analogous results \cite{brylinskihochschild} about of the Hochschild homology/cohomology of the non-derived localisation under the assumption that the localising set $S$ is central.

Without loss of generality, we may assume $A$ and $B$ are cofibrant dg algebras throughout this section. Moreover, we may also assume $L_S(A)$ and $L_S(B)$ are cofibrant over $A$ and $B$ respectively. Given subsets $S\subset H_0(A)$ and $S'\subset H_0(B)$ denote by $S\otimes S'$ the subset $\{\,s\otimes 1 : s\in S\, \} \cup \{\, 1\otimes s' : s\in S'\, \}\subset H_0(A\otimes B)$.

\begin{proposition}
For any $S\subset H_0(A)$, $S'\subset H_0(B)$ then $L_{S\otimes S'}(A\otimes B) \simeq L_S(A)\otimes L_{S'}(B)$.
\end{proposition}

\begin{proof}
It is sufficient to show $L_S(A)\otimes L_{S'}(B)$ is the localisation of $A\otimes B$ as an $A\otimes B$--module. But this is clear since for $N$ any $S\otimes S'$--local module:
\begin{align*}
\RHom_{A\otimes B}(L_S(A)\otimes L_{S'}(B), N) &\simeq \RHom_A(L_S(A),\RHom_B(L_{S'}(B),N))\\
&\simeq \RHom_A(L_S(A), N) \simeq N
\end{align*}
\end{proof}

\begin{theorem}
Let $A$ be a dg algebra and let $M$ be an $L_S(A)$--bimodule. Then $\HH_*(L_S(A),M)\simeq \HH_*(A,M)$ and $\HH^*(L_S(A),M)\simeq \HH^*(A,M)$.
\end{theorem}

\begin{proof}
Since $M$ is $S\otimes S$--local as an $A\otimes A^{\op}$--module then, since $L_{S\otimes S}(A\otimes A^{\op})\simeq L_S(A)\otimes L_S(A)^{\op}$, by \autoref{prop:smashing} $M\simeq (L_S(A)\otimes L_S(A)^{op})\otimes_{A\otimes A^{\op}}^{\mathbb{L}} M$. The Hochschild homology $\HH_*(L_S(A),M)$ is the homology of
\begin{align*}
L_S(A)\otimes_{L_S(A)\otimes L_S(A)^{\op}}^{\mathbb{L}} M &\simeq L_S(A)\otimes_{L_S(A)\otimes L_S(A)^{\op}}^{\mathbb{L}} (L_S(A)\otimes L_S(A)^{\op})\otimes_{A\otimes A^{\op}}^{\mathbb{L}} M\\
& \simeq L_S(A)\otimes_{A\otimes A^{\op}}^{\mathbb{L}} M
\end{align*}
Similarly, $M$ is $S\otimes 1$--local so we also have $M\simeq (L_S(A)\otimes A^{\op})\otimes_{A\otimes A^{\op}}^{\mathbb{L}} M$. The Hochschild homology $\HH_*(A,M)$ is the homology of:
\begin{align*}
A\otimes_{A\otimes A^{\op}}^{\mathbb{L}} M &\simeq A\otimes_{A\otimes A^{\op}}^{\mathbb{L}} (L_S(A)\otimes A^{\op})\otimes_{A\otimes A^{\op}}^{\mathbb{L}} M\\
&\simeq L_{S\otimes 1}^{\modcat{A\otimes A^{\op}}}(A)\otimes_{A\otimes A^{\op}}^{\mathbb{L}} M \simeq L_S(A)\otimes_{A\otimes A^{\op}}^{\mathbb{L}} M
\end{align*}
using the fact that $L_{S\otimes 1}^{\modcat{A\otimes A^{\op}}}(A) \simeq L_S(A)$ as an $A$--bimodule (see the proof of \autoref{thm:dwyer}).

The corresponding statement for Hochschild cohomology is similar. Indeed:
\begin{align*}
\RHom_{L_S(A)\otimes L_S(A)^{\op}}(L_S(A), M) &\simeq \RHom_{L_S(A)\otimes L_S(A)^{\op}}(L_S(A)\otimes_{A\otimes A^{\op}}^{\mathbb{L}}(L_S(A)\otimes L_S(A)^{\op}), M)\\
&\simeq \RHom_{A\otimes A^{\op}}(L_S(A), \RHom_{L_S(A)\otimes L_S(A)^{\op}}(L_S(A)\otimes L_S(A)^{\op}, M))\\
&\simeq \RHom_{A\otimes A^{\op}}(L_S(A),M)\\
& \simeq \RHom_{A\otimes A^{\op}}(L_{S\otimes 1}^{\modcat{A\otimes A^{\op}}}(A),M)\\
& \simeq \RHom_{A\otimes A^{\op}}(A, M)\qedhere
\end{align*}
\end{proof}

\section{Torsion modules and the localisation exact sequence}
There is also the corresponding theory of (derived) colocalisation.

\begin{definition}
An $A$--module $M$ is called \emph{$S$--torsion}, or \emph{$S$--colocal} if $\RHom_A(M,N)=0$ for any $S$--local module $N$. A colocalisation of a module $M$ is an $S$--colocal module $L^S(M)$ with a map $f\co L^S(M)\to M$ such that for any $S$--colocal $A$--module $N$ and any map $f^\prime\co N\to M$ there is a map $g\co N\to L^S(M)$ such that $f\circ g=f^\prime$ and furthermore $g$ is unique in the homotopy category of $\amod$. The category of $S$--colocal $A$--modules will be denoted by $(\amod)^S$
\end{definition}

Thus, a colocalisation functor is right adjoint (in the homotopy category) to the inclusion of colocal $A$--modules into all $A$--modules (just as the localisation functor is left adjoint to the inclusion of local $A$--modules into $A$--modules). Note that the categories of $S$--local and $S$--colocal $A$--modules are naturally pretriangulated (so that their homotopy categories are triangulated).

The following result is standard and easy, cf.~for example \cite{hoveypalmieristrickland:homotopy} regarding such statements.

\begin{proposition}\label{prop:verdierloc}
An $S$--colocalisation functor exists. Moreover, for any $A$--module $M$ there is a homotopy fibre sequence of $A$--modules
\[
L^S(M)\to M\to L_S(M).
\]
\end{proposition}

\begin{proof}
Given an $A$--module $M$ define $L^S(M)$ to be the homotopy fibre of the localisation map $M\to L_S(M)=L_S(A)\otimes_A M$ where $L_S(A)$ is a cofibrant model for a derived localisation of $A$. Given an $S$--local module $N$ we have a cofibre sequence in the derived category of $\ground$--modules:
\[
\RHom_A(L^S(M), N)\to\RHom_A(M,N)\to\RHom_A(L_S(M),N),
\]
and since $L_S(M)$ is a localisation of $M$, the $\ground$--modules $\RHom_A(M,N)$ and $\RHom_A(L_S(M),N)$ are quasi-isomorphic, so $\RHom_A(L^S(M), N)$ is quasi-isomorphic to zero, so $L^S(M)$ is $S$--torsion. A similar argument shows that for any $S$--torsion module $L$ there is a quasi-isomorphism
\[
\RHom_A(L, M)\cong\RHom_A(L,L^S(M)).
\]
\end{proof}

Recall \cite{keller:ondgcat,tabuada:hmo} that the $2$--category of dg categories admits the structure of a closed model category in which a weak equivalence between two dg categories $\mathcal C$ and $\mathcal B$ is a dg functor $F\co\mathcal A\to\mathcal B$ such that the induced triangulated functor on the corresponding derived categories $\mathbb{L}F_!\co\mathcal{\Ho(\amod)\mapsto \Ho(\modcat{B}})$ is an equivalence. The corresponding homotopy category will be denoted by $\Hmo$.

A sequence
\begin{equation}\label{eq:exactseq}
\xymatrix{\mathcal{A}\ar^F[r]&\mathcal B\ar^G[r]&\mathcal C }
\end{equation}
of triangulated categories is \emph{exact} \cite{neeman1992} if $F$ is is full and faithful, the composition $G\circ F$ is zero and the induced functor $\mathcal {B/A}\mapsto \mathcal C$ is cofinal, i.e.~it is induces an equivalence on idempotent completions of $\mathcal {B/A}$ and $ \mathcal C$. If \ref{eq:exactseq} is a diagram of dg categories then it is called exact in $\Hmo$ if the corresponding sequence of derived categories is exact.

Furthermore, a dg category $\mathcal A$ is \emph{pretriangulated} if it contains shifts and cones of objects. For example, the categories $\amod, \amod_S$ and $\amod^S$ are pretriangulated. For a pretriangulated category $\mathcal A$ we denote by $[\mathcal A]^c$ its full subcategory consisting of those objects whose images in the homotopy category $H^0(\mathcal A)$ are compact. For example if $A$ is an ungraded algebra then the category $[\amod]^c$ consists of bounded complexes of finitely generated projective $A$--modules (and complexes quasi-isomorphic to such things).

The following result is a straightforward corollary of \autoref{thm:localisationscoincide}.

\begin{theorem}\label{thm:exactsequence}
The sequence of dg categories and functors \[[\amod^S]^c\mapsto [\amod]^c\mapsto [\modcat{L_S(A)}]^c\]
is exact in $\Hmo$.
\end{theorem}

\begin{proof}
Using \autoref{prop:verdierloc} and \autoref{cor:localcat} we conclude that there is an exact sequence in $\Hmo$:
\[
\amod^S\mapsto \amod\mapsto \modcat{L_S(A)}
\]
Now the result follows from the Neeman--Thomason localisation theorem \cite[Theorem 2.1]{neeman1992}.
\end{proof}

\begin{remark}
Since for a dg algebra $A$ the dg category $[\amod]^c$ is Morita equivalent to $A$, the above result could be viewed as the identification of the homotopy fibre of the localisation map $A\to L_S(A)$ in $\Hmo$.
\end{remark}

\begin{corollary}\label{cor:exactsequence}
There is a cofibre sequence of $K$--theory spectra
\[
\itbbk([\amod^S]^c)\to \itbbk(A)\to \itbbk(L_S(A))
\]
where $\itbbk(?)$ stands for the non-connective $K$--theory spectrum \cite{schlichting:hak}.
\end{corollary}

The last corollary is a generalisation of the result of Neeman--Ranicki \cite{ranickineeman} obtained under the assumption of stable flatness. Similar results hold for other localising invariants, such as the Hochschild and cyclic homology \cite{Keller_exact, Keller_cyclic}, topological Hochschild and cyclic homology \cite{blumbergmandell:localisation}.

It is natural to ask whether the homotopy category of $\amod^S$ is a derived category of some dg algebra. This turns out to be true, at least when the set $S$ is finite. For $s\in S$ denote by $A/s$ the cofibre of the right multiplication by $s$ so there is a homotopy cofibre sequence of $A$--modules: $\xymatrix{A\ar^{\cdot s}[r]& A\ar[r]& A/s}$. Next, set $A/S:=\bigoplus_{s\in S}A/s$; it is a compact $A$--module since each $A/s$ is. Then we have the following result.

\begin{proposition}\label{prop:torsion}
If the set $S$ is finite, then the category $\Ho(\amod^S)$ is equivalent to
the derived category $\Ho(\modcat{\REnd_A(A/S)})$.
\end{proposition}

\begin{proof}
Let us say that $A$--module $M$ is $A/S$--trivial if $\RHom_A(A/S, M)=0$. Then $M$ is $A/S$--trivial if an only if $\RHom_A(A/s, M)=0$ for all $s\in S$ and this is, in turn, equivalent to the condition that $M$ is $s$--local for all $s\in S$ or that $M$ is $S$--local. Thus, being $A/S$--trivial is the same as being $S$--local. Therefore, an $A$--module $N$ is $S$--colocal if and only if it is $A/S$--torsion in the sense of \cite{Dwyer_Greenlees}. The result now follows from op.cit., Theorem 2.1.
\end{proof}

\begin{remark}
The paper of Dwyer--Greenlees used in the proof of the proposition above was written in the language of homotopy categories, however it is easy to check that the equivalence of homotopy categories proved in it, actually came from a Quillen equivalence of the underlying closed model categories and the same statement holds for the categories $\amod^S$ and $\modcat{\REnd_A(A/S)}$.
\end{remark}

\begin{remark}
Note that \autoref{thm:exactsequence}, \autoref{cor:exactsequence} and \autoref{prop:torsion} continue to be valid for the more general notion of derived Cohn localisation.
\end{remark}

\begin{remark}\label{recollement}
The restriction functor $\modcat{L_S(A)}\to\amod$ has both a left adjoint $L_S(A)\otimes_{A}-$ and a right adjoint $\Hom_A(L_S(A),-)$. Passing to the homotopy categories, the three functors induce half of the arrows of a recollement
\begin{equation*}
\Ho(\modcat{L_S(A)})\quad \substack{\longleftarrow\\
	 \longrightarrow \\
	  \longleftarrow}\quad
\Ho(\amod)\quad \substack{\longleftarrow\\
	\longrightarrow \\
	\longleftarrow}\quad
\Ho(\amod^S)
\end{equation*}
(see, e.g.~\cite[\S5.5]{krause}). Hence, to any $A$--module $M$ is associated two canonical fibre sequences of $A$--modules depending on $S$, that in \autoref{prop:verdierloc}, and another $K_S(M)\to M \to K^S(M)$, in which $K_S(M)$ is $S$--local and $K^S(M)$ is \emph{$S$--complete}, i.e.~$\RHom_A(N,K^S(M))=0$ for any $S$--local module $N$. Note that the functor $M\mapsto K^S(M)$ could be viewed as a (derived) completion of $M$ at $S$. It can also be obtained as a Bousfield localisation with respect to a homology theory on $A$--modules given by the $A$--module $\RHom_A(A/S, A)$, at least when $S$ is a finite set (\cite[Proposition 4.8]{Dwyer_Greenlees}).
\end{remark}
\section{Strictification of homotopy unital dg algebras}
Let $A$ be a non-unital dg algebra but such that its homology $H(A)$ is a unital graded algebra. We say that $A$ is \emph{homotopy unital}. As a curious application of our derived localisation techniques we show that $A$ could always be strictified to a unital dg algebra. More precisely:

\begin{theorem}
Let $A$ be a homotopy unital dg algebra. Then there exists a (unital) dg algebra $B$ and a quasi-isomorphism of non-unital dg algebras $A\to B$. The dg algebra $B$ is determined uniquely up to a quasi-isomorphism of (unital) dg algebras.
\end{theorem}

\begin{proof}
Let $A_+$ be the (unital) dg algebra obtained by adjoining a unit to $A$; thus $A_+\cong A\oplus\ground\cdot 1$; there is an embedding of non-unital dg algebras $A\hookrightarrow A_+$. Let $e$ be a cycle in $A$ representing the unit in $H(A)$; we will view it also as a cycle in $A_+$, clearly its homology class is in the centre of $H(A_+)$. It follows, from \autoref{cor:central} that the composite map $A\to A_+\to L_eA_+$ is a multiplicative quasi-isomorphism and so the existence part of the theorem is proved.

We claim that in the case $A$ is already unital, the dg algebra $L_eA_+$ is quasi-isomorphic to $A$ as a unital dg algebra. Indeed, $e$ could be chosen to be the strict unit; then $f(e)$ is in the centre of $A_+$  and it follows that $L_eA_+\simeq A_+[e^{-1}]\cong A$.

Now assume that there is a homotopy unital dg algebra $A$ and two multiplicative quasi-isomorphisms $g\co A\to B$ and $h\co A\to C$ where both $B$ and $C$ are unital. Applying the above construction to the maps $g$ and $h$ we obtain quasi-isomorphisms of unital dg algebras $L_eA_+\to L_{g(e)}B_+$ and $L_eA_+\to L_{h(e)}C_+$. But we have already established that $L_{g(e)}B_+$ and $L_{h(e)}C_+$ are quasi-isomorphic as unital dg algebras to
$B$ and $C$ respectively. Therefore $B$ and $C$ were quasi-isomorphic to begin with, proving uniqueness.
\end{proof}

\begin{remark}
The above result is not difficult to prove in the case $\ground$ is a field, but it was pointed out to the authors by F.~Muro that the general case is substantially harder; a version of it was proved in \cite[\S5.4.3]{lurie}, using sophisticated machinery of $\infty$--categories. The essentially trivial proof above, of course, relies on the techniques of derived localisation, specifically, \autoref{cor:central}.
\end{remark}

\section{Idempotent ideals and derived quotients}
By analogy with derived localisation, one can ask whether there exists a homotopy invariant way to quotient out an ideal in the homology of a dg algebra and which is characterised by a suitable universal property. Of course, the answer is negative in general, since the homotopy category of dg algebras does not have colimits. Nevertheless, there is one interesting special case when such a construction is possible, and it then reduces to derived localisation.

Let $x\in H(A)$. We say that an $A$--algebra $B$ is \emph{$x$--killing} if the image of $x$ in $H(B)$ is zero.

\begin{definition}
The derived quotient of $A$, by an idempotent $e\in H_0(A)$, denoted by $A/^{\mathbb L}AeA$, is the initial object in the full
subcategory on $e$--killing dg algebras of $\derunderalg{A}$.
\end{definition}

It is clear that the derived quotient is unique up to homotopy, if it exists.

\begin{proposition}
For any idempotent $e\in H_0(A)$ the derived quotient $A/^{\mathbb L}AeA$ exists and is quasi-isomorphic as an $A$--algebra to $L_{1-e}A$.
\end{proposition}

\begin{proof}
One only has to observe that $e$--killing $A$--algebras are precisely $(1-e)$--inverting $A$--algebras.
\end{proof}

Our notation for the derived quotient is justified by the analogous construction in the non-derived context: if $e$ is an idempotent $0$--cycle in $A$, then the initial \emph{strictly} $e$--killing $A$--algebra (equivalently, the non-derived localisation
$A[(1-e)^{-1}]$) is $A/AeA$. The following result might look slightly surprising although its proof is very simple.

\begin{lemma}\label{lem:strict_idempotent}
Let $e$ be an idempotent in $H_0(A)$. Then there is a dg algebra $A^\prime$ quasi-isomorphic to $A$ such that $e$ has a representative in $A^\prime$ that is a strict idempotent (i.e.~not just up to homotopy).
\end{lemma}

\begin{proof}
Consider the algebra $\ground\times \ground$, the direct product of two copies of $\ground$ and let $C$ be its cofibrant replacement $C:=\ground\langle x,y\rangle$ with $d(y)=x^2-x$. It is then clear that the inclusion $\ground\times \ground\to H(A)$ corresponding to the idempotents $e,1-e\in H(A)$ lifts to a map of dg algebras $C\to A$, and without loss of generality we may assume that $A$ is a cofibrant $C$--algebra. Then setting $A^\prime:=A*_C(\ground\times\ground)$, we are done, since $\ground\times\ground$ is left proper, cf.~\autoref{thm:flatleftproper}.
\end{proof}

Let us investigate what the category of $(1-e)$--torsion $A$--modules looks like. If $e$ is a strict idempotent $0$--cycle in $A$ we can form the dg algebra $eAe\cong \Hom_A(Ae,Ae)$ whose elements are products $eae$ where $a\in A$. If $f$ is an idempotent homologous to $e$, then $fAf$ and $eAe$ are quasi-isomorphic dg algebras, since right multiplication by $e$ and by $f$ give maps $Af\to Ae$ and $Ae\to Af$ inverse to each other in $\Ho(\amod)$. Now if $e$ is an idempotent in $H_0(A)$ we will be abusing notation and still write $eAe$ for $eA^\prime e$, where $A'$ is as in \autoref{lem:strict_idempotent}. Thus, the dg algebra $eAe$ is well-defined up to quasi-isomorphism.

We can now give a characterisation of torsion $(1-e)$--modules.

\begin{proposition}
Let $e\in H_0(A)$ be an idempotent. Then the full subcategory of the derived category $\Ho(\modcat{A})$ of $A$ consisting of the $(1-e)$--torsion modules is equivalent to the derived category $\Ho(\modcat{eAe})$ of $eAe$.
\end{proposition}

\begin{proof}
An $A$--module is $(1-e)$--torsion if and only if it is $A/(1-e)$--torsion, cf. the proof of \autoref{prop:torsion}. We will assume, without loss of generality, that $1-e$ is a strict idempotent. It is then clear that the left $A$--module $A/(1-e)$ (which is, by definition, the homotopy cofibre of the right multiplication map by $1-e$ on $A$) is quasi-isomorphic to $Ae\oplus\Sigma Ae$. Thus, the category of $A/(1-e)$--torsion modules coincides with the category of $Ae$--torsion modules. Clearly $Ae$ is a compact left $A$--module. By \cite[Theorem 2.1]{Dwyer_Greenlees} the homotopy categories of $Ae$--torsion modules and of $\REnd_A(Ae, Ae)$--modules are equivalent. Since $Ae$ is a cofibrant left $A$--module, $\REnd_A(Ae, Ae)\simeq \End_A(Ae, Ae)\cong eAe$.
\end{proof}

\begin{remark}
Let $A$ be a dg algebra and $e\in H_0(A)$ an idempotent. With $s=1-e$, the recollement of \autoref{recollement} is
\begin{equation*}
\Ho(\modcat{A/^\mathbb{L}AeA})\quad \substack{\longleftarrow\\
	\longrightarrow \\
	\longleftarrow}\quad
\Ho(\amod)\quad \substack{\longleftarrow\\
	\longrightarrow \\
	\longleftarrow}\quad
\Ho(\modcat{eAe}).
\end{equation*}

This is a generalisation of a recollement established by Cline, Parshall and Scott which plays an important role in their theory of quasi-hereditary algebras. Suppose that $A$ has vanishing differential. Then $A/^{\mathbb{L}}AeA\simeq A/AeA$ if and only if $A\to A/AeA$ is a homological epimorphism, c.f.~\autoref{lem:selftensor} and \autoref{rem:homological_epimorphism}. In \cite{CPS1,CPS2} it is shown that the latter is indeed equivalent to the existence of the above recollement with $A/AeA$ in place of $A/^{\mathbb{L}}AeA$, and also equivalent to $AeA$ being a \emph{stratifying ideal} of $A$, i.e. to $Ae\otimes_{eAe}^{\mathbb{L}}eA \simeq AeA$.
\end{remark}

\subsection{Drinfeld's quotient}
A derived quotient by an idempotent ideal, as a special case of derived localisation, may be computed as a certain derived coproduct, c.f.~\autoref{cor:alglocalisation}. We will now describe another, more economical way to form a derived quotient. We will call it the Drinfeld quotient since it is a dg algebra version of Drinfeld's categorical construction \cite{Drinfeld_dg}.

We will assume that the idempotent $e\in H_0(A)$ admits a representative in $A$ that is strict; by \autoref{lem:strict_idempotent} this results in no loss of generality. Consider the algebra $A_e:=A\langle h\rangle/(he=eh=h)$, obtained by freely attaching the generator $h$ in homological degree $1$ to $A$ and quotienting out by the relations specified. The differential $d$ in $A_e$ extends that on $A$ and $d(h)=e$.

\begin{proposition}
Assume that $A$ is left proper. The algebra $A_e$ is quasi-isomorphic to the derived quotient $A/^{\mathbb L}AeA$.
\end{proposition}

\begin{proof}
Consider the algebra $\ground\times \ground$ spanned by two idempotents $e$ and $1-e$; then $A$ is a $\ground\times \ground$--algebra (since we assumed that $e$ is strict in $A$). Then the dg algebra $(\ground\times\ground)_e$ is easily seen to be quasi-isomorphic to $\ground\cong (\ground\times\ground)[(1-e)^{-1}]\simeq L_{1-e}(\ground\times\ground)$ as $\ground\times \ground$--algebras; here the derived localisation is quasi-isomorphic to the non-derived one, since $\ground\times\ground$ is commutative, c.f.~\autoref{cor:central}. Furthermore the inclusion $\ground\times\ground\hookrightarrow(\ground\times\ground)_e$ is a cofibration (although not a cell inclusion), in contrast to the inclusion $\ground\langle s \rangle \hookrightarrow \ground \langle s,s^{-1}\rangle$, which is not. Then we have the following string of quasi-isomorphisms of dg algebras
\begin{align*}
A/^{\mathbb L}AeA&\simeq L_{1-e}A\\
&\simeq A*^{\mathbb L}_{\ground\times \ground}L_{1-e}(\ground\times \ground)\\
&\simeq A*^{\mathbb L}_{\ground\times \ground} \ground\\
&\simeq  A*_{\ground\times \ground} (\ground\times\ground)_e\\
&\cong A_e
\end{align*}
as required. We are making use of \autoref{lem:pushoutlemma} for the second quasi-isomorphism and \autoref{cor:balancing} for the fourth.
\end{proof}

\begin{remark}
The construction of $A_e$ is indeed a version of Drinfeld's quotient. Note that a set $e_1,\ldots,e_n$ of strict orthogonal idempotents in a dg algebra $A$  with $\sum_{i=1}^ne_n=1$ determines a dg category with $n$ objects (let us call them $e_1,\ldots,e_n$) so that the space of morphisms from $e_i$ to $e_j$ is $e_iAe_j$. Conversely, any dg category with $n$ objects gives rise to a dg algebra, called its category algebra, whose elements are morphisms and the product corresponds to their composition. Then for a single idempotent $e$, the dg algebra $A_e$ is the category algebra of the Drinfeld quotient of the corresponding category with two objects $e$ and $1-e$ by the object $e$.
\end{remark}

\section{The group completion theorem}
Let $M\in\SMon$ be a simplicial monoid. Taking the monoid algebra degree-wise, we obtain a simplicial algebra $\ground[M]\in\SAlg$. Further, taking the Dold--Kan normalisation functor $N$, we get a chain complex $C(M):=N\ground[M]$. Note that the $X\mapsto C(X)$ from the category of simplicial sets to chain complexes is a lax monoidal functor. In particular, there is a natural transformation (Eilenberg--Zilber map); $C(X)\otimes C(Y)\to C(X\times Y)$ satisfying an appropriate associativity condition; in addition $C$ is a \emph{strictly unital} functor meaning that the simplicial monoidal unit (the one-point simplicial set) gets sent to the monoidal unit in $\ground$--modules, i.e.~$\ground$. It follows that $C(M)$ is a monoid in $\ground$--modules, i.e.~a dg algebra, so the functor $X\mapsto C(X)$ is a functor from $\SMon\to \dgalg$. We can regard $\pi_0(M)$ as a subset of $H_0(M)\subset H(M)=H(C(M))$.

\subsection{Simplicial localisation}
Recall that the category of simplicial monoids is a closed model category and the \emph{Dwyer--Kan simplicial localisation} of $M$ at $S\subset \pi_0(M)$, which we denote $L_S^{\mathrm{DK}}(M)$, is constructed in \cite{dk:simploc} by first choosing a zero simplex of $M$ in each connected component in $S$ to obtain a map $\tilde{S}\to M$ from the free monoid $\tilde{S}$ generated by $S$ to $M$. Then replace this by a weakly equivalent cofibration $\tilde{S}\cof \tilde M$ and form the pushout along the map $\tilde{S}\to \tilde{S}[\tilde{S}^{-1}]$, where $\tilde{S}[\tilde{S}^{-1}]$ is the monoid obtained by freely adjoining inverses to all the zero simplices of $\tilde{S}$. Note that, since the category of simplicial monoids is left proper, this is a homotopy pushout.

\begin{theorem}\label{thm:simploc}
Let $M$ be a simplicial monoid and let $S\subset \pi_0(M)$ be an arbitrary subset. The dg algebra $L_{S}(C(M))$ is quasi-isomorphic to $C(L_S^{\mathrm{DK}}(M))$.
\end{theorem}

\begin{proof}
The functor $X\mapsto C(X)$ factors as
\[
X\mapsto \ground[X] \mapsto N\ground [X]\mapsto i\circ N \ground [X] \simeq C(X),
\]
where $i$ is the inclusion of non-negatively graded $\ground$--modules into $\mathbb{Z}$--graded $\ground$--modules.

The functor $X\mapsto \ground [X]$ is a strongly monoidal left Quillen functor (strongly monoidal meaning that the natural transformation $\ground[X]\otimes\ground[Y]\to\ground[X\times Y]$ is an isomorphism for all simplicial sets $X$ and $Y$). Therefore it induces a left Quillen functor $\SMon\to\SAlg$ from simplicial monoids to monoids in simplicial $\ground$--modules, i.e.~simplicial $\ground$--algebras (recall that the category of simplicial algebras is a closed model category with weak equivalences being quasi-isomorphisms of the associated chain complexes and fibrations being level-wise surjective maps). Moreover, the functor $N$ is part of a Quillen equivalence between $\SAlg$ and $\dgalg_{\geq 0}$, see \cite{schwedeshipley:equivs}. Finally the functor $i$ is a left Quillen functor $\dgalg_{\geq 0}\to \dgalg$. Since Quillen equivalences and left Quillen functors preserve homotopy pushouts up to weak equivalence, the functor $X \mapsto C(X)$ from $\SMon$ to $\dgalg$ preserves homotopy pushouts.

Since $C(\tilde{S}) \simeq \ground\langle S \rangle$ and $C(\tilde{S}[\tilde{S}^{-1}])\simeq \ground \langle S, S^{-1} \rangle$, then $C(L_S^{\mathrm{DK}}(M))\simeq C(M)\ast_{\ground \langle S \rangle}^{\mathbb{L}} \ground \langle S, S^{-1} \rangle\simeq L_S(C(M))$.
\end{proof}

\subsection{Group completion}
We can use the theory of derived localisation to give a very natural proof of the group completion theorem. Many variations and proofs of this theorem have been given (see, for example, \cite{mcduffsegal:gct,jardine:gct,moerdijk:gct,tillmann:stablemapping,ps:gct,pm:gct}). However, in addition to a new proof, our approach also offers a conceptual explanation of the group completion theorem by showing it is really just a topological restatement of our results concerning the vanishing of higher derived terms of the derived localisation.

Recall the following fact, which says that simplicial localisation at all of $\pi_0(M)$ is the \emph{group completion} of $M$.

\begin{proposition}[{\cite[5.5]{dk:simploc}}]
Let $M$ be a simplicial monoid. The simplicial localisation $L_{\pi_0(M)}^{\mathrm{DK}}(M)$ is homotopy equivalent to $\Omega B M$, the based loop space of the classifying space of $M$.
\end{proposition}

Combining this fact with \autoref{thm:simploc} above we obtain the following.

\begin{theorem}\label{thm:groupcompl}
The dg algebra $L_{\pi_0(M)}C(M)$ is quasi-isomorphic to $C(\Omega \B M)$.\qed
\end{theorem}

Of course, a version of \autoref{thm:groupcompl} holds for topological monoids; below $\Omega$ and $\B$ stand for the topological versions of the loop space and the classifying space respectively and $C(X)$ stands for the singular chain complex on a topological space $X$.

\begin{corollary}\label{cor:groupcompl}
Let $M$ be a monoid in topological spaces. Then the dg algebra $L_{\pi_0 (M)}C(M)$ is quasi-isomorphic to $C(\Omega \B M)$.
\end{corollary}

\begin{proof}
The functor $\operatorname{Sing}$ associating to a topological space the simplicial set of its singular simplices is right adjoint (to the geometric realisation functor) and so, it preserves products. It follows that it transforms topological monoids into simplicial monoids, thus $C(M)$ is indeed a dg algebra. Additionally, the functor $\operatorname{Sing}$ transforms topological loop spaces and classifying spaces to their simplicial counterparts (up to homotopy). The desired statement now follows from \autoref{thm:groupcompl}.
\end{proof}

\autoref{thm:groupcompl} (or \autoref{cor:groupcompl}) could be viewed as a very general form of the group completion theorem. It is most useful when $L_{\pi_0(M)}C(M)$ has no higher derived terms, i.e.~when $H(L_{\pi_0(M)}C(M))\cong H(M)[\pi_0(M)^{-1}]$. By \autoref{thm:stableflat} this happens if and only $H(M)[\pi_0(M)^{-1}]$ is stably flat over $H(M)$. Specialising further, and taking into account \autoref{thm:ore} we have the following corollary.

\begin{corollary}
If $\pi_0(M)\subset H(M)$ is an Ore set, then
\[
H(\Omega \B M)\cong H(M)[\pi_0(M)^{-1}].
\]
\end{corollary}

In the special case that $\pi_0(M)$ is central in $H(M)$, we recover the McDuff--Segal group completion theorem \cite{mcduffsegal:gct}.

An alternative formulation of the group completion theorem, which is sometimes more suited to certain applications, is to consider the action of $\pi_0(M)$ on the homology of a certain homotopy colimit $\overline{M}_\infty$. More precisely, take any sequence $x_1,x_2,\dots$ in $\pi_0(M)$ and set
\[
\overline{M}_\infty := \hocolim M \xrightarrow{\cdot x_1} M \xrightarrow{\cdot x_2} \dots
\]
The group completion theorem then becomes the statement that if the (left) action of $\pi_0(M)$ on $H(\overline{M}_\infty)$ is by isomorphisms then $H(\Omega \B M)\cong H(\overline{M}_\infty)$. This formulation follows almost immediately from \autoref{thm:simploc} and \autoref{thm:localisationscoincide} by noting that the construction of $\overline{M}_\infty$ is such that $C(\overline{M}_\infty)$ obviously satisfies the universal property of the localisation of $C(M)$ as a $C(M)$--module, excepting the fact that it may not be $\pi_0(M)$--local, which is precisely the hypothesis, cf.~\autoref{thm:ore2}. Note, however, this hypothesis, namely that for such a sequence $x_1,x_2,\dots$ the module $C(\overline{M}_\infty)$ is $\pi_0(M)$--local, implies that $\pi_0(M)$ is an Ore set in $H(M)$, by \autoref{thm:ore2} (under the assumption $\pi_0(M)$ is countable then, by \autoref{thm:ore}, these hypotheses are in fact equivalent), in which case $H(\overline{M}_\infty)\cong H(M)[\pi_0(M)^{-1}]$. Therefore, this version of the group completion theorem is essentially equivalent to the version above.

\section{Localisation of dg bialgebras}
It is frequently the case that we wish to localise a dg algebra which has additional structure, namely that of a dg bialgebra, for example when considering the dg algebra of chains on a simplicial monoid. Often this dg bialgebra structure will lead to the derived localisation being more easily understood. In this section we will explain this and we will give examples where one can identify completely the localisations of some important dg algebras.

As a motivating example, consider the case of $A:=C(M)$, the dg algebra of chains on a simplicial monoid $M$. The functor $X \mapsto C(X)$ is also \emph{colax comonoidal}. In particular, there is a natural transformation (the Alexander--Whitney map) $C(X\times Y)\to C(X)\otimes C(Y)$ satisfying an appropriate associativity condition. It follows that for any simplicial set $X$ its simplicial chain complex $C(X)$ is a comonoid in $\ground$--modules, i.e.~a dg coalgebra. Therefore, $A$ is a dg bialgebra. Note that there is a 1-1 correspondence between $\pi_0(M)$ and the set of grouplike elements in $A$. We will model the algebraic part of this situation more abstractly. For simplicity, we will restrict to modelling the case where $\pi_0(M)\cong \mathbb{N}$, but more general results can easily be formulated.

Let $A$ be an $\mathbb{N}$--graded monoid in the monoidal category of connected cocommutative dg coalgebras (connected here means having a unique grouplike element). In other words $A$ is a dg bialgebra with $A=\bigoplus_{i\geq 0} A_i$ where each $A_i$ is a connected cocommutative dg coalgebra and the algebra structure is determined by coalgebra maps $A_i\otimes A_j \to A_{i+j}$. Assume further that $A_0=\ground$. It follows that the unique grouplike element $s\in A_1$ freely generates the monoid of grouplike elements in $A$.

Denote by $A_\infty$ the coalgebra obtained as the direct limit of the directed system
\[
A_0\xrightarrow{\cdot s} A_1\xrightarrow{\cdot s} A_2 \xrightarrow{\cdot s}\dots
\]
of coalgebras where $\cdot s$ is right multiplication by the grouplike element $s\in A_1$. The underlying space of $A_\infty$ is the direct limit of the underlying $\ground$--modules and the coalgebra structure is induced by the coalgebra structures on the $A_i$. Then $A_\infty[s,s^{-1}] := A_\infty\otimes \ground [s,s^{-1}]$ is a left $A$--module in coalgebras. The coalgebra structure is the tensor product of the coalgebra structure on $A_\infty$ with that on $\ground [s,s^{-1}]$ defined by declaring $s,s^{-1}$ to be grouplike elements. Then $A$ acts on the left on $A_\infty$ and the corresponding action of $A$ on $A_\infty[s,s^{-1}]$ is defined by declaring $A_i$ to act on $A_\infty[s,s^{-1}]$ by the action on $A_\infty$ multiplied by $s^i$.

\begin{proposition}
Let $A$ be as above. If left multiplication by $s$ on $A_\infty$ is a quasi-isomorphism then as $A$--modules, $L_s(A)\simeq A_\infty[s,s^{-1}]$.
\end{proposition}

\begin{proof}
It is not difficult to see that $A_\infty[s,s^{-1}]$ is isomorphic to the direct limit of the direct system $A\xrightarrow{\cdot s}A\xrightarrow{\cdot s}\dots$. Indeed, this limit splits as a direct sum of $\mathbb{Z}$ copies of directed systems each with limit $A_\infty$ so that it is isomorphic to the direct sum of $\mathbb{Z}$ copies of $A_\infty$ and inspection shows that this is isomorphic to $A_\infty[s,s^{-1}]$ as an $A$--module. But now the statement follows from \autoref{oneelementore}.
\end{proof}

\begin{remark}
The proposition above can be regarded as an algebraic form of (a version of) the group completion theorem. Indeed, set $A=C(M)$ for a simplicial monoid $M=\coprod_{i\in \mathbb{N}} M_i$. Then $L_s(A)\simeq C(\Omega \B M)$ and so the proposition says that if left multiplication by the generator $s\in \pi_0(M_1)$ is a quasi-isomorphism on $M_\infty:=\lim_{\rightarrow}M_i$ then $H(\Omega \B M)\cong H(\mathbb{Z}\times M_\infty)$.
\end{remark}

From now on, $\ground$ will always be assumed to be a field of characteristic zero. If the homology class $[s]$ of $s$ is central in $H(A)$, then $H(A_\infty)$ has an algebra structure induced by the algebra structure on $A$ (the centrality of $[s]$ ensures that this structure is well-defined); moreover it is compatible with the coalgebra structure so that $H(A)$ is a cocommutative bialgebra. We will also assume, in fact, that $H(A)$ is (graded) commutative; it is, thus, a commutative bialgebra with a unique (invertible) grouplike element. Therefore it is a Hopf algebra (see, for example, \cite[Corollary 7.6.11]{radford}).

\begin{proposition}
Let $A$ be as above, with $H(A)$ commutative. Then $H(L_s(A))$ is a Hopf algebra and as Hopf algebras $H(L_s(A))\cong H(A_\infty)[s,s^{-1}]$.
\end{proposition}

\begin{proof}
Since $H(A)$ is commutative then $H(L_s(A))$ is simply the non-derived localisation of $H(A)$, hence it is also commutative and, since $s$ is invertible, it is a Hopf algebra.
\end{proof}

In this situation, $H(A_\infty)$ is the universal enveloping algebra of its primitive elements, in which case it is just a polynomial bialgebra $S(V)$ generated by the space of primitives $V$. In many examples we will be able to identify the space of primitives explicitly.

\begin{remark}
If, moreover, $A$ is commutative then there is a quasi-isomorphism of dg algebras $H(L_s(A))\to L_s(A)$ by choosing representing cycles for the primitive elements. Thus if $A$ is quasi-isomorphic to a dg commutative algebra then $L_s(A)$ is formal, i.e.~it is in the same quasi-isomorphism class as its homology.
\end{remark}

\subsection{Cyclic homology}
We show how one can interpret the cyclic homology of a dg algebra as the derived localisation of a suitable noncommutative dg algebra.

Given a dg algebra $A$ define
\[
A_{\gl}=\bigoplus_{n=0}^\infty\CE(\gl_n(A))
\]
where $\CE(\gl_n(A))$ stands for the Chevalley--Eilenberg complex of $\gl_n(A)$, the dg Lie algebra of $n\times n$ matrices with entries in $A$. Recall that for a dg Lie algebra $\mathfrak g$ its Chevalley--Eilenberg complex has underlying space defined as $\CE(\mathfrak{g})=S(\Sigma\mathfrak{g})$ where $S$ stands for the cofree connected cocommutative coalgebra on a graded vector space and the differential is induced by the Lie bracket on $\mathfrak g$, it is therefore a connected cocommutative dg coalgebra.

The matrix block addition
\begin{equation}\label{blockadd}
\gl_n(A)\oplus \gl_m(A)\to\gl_{n+m}(A)
\end{equation}
induces a map of dg vector spaces
\[
\CE(\gl_n(A))\otimes \CE(\gl_m(A))\to \CE(\gl_{n+m}(A))
\]
and thus, the structure of a dg bialgebra on $A_{\gl}$.

Since the map \ref{blockadd} is commutative up to conjugation action of $\GL_{m+n}$ and since the latter action is trivial on the Chevalley--Eilenberg cohomology, we conclude that $H(A_{\gl})$ is a graded commutative and cocommutative bialgebra. Moreover, the element $s=1\in \ground \subset \HCE(\gl_1(A))\subset H(A_\gl)$ generates the monoid of grouplike elements. There is a nested system of dg Lie algebras $\gl_0\subset \gl_1\subset\dots$ and we denote the direct limit of this system by $\mathfrak {gl}_\infty$.

Therefore we have following result.

\begin{proposition}\label{thm:lodayquillen}
There is a quasi-isomorphism of $A_\gl$--modules
\[
L_s(A_\gl)\simeq \CE(\gl_\infty(A))[s,s^{-1}].
\]\qed
\end{proposition}

Moreover, we may identify the space of primitives of the Hopf algebra $H(L_s(A_{\gl}))$ explicitly.
\begin{theorem}
There is an isomorphism of Hopf algebras
\[
H(L_s(A_\gl))\cong S(\Sigma^{-1}\HCC(A))[s,s^{-1}]
\]
where $\HCC(A)$ is the cyclic homology of $A$.
\end{theorem}

\begin{proof}
According to Loday--Quillen and Tsygan \cite[Chapter 10]{loday1992:cyclichombook} there is an isomorphism
\[
\HCE(\gl_\infty(A))[s,s^{-1}]\cong S(\Sigma^{-1}\HCC(A))[s,s^{-1}]
\]
where the space of primitive elements is precisely $\Sigma^{-1}\HCC(A)$.
\end{proof}

\begin{remark}\label{lqtformalremark}
Since the conjugation action of $\GL_n$ is trivial on homology, $A_\gl$ is actually quasi-isomorphic to the subalgebra of $\GL_\infty$--invariants which is a dg commutative algebra. Thus $L_s(A_\gl)$ is formal.
\end{remark}

\subsection{Graph homology}
Our next example concerns Kontsevich's graph homology. There are many versions of Kontsevich's graph complex as it is, essentially, the vacuum part of the Feynman transform of a modular operad \cite{getzlerkapranov1998:modularoperads}. Our result can be formulated in this generality but we shall refrain from doing that and instead, work it out in the simplest case of \emph{commutative graph homology}. This is, in some sense, a nonlinear analogue of cyclic homology.

Let $\mathfrak l_n$ be the space of polynomial symplectic vector fields in the linear symplectic vector space $\ground^{2n}$ vanishing at zero; it is a Lie algebra with respect to the commutator bracket. There is a nested system of Lie algebras ${\mathfrak l}_0\subset{\mathfrak l}_1\subset\dots$ and we denote the direct limit of this system by ${\mathfrak l}_\infty$. Note that analogously to the matrix case there is a map ${\mathfrak l}_n\oplus {\mathfrak l}_m\to {\mathfrak l}_{n+m}$ which is commutative up to the action of the linear symplectic group. Let us now define:
\[
A_{\mathfrak l}=\bigoplus_{n=0}^\infty\CE({\mathfrak l}_n)
\]
Just as $A_{\gl}$, the space $A_{\mathfrak l}$ has the structure of an graded monoid in connected cocommutative dg coalgebras and, since the action of the linear symplectic group is trivial on homology, $H(A_\mathfrak{l})$ is a graded commutative and cocommutative bialgebra. The element $s=1\in\ground\in \CE(\mathfrak{l}_1)\subset (A_{\mathfrak{l}})$ generates the monoid of grouplike elements. Then we have the following result.

\begin{proposition}
There is a quasi-isomorphism of $A_\mathfrak{l}$--modules
\[
L_s(A_\mathfrak{l})\simeq \CE(\mathfrak{l}_\infty)[s,s^{-1}].
\]\qed
\end{proposition}

As above, we may identify the space of primitives of the graded Hopf algebra $H(L_s(A_\mathfrak{l}))$ explicitly; the only difference is that the Loday--Quillen--Tsygan result should be replaced by the corresponding result of Kontsevich \cite{kontsevich1993:ncsg,kontsevich1994:feynman}.

\begin{theorem}
There is an isomorphism of Hopf algebras
\[
H(L_s(A_\mathfrak{g}))\cong S(\Sigma^{-1}\mathrm{H}^\Gamma)[s,s^{-1}]
\]
where $\mathrm{H}^\Gamma$ is the (commutative) graph homology.\qed
\end{theorem}

\begin{remark}
Similar to \autoref{lqtformalremark}, since the conjugation action of the linear symplectic group is trivial on homology, $A_\mathfrak{l}$ is quasi-isomorphic to the subalgebra of invariants which is a dg commutative algebra. Thus $L_s(A_\mathfrak{l})$ is formal.
\end{remark}

\subsection{Algebraic $K$--theory}
Let $A$ be an associative algebra and form the topological monoid $M=\coprod_{n=0}^\infty\B\GL_n(A)$. The monoid structure is induced by the block addition of matrices $GL_n(A)\times GL_m(A)\to GL_{n+m}(A)$. Since the latter is commutative up to conjugation, this monoid is homotopy commutative.

Consider the cocommutative dg bialgebra $\mathcal{A} =C(M)$. Denote by $s$ the element in $H_0\B\GL_1(A)$ corresponding to the base point in $\B\GL_1(A)$. Note that the direct limit of the sequence of complexes $C(\B\GL_0(A))\xrightarrow{\cdot s}C(\B\GL_1(A))\xrightarrow{\cdot s}\dots$ can be identified with $C(\B\GL_\infty)(A)$. We obtain the following:

\begin{proposition}
There is a quasi-isomorphism of $\mathcal{A}$--modules
\[
L_s(\mathcal A)\simeq C(\B\GL_\infty(A))[s,s^{-1}].
\]\qed
\end{proposition}

The $K$--groups of $A$ are defined as $K_i(A)=\pi_i(\Omega \B M), i=1,2,\dots$. Denote by $K_*^{\mathbb{Q}}(A)=\pi_*(\Omega \B M)\otimes\mathbb{Q}=\oplus_{i=1}^\infty\Sigma^iK_i(A)\otimes\mathbb{Q}$, the direct sum of (appropriately shifted) $K$--groups of $A$ tensored with the rationals. Then we have:

\begin{theorem}\label{ktheory}
There is an isomorphism of Hopf algebras
\[
H(L_s(\mathcal A))\otimes \mathbb Q\cong S(K_*^{\mathbb{Q}}(A))[s,s^{-1}].
\]\qed
\end{theorem}

Similar results can be obtained taking $\mathcal A$ to be $C(\coprod_P\B\Aut(P))$, where the coproduct ranges through isomorphism classes of finitely generated projective $A$--modules. It can also be extended to topological $K$--theory (assuming that $A$ is a Banach $\mathbb C$--algebra). The details are standard and we omit them.

\begin{remark}\label{doubleloop}
It is also easy to see that the dg algebra $L_s(\mathcal{A})$ is rationally formal. Indeed, we have $\Omega\B M\simeq (\Omega\B M)_0\times\mathbb Z$ where $(\Omega\B M)_0$ is the connected component of the identity of the monoid $\Omega\B M$. It is known that the monoid $(\Omega\B M)_0$ is a double loop space (actually an infinite loop space). Any connected loop space rationally splits as a disjoint union of products of odd spheres and loop spaces on odd spheres. It follows that the rational chain dg algebra on $(\Omega\B M)_0$ is formal and thus, so is the rational chain dg algebra on $\Omega\B M$.
\end{remark}

\begin{remark}
\autoref{ktheory} is somewhat circular in the sense that the rational $K$--groups of $A$ are more or less by definition the space of primitive elements of the graded Hopf algebra $H(\Omega \B M)$ which we know from \autoref{thm:groupcompl} is precisely $H(L_s(\mathcal{A}))$.

Since $H(L_s(A_\gl))$ is a graded Hopf algebra, one could likewise \emph{define} the cyclic homology of $A$ as the set of primitive elements in this Hopf algebra, precisely analogously to defining the algebraic $K$--groups of $A$ as the homotopy groups of the group completion of the topological monoid $M$.

A similar remark applies to $A_{\mathfrak l}$, so that the graph homology $\mathrm{H}^\Gamma$ can be identified with the space of primitive elements in the Hopf algebra $H(L_s(A_\mathfrak{l}))$.
\end{remark}

\begin{remark}
The computation of the derived localisation of this last example, being the group completion of a topological monoid, can of course be viewed as an application of the classical group completion theorem, together with \autoref{thm:groupcompl}. However, the first two examples do not arise from topology and illustrate our `algebraic' version of the group completion theorem.
\end{remark}

\section{The stable mapping class group}
In this section we will examine examples of monoids arising from mapping class groups of surfaces and explain how their derived localisations can be computed in terms of the stable mapping class group. While some of these examples can be computed using the classical group completion theorem (although understanding this as computing the derived localisation relies on \autoref{thm:groupcompl}), we show that one can also compute `partial group completions' using our more general theory.

\subsection{The pair of pants gluing}
We begin with an example that follows the same pattern as the examples in the previous section.

Let $\Diff_{g,1}$ denote the topological category whose objects are compact oriented surfaces of genus $g$ with one non-empty boundary component equipped with a collar neighbourhood and whose morphisms are orientation preserving diffeomorphisms which preserve the collar neighbourhoods pointwise. Denote by $\Gamma_{g,1}=\pi_0\Diff_{g,1}$ the corresponding discrete groupoid of connected components. Note that, for a surface $S\in \Gamma_{g,1}$ the group of automorphisms of $S$ as an object in the groupoid $\Gamma_{g,1}$ is the usual mapping class group of $S$, which we denote $\Gamma(S)$. Moreover, the inclusion $\Gamma(S)\hookrightarrow \Gamma_{g,1}$ is an equivalence of categories.

Let $P$ be a fixed surface of genus zero with $3$ ordered, collared boundary components. There are functors
\begin{equation}\label{eq:stable}
\Gamma_{g,1}\times \Gamma_{g',1}\to \Gamma_{g+g',1}
\end{equation}
which are defined on objects as sending $(S,S')$ to $S\#_P S'$, where $S\#_P S'$ denotes the surface obtained by gluing (using the collar neighbourhoods) the boundary of $S$ to the first boundary component of $P$ and the boundary of $S'$ to the second boundary component of $P$. On morphisms these functors are given by extending the diffeomorphisms from $S$ and $S'$ to a diffeomorphism from the surface $S\#_P S'$ by setting it to be the identity on $P$ cf.~\cite{Miller}.

\begin{remark}
The categories $\Diff_{g,1}$ and $\Gamma_{g,1}$ as defined have a proper class of objects. This technical detail can be easily overcome by choosing fixed surfaces for each genus and considering the small subcategories whose objects are obtained from all ways of gluing these surfaces via the operation $\#_P$.
\end{remark}

This map is well-defined and strictly associative. Taking in \ref{eq:stable} $g'=1$ and restricting to a single object with the identity map in $\Gamma_{1,1}$ we obtain an embedding $\Gamma_{g,1}\subset\Gamma_{g+1,1}$; we denote $\Gamma_{\infty}= \colim_{g\to\infty}\Gamma_{g,1}$.

Taking classifying spaces we obtain an associative topological monoid $M'_1=\coprod_{g=0}^{\infty} B\Gamma_{g,1}$. It is not hard to see that $M'_1$ is homotopy commutative. Consider the dg algebra $\mathcal{A}'_1 = C(M'_1)$. Denote by $s$ the element in $H_0(B\Gamma_{1,1})$ corresponding to the basepoint in $B\Gamma_{1,1}$.

\begin{proposition}\label{thm:closedglue}
There is a quasi-isomorphism of $\mathcal{A}'_1$--modules
 \[
L_s(\mathcal{A}'_1)\simeq C(B\Gamma_{\infty})[s,s^{-1}].
\]
\end{proposition}

We know \cite{MadsenWeiss} that $H(B\Gamma_{\infty})\otimes\mathbb Q$ is a Hopf algebra which is isomorphic to $\mathbb Q[\kappa_1,\kappa_2,\dots]$ where $\kappa_i, i=1,2,\dots$ are the Morita--Mumford--Miller classes. We obtain the following result.

\begin{theorem}\label{thm:millergluing}
There is an isomorphism of Hopf algebras
\[
H(L_s(\mathcal{A}'_1))\otimes \mathbb Q\cong \mathbb Q[\kappa_1,\kappa_2,\dots][s,s^{-1}].
\]
\end{theorem}

\begin{remark}
It is easy to see that the dg algebra $L_s(\mathcal{A}'_1)$ is formal. The same justification as in Remark \ref{doubleloop} applies.
\end{remark}

\subsection{Punctures}
We may generalise the example above to allow surfaces to have punctures. Let $\Diff_{g,1}^{(h)}$ denote the topological category whose objects are compact oriented surfaces of genus $g$ with one non-empty boundary component equipped with a collar neighbourhood and $h$ unlabelled punctures and whose morphisms are orientation preserving diffeomorphisms which preserve the collar neighbourhoods pointwise. Denote by $\Gamma_{g,1}^{(h)}=\pi_0\Diff_{g,1}^{(h)}$. Again, note that for a surface $S\in \Gamma_{g,1}$ the group of automorphisms of $S$ as an object in the groupoid $\Gamma_{g,1}^{(h)}$ is the usual mapping class group of a surface with punctures where punctures may be permuted. The gluing along a pair of pants $P$ defined above extends to these groupoids and we denote by $M_1=\pi_{g=0,h=0}^{\infty}B\Gamma_{g,1}^{(h)}$ the corresponding topological monoid and $\mathcal{A}_1=C(M_1)$ the corresponding dg algebra. Once again, it is straightforward to see that $M_1$ is homotopy commutative.

The main difference between $\mathcal{A}'_1$ and $\mathcal{A}_1$ when it comes to computing the localisation is that the monoid of grouplike elements in $\mathcal{A}_1$ is now $\ground [s,t]$ where $s\in H_0(B\Gamma_{1,1}^{(0)})$ and $t\in H_0(B\Gamma_{0,1}^{(1)})$ correspond to the basepoints in $B\Gamma_{1,1}^{(0)}$ and $B\Gamma_{0,1}^{(1)}$, illustrated in \autoref{fig:closedgenerators}.

\begin{figure}[ht!]
  \centering
  \subfigure[Genus one surface corresponding to $s$]{
    \includegraphics[scale=1.2]{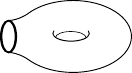}}
  \qquad\qquad
  \subfigure[Genus zero surface with one puncture corresponding to $t$]{
    \hspace{10pt}
    \raisebox{10pt}{\includegraphics[scale=1.2]{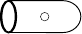}}
    \hspace{10pt}}
  \caption{Generators of the monoid of grouplike elements in $\mathcal{A}_1$}
  \label{fig:closedgenerators}
\end{figure}

Denote by $D$ the topological monoid $\coprod_{h\geq 0} B(S^1\wr \Sigma_h)$ where $S^1\wr \Sigma_h$ is the wreath product of the circle group with $\Sigma_h$, the symmetric group on $h$ letters. Then $C(D)$ is a dg $k[t]$--algebra by sending $t$ to the basepoint of $B(S^1\wr \Sigma_1)\simeq \mathbb{C}P^{\infty}$. We define $S^1\wr\Sigma_{\infty}=\colim_{h\to \infty} S^1\wr \Sigma_h$.

\begin{proposition}\label{prop:punctures}
There are quasi-isomorphisms of $\mathcal A_1$--modules:
\begin{enumerate}
\item $L_s(\mathcal A_1)\simeq C(B\Gamma_{\infty})\otimes C (D)[s,s^{-1}]$
\item
$L_{t,s}(\mathcal A_1)\simeq C(B\Gamma_{\infty})\otimes C (B (S^1\wr \Sigma_\infty))[t,t^{-1},s,s^{-1}]$
\end{enumerate}
\end{proposition}

\begin{proof}
Since the homology of $\mathcal{A}_1$ is commutative then
\[
L_s(\mathcal{A}_1)\simeq \mathcal{A}_1\otimes_{\ground[s]} \ground[s,s^{-1}] \simeq \colim(\mathcal{A}_1\xrightarrow{\cdot s}\mathcal{A}_1\dots).
\]
But $\colim(\mathcal{A}_1\xrightarrow{\cdot s}\mathcal{A}_1\dots)$ decomposes as
\[
\bigoplus_{h=0}^{\infty}\colim\left (C(B\Gamma_{0,1}^{(h)})\xrightarrow{\cdot s}C(B\Gamma_{1,1}^{(h)})\xrightarrow{\cdot s}\dots\right)[s,s^{-1}]
\]
and by B\"odigheimer--Tillmann \cite[Theorem 1.1]{boedigheimertillmann}
\[
\colim\left (C(B\Gamma_{0,1}^{(h)})\xrightarrow{\cdot s}C(B\Gamma_{1,1}^{(h)})\xrightarrow{\cdot s}\dots\right)\simeq C(B\Gamma_\infty)\otimes C(B(S^1\wr \Sigma_h)),
\]
part (1) of the Proposition now follows.

For part (2) observe that $L_{t,s}(\mathcal{A}_1)\simeq L_t(L_s(\mathcal{A}_1))\simeq L_s(\mathcal{A}_1)\otimes_{\ground[t]}\ground[t,t^{-1}]$ and then unwrap the corresponding colimit similarly.
\end{proof}

Note that $H(L_{t,s}(\mathcal{A}_1))$ is a commutative and cocommutative Hopf algebra, so we can easily identify the algebra structure. However, $H(L_s(\mathcal{A}_1))$ is not a Hopf algebra, the grouplike element $t$ not being invertible.

\begin{theorem}\label{thm:millergluingpunctures}\mbox{}
\begin{enumerate}
\item There is an isomorphism of bialgebras
\[H(L_s(\mathcal A_1))\otimes \mathbb Q\cong \mathbb Q[\kappa_1,\kappa_2,\dots]\otimes \mathcal{D}[s,s^{-1}],\]
where $\mathcal{D}\subset \mathbb{Q}[\nu,t]$ is the subalgebra of elements of the form $\lambda + tx$ with $\lambda\in \mathbb{Q}, x \in\mathbb{Q}[\nu,t]$.
\item There is an isomorphism of Hopf algebras
\[H(L_{t,s}(\mathcal A_1))\otimes \mathbb Q\cong\mathbb Q[\nu,\kappa_1,\kappa_2,\dots] [t,t^{-1},s,s^{-1}].\]
\end{enumerate}
Here the element $\nu$ has degree $2$.
\end{theorem}

\begin{proof}
Observe that we are working rationally, $\Sigma_h$ is a finite group and $H(S^1\wr \Sigma_h,\mathbb{Q})\cong H(\Sigma_h,\mathbb{Q}[x_1,\dots,x_h])\cong \mathbb{Q}[\nu]$ for $h>0$, where $\nu,x_i$ have degree $2$. Thus $H(D)\otimes \mathbb{Q}\cong \mathcal{D}\subset \mathbb{Q}[\nu,t])$ is the subalgebra of elements of the form $\lambda + tx$ with $\lambda\in \mathbb{Q},x\in\mathbb{Q}[\nu,t]$ and $H(S^1\wr \Sigma_\infty,\mathbb{Q})\cong \mathbb{Q}[\nu]$.

Since $H(L_{t,s}(\mathcal{A}_1))$ is a commutative and cocommutative Hopf algebra part (2) is now clear. For part (1) observe that homotopy commutativity implies that $H(L_{t,s}(\mathcal{A}_1))\cong L_t(H(L_s(\mathcal{A}_1)))$ and therefore $H(L_s(\mathcal{A}_1))$ is the obvious subalgebra of $H(L_{t,s}(\mathcal{A}_1))$ and so the algebra structure is as described.
\end{proof}

\subsection{Open gluing}
We now discuss a monoid which uses \emph{open gluing} of surfaces and generalises the examples above even further.

Let $\Diff_{g,\ast}^{(h)}$ denote the topological category whose objects are compact oriented surfaces of genus $g$ with non-empty boundary and $2$ ordered parametrised intervals embedded in the boundary, with $h$ boundary components which do not contain embedded intervals. The embedded intervals will be called \emph{open boundaries} and the $h$ boundary components not containing open boundaries will be called \emph{free boundaries}. The morphisms are orientation preserving diffeomorphisms which preserve the open boundaries pointwise. Let $\Gamma_{g,\ast}^{(h)}=\pi_0\Diff_{g,\ast}^{(h)}$ be the corresponding discrete groupoid. Abusing notation slightly, we denote by $\Gamma_{g,1}^{(h)}$ and $\Gamma_{g,2}^{(h)}$ the full subcategories on surfaces with $h+1$ or $h+2$ boundary components respectively.

\begin{remark}
The objects of $\Diff_{g,\ast}^{(h)}$ should, of course, be two dimensional manifolds with corners compatible with the embedded intervals and equipped with appropriate collar neighbourhoods of the intervals in order for gluing to be well-defined, but we skip the details here.
\end{remark}

Given surfaces $S\in\Diff_{g,\ast}^{(h)}$ and $S'\in\Diff_{g,\ast}^{(h)}$ the surface $S\#_I S'$ is defined by gluing the first interval in the boundary of $S'$ to the second interval in the boundary of $S$. Applying the classifying space functor we obtain by this gluing a topological monoid $M_\ast=\coprod_{g=0,h=0}^{\infty} B\Gamma_{g,\ast}^{(h)}$. It is not hard to see that the monoid $M_1$ considered above is equivalent to the submonoid of $M_\ast$ corresponding to surfaces having precisely one non-free boundary component, cf.~\cite[Section~4]{tillmann:stablemapping}.

\begin{lemma}
The commutative monoid $\pi_0(M_\ast)$ has three generators $s,t$ and $u$ subject to the relation $u^2=ut$.
\end{lemma}

\begin{proof}
The generator $s$ corresponds to the surface of genus one with two boundary components one of which is a free boundary, the generator $t$ corresponds to the surface of genus zero with a single boundary component and the generator $u$ corresponds to the surface of genus zero having two boundary components and no free boundary, as shown in \autoref{fig:pi0generators}. The relation $u^2=ut$ is obvious. Next, a surface with two open boundaries is determined uniquely by its genus $g\geq 0$, the number of free boundary components $h\geq 0$, and the number of boundary components containing open boundaries $k\in\{1,2\}$. Thus, any such surface can be uniquely decomposed as a product $s^{g}\cdot t^{h}\cdot u^{k-1}$ which implies the lemma.
\end{proof}

\begin{figure}[ht!]
  \centering
  \subfigure[Genus one surface corresponding to $s$]{
    \includegraphics[scale=1.2]{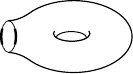}}
  \\
  \subfigure[Genus zero surface with free boundary corresponding to $t$]{
    \includegraphics[scale=1.2]{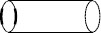}}
  \qquad\qquad
  \subfigure[Genus zero surface corresponding to $u$]{
    \includegraphics[scale=1.2]{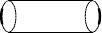}}
  \caption{Generators of $\pi_0(M_\ast)$}
  \label{fig:pi0generators}
\end{figure}

Consider now the dg algebra $\mathcal A_*=C(M_*)$. It is a cocommutative dg bialgebra; the elements $s,t,u\in\pi_0(M_*)\subset H_0(M_*)$ are the grouplike generators. Even though we do not have homotopy commutativity for this monoid we do still have homological stability \cite{harer}, in other words $l_s\co H_k(M_*)\to H_k(M_*)$ and $l_u\co H_k(M_*)\to H_k(M_*)$ are isomorphisms in the stable range, i.e.~on summands with genus $g\gg k$.

\begin{proposition} There are quasi-isomorphisms of $\mathcal A_*$--modules:
\begin{enumerate}
\item $L_s(\mathcal A_*)\simeq C(B\Gamma_{\infty})\otimes C(D)\otimes_{\ground [t]}\ground[u,t, s,s^{-1}]/(u^2-ut)$
\item $L_{u,s}(\mathcal A_*)\simeq C(B\Gamma_{\infty})\otimes C ( B (S^1\wr \Sigma_\infty))[t,t^{-1},s,s^{-1}]$
\end{enumerate}
\end{proposition}

\begin{proof}
Homological stability implies immediately that $\mathcal A_*\otimes_{\ground[s]}\ground [s,s^{-1}]$ is $s$--local. By \autoref{oneelementore} we have a quasi-isomorphism of $\mathcal A_*$--modules
\[
L_s(\mathcal A_*)\simeq \mathcal A_*\otimes_{\ground[s]}\ground[s,s^{-1}].
\]
The proof of part (1) now proceeds as in \autoref{prop:punctures} noting that homological stability implies the various colimits do not depend on the parameter $u$.
Similarly, for part (2) observe that $L_s(\mathcal A_*)\otimes_{\ground[u]}\ground [u,u^{-1}]$ is $u$--local using the fact that $s$ is invertible in homology together with homological stability, then proceed as in \autoref{prop:punctures}.
\end{proof}

We obtain the following result, analogous but more general than \autoref{thm:millergluing} and \autoref{thm:millergluingpunctures}.

\begin{theorem}\label{thm:opengluings}\mbox{}
\begin{enumerate}
\item There is an isomorphism of bialgebras
\[
H(L_s(\mathcal A_*))\otimes \mathbb Q\cong \mathbb Q[\kappa_1,\kappa_2,\dots]\otimes \mathcal{D}_*[s,s^{-1}]
\]
where $\mathcal{D}_*\subset \mathbb{Q}[\nu,t,u]/(u^2-ut)$ is the subalgebra of elements of the form $\lambda + tx$ with $\lambda\in \mathbb{Q}, x \in\mathbb{Q}[\nu,t,u]/(u^2-ut)$.
\item There is an isomorphism of Hopf algebras
\[
H(L_{u,s}(\mathcal A_*))\otimes \mathbb Q\cong\mathbb Q[\nu,\kappa_1,\kappa_2,\dots][t,t^{-1},s,s^{-1}].
\]
\end{enumerate}
Here the generator $\nu$ has degree $2$.
\end{theorem}

\begin{proof}
For part (2), although we do not have homotopy commutativity of $M_*$, we do have that $M_1$ is (equivalent to) a submonoid of $M_*$ so there is a dg algebra map from $L_{t,s}(\mathcal{A}_1)$ to $L_{u,s}(\mathcal{A}_*)$. By inspection this map is the identity as a map of vector spaces, which shows the algebra structures coincide.

Similarly, for part (1), $L_s(\mathcal{A}_1)$ is the obvious subalgebra of $L_s(\mathcal{A}_*)$ so we only need to verify that the multiplicative action of the element $u$ is as described, which is clear.
\end{proof}

\begin{remark}
We have considered here the mapping class groups of surfaces with free boundary and no restriction on what diffeomorphisms do to the free boundary/punctures because this is the case most commonly associated with the ribbon graph complex. However, we could equally have considered versions where we require that the free boundaries/punctures are not allowed to be permuted (sometimes called the pure mapping class group), or where we require that each free boundary is parametrised and diffeomorphisms can only permute free boundaries in a way that is compatible with the given parametrisation (in fact, B\"odigheimer--Tillmann \cite{boedigheimertillmann} reserve the term `free boundary' for this notion). To obtain corresponding results, we replace $S^1\wr \Sigma_k$ by $(S^1)^{\times k}$ in the first case and by $\Sigma_k$ in the second case.
\end{remark}

\subsection{Closed gluing}
We may also consider a monoid which uses usual gluing of surfaces along their boundary. Indeed, let $\Diff_{g,2}$ be the category whose objects are compact oriented surfaces of genus $g$ with two ordered non-empty boundary components equipped with a collar neighbourhood and whose morphisms are orientation preserving diffeomorphisms which preserve the collar neighbourhoods pointwise. Then gluing the first boundary component of one surface to the second boundary component of another induces an associative operation on the groupoids $\Gamma_{g,2}=\pi_0 \Diff_{g,2}$ making $N=\coprod_{g=0}^{\infty} B\Gamma_{g,2}$ into a topological monoid. Then the dg algebra $\mathcal{B}=C(N)$ has a grouplike element $s\in H_0(B\Gamma_{1,2})$ which generates the monoid of grouplike elements.

Again, we do not have homotopy commutativity so the standard argument does not quite apply. Nevertheless, we do have homological stability, which is sufficient to deduce the following.

\begin{proposition}
There is a quasi-isomorphism of $\mathcal{B}$--modules
\[
L_s(\mathcal{B})\simeq C(B\Gamma_\infty)[s,s^{-1}].
\]
\qed
\end{proposition}

\begin{theorem}
There is an isomorphism of Hopf algebras
\[
H(L_s(\mathcal{B}))\otimes\mathbb{Q} \cong \mathbb{Q}[\kappa_1,\kappa_2,\dots][s,s^{-1}].
\]
\end{theorem}

\begin{proof}
Once again, the problem lies in determining that the algebra structure is as stated, since commutativity is no longer so clear.

There is a map of monoids $M_1'\to N$ (more accurately, this map is between topological monoids equivalent to $M_1'$ and $N$) which for a surface of genus $1$ is illustrated in \autoref{fig:tillmannpic}, this observation (and elegant illustration) is from Tillmann \cite{tillmann:stablemapping}. Therefore, by the universal property of localisation, there is a map of dg algebras $L_s(\mathcal{A}_1')\to L_s(\mathcal{B})$. By inspection this is the identity on homology as a map of vector spaces and hence it is an isomorphism of algebras, in fact of Hopf algebras, on homology as required.
\end{proof}

\begin{figure}[ht!]\label{fig:tillmannpic}
  \centering
    \includegraphics[scale=1.2]{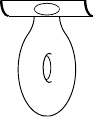}\qquad\raisebox{10mm}{$\longmapsto$}\qquad\includegraphics[scale=1.2]{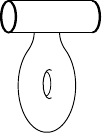}
  \caption{Map from open gluing to closed gluing}
\end{figure}

\begin{remark}
We can also obtain corresponding results for closed gluings of surfaces with punctures; we omit these extra examples since the methods are the same as those already employed above.
\end{remark}

\subsection{Localisation of the dg algebra of ribbon graphs}
Consider the anticyclic operad $\ass$ whose algebras are dg associative algebras with an odd scalar product. Viewing $\ass$ as a \emph{modular} operad with trivial self-gluings we can form its \emph{Feynman transform} $\fass$ \cite{getzlerkapranov1998:modularoperads}. The vacuum part $\fass((0))$ of it is just Kontsevich's ribbon graph complex \cite{kontsevich1994:feynman} and its cohomology is essentially the homology of moduli spaces of Riemann surfaces with unlabelled boundary components. Similarly, the dg space $\fass((n))$ for $n > 0$ is the complex of ribbon graphs with $n$ marked legs; the corresponding homology is the homology of the moduli spaces of Riemann surfaces with $n$ labelled open intervals embedded into the boundary \cite{costello2007:ribbondecomp}. In view of the correspondence between mapping class groups and moduli spaces of Riemann surfaces, the dg space $\fass((2))$ is quasi-isomorphic to $C(M_*)$. Moreover, the operadic composition determines the structure of a dg algebra on $\fass((2))$ which reflects the gluing of Riemann surfaces along open boundaries. This operation corresponds to the multiplication on the monoid $M_*$ and we conclude that $\fass((2))$ and $C(M_*)$ are quasi-isomorphic as dg algebras. Under this quasi-isomorphism the class $s\in H_0(M_*)$ is represented by any trivalent ribbon graph of topological genus one, with two legs and a single boundary component. Similarly, the class $u\in H_0(M_*)$ is represented by a trivalent ribbon graph of topological genus zero, two single boundary components and a leg attached to each boundary component, see \autoref{fig:ribbongraphgenerators}.

\begin{figure}[ht!]
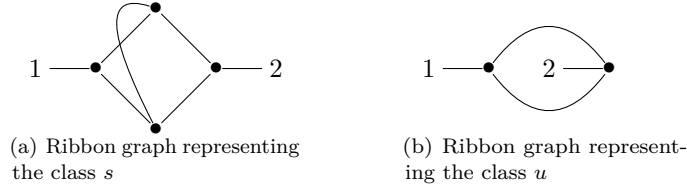

  \centering
  \subfigure[Ribbon graph representing the class $s$]{
{\xygraph{!{<0cm,0cm>;<0.8cm,0cm>:<0cm,0.8cm>::}
&& *{\bullet}="t"\\
{1}="1" & *{\bullet}="l" && *{\bullet}="r" & {2}="2"\\
&& *{\bullet}="b"\\
"1"-"l" "l"-"t" "l"-"b" "t"-"r" "b"-"r" "r"-"2" "b"-@`{"t"+(-1.3,0.5)}"t"}}
}\qquad\qquad
\subfigure[Ribbon graph representing the class $u$]{
\xygraph{!{<0cm,0cm>;<0.8cm,0cm>:<0cm,0.8cm>::}
&& *{\phantom{\bullet}}="t"\\
{1}="1" & *{\bullet}="l" &{2}="2" & *{\bullet}="r" & {\phantom{2}}\\
&& *{\phantom{\bullet}}="b"\\
"1"-"l" "l"-@`{"t"+(0,0.4)}"r" "l"-@`{"b"+(0,-0.4)}"r" "r"-"2"}}
  \caption{Ribbon graphs corresponding to generators of $\pi_0(M_*)$}
  \label{fig:ribbongraphgenerators}
\end{figure}

We obtain the following corollary, from \autoref{thm:opengluings}.

\begin{corollary}\mbox{}
\begin{enumerate}
\item The dg algebra $L_s(\fass((2)))$ has homology \[\mathbb Q[\kappa_1,\kappa_2,\dots]\otimes \mathcal{D}_*[s,s^{-1}],\] where $\mathcal{D}_*\subset \mathbb{Q}[\nu,t,u]/(u^2-ut)$ is the subalgebra of elements of the form $\lambda + tx$ with $\lambda\in \mathbb{Q}, x \in\mathbb{Q}[\nu,t,u]/(u^2-ut) $.
\item The dg algebra $L_{u,s}(\fass((2)))$ has homology \[\mathbb Q[\nu, \kappa_1,\kappa_2,\dots][t,t^{-1},s,s^{-1}].\]
\end{enumerate}
\end{corollary}

\begin{remark}
The dg algebra $\fass((2))$ is a cofibrant dg $\ground[s]$--algebra so the localisation $L_s(\fass((2)))$ can be obtained by just adjoining a single strict inverse to the element $s$. A similar remark holds for the localisation $L_{u,s}(\fass((2)))$.
\end{remark}

\bibliography{references}
\bibliographystyle{alphaurl}

\end{document}